\documentclass[a4paper,11pt]{amsart}
\usepackage{amsmath, amssymb, amsthm}
\usepackage{enumerate}
\usepackage{verbatim}
\usepackage{hyperref}
\usepackage{framed}
\usepackage{color}
\usepackage{footmisc}

\input{xy} 
\xyoption{all} 
\CompileMatrices 

\allowdisplaybreaks[2]

\newcounter{theorem}
\newtheorem{thm}[theorem]{Theorem}
\newtheorem{lemma}[theorem]{Lemma}
\newtheorem{prop}[theorem]{Proposition}
\newtheorem{cor}[theorem]{Corollary}

\newtheorem{conjecture}[theorem]{Conjecture}

\theoremstyle{definition}
\newtheorem{defn}[theorem]{Definition}

\theoremstyle{remark}
\newtheorem*{remark*}{Remark} 
\newtheorem{remark}[theorem]{Remark}
\newtheorem{example}[theorem]{Example}
\newtheorem{question}[theorem]{Question}

\newcommand{\eps}{\varepsilon}
\newcommand{\dl}{\delta}

\numberwithin{equation}{section}
\numberwithin{theorem}{section}

\newcommand{\Z}{\mathcal Z}
\newcommand{\N}{\mathbb N}
\newcommand{\Ad}{\mathrm{Ad}}
\newcommand{\Cu}{\mathcal{C}u}
\newcommand{\dimnuc}{\dim_{\mathrm{nuc}}}
\newcommand{\dr}{\mathrm{dr}}
\newcommand{\R}{\mathcal R}
\newcommand{\vnotimes}{\,\overline{\otimes}\,}
\newcommand{\F}{\mathcal F}
\newcommand{\M}{\mathcal M}
\newcommand{\Q}{\mathcal Q}

\newcommand{\labelledthing}[2]{\hspace{4pt}\buildrel {#2} \over #1 \hspace{3pt}} 

\newcommand{\labelledrightarrow}{\labelledthing{\longrightarrow}}

\title[Dimension and $2$-coloured classification]{Covering dimension of C$^{*}$-algebras and 2-coloured classification}
\author[J.\ Bosa]{Joan Bosa}
\address{\hskip-\parindent Joan Bosa. School of Mathematics and Statistics, University of Glasgow, Glasgow, G12 8QW, Scotland.}
\email{joan.bosa@glasgow.ac.uk}

\author[N.\ Brown]{Nathanial P.\ Brown}
\address{\hskip-\parindent Nathanial P.\ Brown. Department of Mathematics, The Pennsylvavia State University, University Park, State College, PA16802, USA.}
\email{nbrown@math.psu.edu}
\author[Y.\ Sato]{Yasuhiko Sato}
\address{\hskip-\parindent Yasuhiko Sato, Graduate School of Science, Kyoto University, Sakyo-ku, Kyoto 606-8502, Japan.}
\email{ysato@math.kyoto-u.ac.jp}
\author[A.\ Tikuisis]{Aaron Tikuisis}
\address{\hskip-\parindent Aaron Tikuisis, Institute of Mathematics, School of Natural and Computing Sciences, University of Aberdeen, AB24 3UE, Scotland.}
\email{a.tikuisis@abdn.ac.uk}
\author[S.\ White]{Stuart White}
\address{\hskip-\parindent Stuart White, School of Mathematics and Statistics, University of Glasgow, Glasgow, G12 8QW, Scotland and Mathematisches Institut der WWU M\"unster, Einsteinstra\ss{}e 62, 48149 M\"unster, Germany.}
\email{stuart.white@glasgow.ac.uk}
\author[W.\ Winter]{Wilhelm Winter}
\address{\hskip-\parindent
Wilhelm Winter, Mathematisches Institut der WWU M\"unster, Einsteinstra\ss{}e 62, 48149 M\"unster, Germany.}
\email{wwinter@uni-muenster.de}
\thanks{Research partially supported by EPSRC (grant no. I019227/1-2), by NSF (grant no. DMS-1201385), by JSPS (the Grant-in-Aid for Research Activity Start-up 25887031), by NSERC (PDF, held by AT), by an Alexander von Humboldt foundation fellowship (held by SW) and by the DFG (SFB 878).}

\subjclass[2010]{46L05, 46L35}
\begin{document}

\dedicatory{\begin{center}\emph{Dedicated to George Elliott on the occasion of his 70th birthday.}
\end{center}}

\begin{abstract}
We introduce the concept of finitely coloured equivalence for unital $^*$-{}homomorphisms between $\mathrm C^*$-algebras, for which unitary equivalence is the $1$-coloured case.  We use this notion to classify $^*$-homomorphisms from separable, unital, nuclear $\mathrm C^*$-algebras into ultrapowers of simple, unital, nuclear, $\Z$-stable $\mathrm C^*$-algebras with compact extremal trace space up to $2$-coloured equivalence by their behaviour on traces; this is based on a $1$-coloured classification theorem for certain order zero maps, also in terms of tracial data.

As an application we calculate the nuclear dimension of non-AF, simple, separable, unital, nuclear, $\Z$-stable $\mathrm C^*$-algebras with compact extremal trace space: it is 1. In the case that the extremal trace space also has finite topological covering dimension, this confirms the remaining open implication of the Toms-Winter conjecture. Inspired by homotopy-rigidity theorems in geometry and topology, we derive a ``homotopy equivalence implies isomorphism'' result for large classes of $\mathrm{C}^{*}$-algebras with finite nuclear dimension.
\end{abstract}

\maketitle

\tableofcontents

\renewcommand*{\thetheorem}{\Alph{theorem}}
\clearpage\section*{Introduction}

\noindent
Extending covering dimension to the noncommutative context, WW and Zacharias introduced nuclear dimension in \cite{WZ:Adv}. This is a topological concept formed by approximating $\mathrm C^*$-algebras by noncommutative partitions of unity, and then measuring the dimension through a covering number for these approximations. Via the Gelfand transform, commutative $\mathrm C^*$-algebras are of the form $C_0(X)$, and the nuclear dimension recaptures the dimension of the underlying space $X$.
Through examples, we see that the nuclear dimension starkly divides simple nuclear $\mathrm C^*$-algebras into two classes: the topologically infinite dimensional $\mathrm C^*$-algebras, where higher dimensional topological phenomena can occur, and low dimensional algebras. Indeed, all Kirchberg algebras have finite nuclear dimension (\cite{WZ:Adv,MS:DMJ,RSS:arXiv}), while the exotic examples of simple, infinite but not purely infinite $\mathrm{C}^*$-algebras from \cite{R:Acta} cannot have finite nuclear dimension.\footnote{By Kirchberg's dichotomy theorem \cite[Theorem 4.1.10]{R:Book} and the main result of \cite{W:Invent2}.}  We are left to decide when simple, stably finite $\mathrm C^*$-algebras enjoy this property. 

The following conjecture of Toms and WW predicts the answer: finite nuclear dimension should coincide with two other regularity properties of very different natures (cf. \cite{ET:BAMS,W:Invent1}; the precise form stated here is given in \cite[Conjecture 9.3]{WZ:Adv}).

\begin{conjecture}[Toms-Winter]\label{Con:TW} Let $A$ be a simple, separable, unital, infinite dimensional, nuclear $\mathrm C^*$-algebra. The following are equivalent.
\begin{enumerate}
\item $A$ has finite nuclear dimension.
\item $A$ tensorially absorbs the Jiang-Su algebra of \cite{JS:AJM} ($A$ is $\Z$-stable).
\item $A$ has strict comparison. 
\end{enumerate}
\end{conjecture}

Let us give a brief description of properties (ii) and (iii) above; precise details are found in Section \ref{sect:2}.

A $\mathrm C^*$-algebra is $\Z$-stable if $A\otimes\Z\cong A$.  This should be thought of as the $\mathrm C^*$-algebra version of being a McDuff factor (a von Neumann II$_1$ factor that absorbs the hyperfinite II$_1$ factor $\R$ tensorially --- a fundamental concept in the theory of II$_1$ factors).  Indeed, the Jiang-Su algebra $\Z$ is the minimal nontrivial unital $\mathrm C^*$-algebra $\mathcal{D}$ which has two fundamental properties enjoyed by $\R$: $\mathcal{D}\cong \mathcal{D}\otimes \mathcal{D}$ and 
every unital endomorphism of $\mathcal{D}$ is approximately inner
\cite{W:JNCG}. 
As such, $\Z$-stability is the weakest tensorial absorption hypothesis analogous to the McDuff property that one can apply to $\mathrm C^*$-algebras (see \cite{TW:TAMS} for these \emph{strongly self-absorbing algebras} $\mathcal{D}$ and $\mathcal{D}$-absorption).  

Very roughly, $A$ has strict comparison whenever traces determine the order on positive elements. This $\mathrm C^*$-property is inspired by the fact that traces determine the order on projections in any von Neumann algebra of type II$_1$. Unlike the $\mathrm{W}^*$-case, however, having strict comparison isn't automatic (\cite{V:JFA}). With the help of the Cuntz semigroup, strict comparison can also be given an algebraic flavour via a natural order-completeness property (\cite{R:IJM}). 

The implications (i) $\Longrightarrow$ (ii) and (ii) $\Longrightarrow$ (iii) of Conjecture \ref{Con:TW} were proved by WW and by R{\o}rdam, respectively (\cite{W:Invent1,W:Invent2,R:IJM}), while (iii) $ \Longrightarrow $ (ii) is known to hold whenever the tracial state space $T(A)$ is a Bauer simplex (i.e., it is nonempty and its extreme boundary $\partial_e T(A)$ is compact) and $\partial_e T(A)$ has finite covering dimension (\cite{KR:Crelle,S:Preprint2,TWW:IMRN}) --- these extend the work \cite{MS:Acta} which handles the case of finitely many extremal traces. In more recent breakthroughs, (ii) $ \Longrightarrow $ (i) was proven in the monotracial case (\cite{MS:DMJ,SWW:arXiv}).  In the present paper, the following theorem is established.  
\begin{thm} 
\label{thm:TW}
The implication (ii) $ \Longrightarrow $ (i) of the Toms-Winter conjecture holds whenever $T(A)$ is a Bauer simplex. Hence the Toms-Winter conjecture holds if in addition $\partial_e T(A)$ has finite topological dimension.
\end{thm}

It turns out that the passage from the monotracial case to Bauer simplices goes well beyond a straightforward generalization, and in fact requires a whole arsenal of new machinery. We describe the main ingredients of the proof of Theorem \ref{thm:TW} below, but first discuss an application. 

It is well-known that the class $\mathcal{C}$ of simple approximately subhomogeneous (ASH) algebras with no dimension growth provides models in the context of stable finiteness and finite nuclear dimension, meaning that if a stably finite algebra $A$ is as in the Toms-Winter conjecture then it ought to be isomorphic to an element in $\mathcal{C}$, because this class exhausts the Elliott invariant. Inspired by classical theorems in geometry (e.g., Mostow's rigidity theorem), let us say that a simple $\mathrm C^*$-algebra $A$ is \emph{homotopy rigid} if it is isomorphic to an element in $\mathcal{C}$ whenever it is homotopic to an element in $\mathcal{C}$.\footnote{Since homotopy doesn't preserve Elliott invariants, the two elements in $\mathcal{C}$ will be different, in general.} In \cite{MS:DMJ}, Matui and YS proved that having a unique trace implies homotopy rigidity; in their paper quasidiagonality is a crucial assumption which here follows from work of Voiculescu (\cite{V:Duke}). We note that the third condition in our next result simultaneously generalizes the first two.
 
\begin{thm} 
\label{thm:Rigidity}
Let $A$ be a simple, separable, unital, infinite dimensional $\mathrm{C}^{*}$-algebra with finite nuclear dimension. 
The following hypotheses imply homotopy rigidity. 
\begin{enumerate} 
\item $A$ is monotracial (the Matui-Sato theorem). 
\item $A$ has real rank zero and $T(A)$ is a Bauer simplex.
\item Projections separate traces on $A$ and $T(A)$ is a Bauer simplex.
\end{enumerate} 
\end{thm} 

In an earlier version of this paper, we included the hypothesis that all traces are quasidiagonal in the above theorem; we are grateful to Narutaka Ozawa for pointing out that this hypothesis is superfluous.

\bigskip
To prove the results above we need a number of tools. The first such device, and arguably the main technical result of this paper, classifies a certain kind of maps between $\mathrm{C}^{*}$-algebras up to approximate unitary equivalence. Statements like this, with various conditions on the domain and target algebras, and with various classifying invariants, are by now omnipresent in the structure and classification theory of $\mathrm{C}^{*}$-algebras; they are often referred to as \emph{uniqueness results} for $^{*}$-homomorphisms. Approximate unitary equivalence is the equivalence relation of choice in this context, as it provides access to Elliott's intertwining argument (cf.\ \cite{Ell:intertwining-in-tuxedo}); the invariants will then inevitably keep track of tracial information and of ideal structures. When comparing $^{*}$-homomorphisms from this point of view, the invariants will also have to include all sorts of homological --- i.e., $K$-theoretic --- information. We will eliminate this complication by considering \emph{cones over $^{*}$-homomorphisms} or, more generally, certain types of order zero (i.e., orthogonality preserving) completely positive contractions. 

Very roughly, our result then says that two cones over $^{*}$-homomorphisms agree up to approximate unitary equivalence, provided the $^*$-homomorphisms carry the same tracial information. 
To be more precise, given a $\Z$-stable $\mathrm{C}^*$-algebra $B$, and a $^*$-homomorphism $\hat{\phi}:A\rightarrow B$, let $k$ be a positive contraction in $\Z$ with spectrum $[0,1]$. This gives a $^*$-homomorphism $C_0((0,1])\rightarrow\Z$ which maps the generator $\mathrm{id}_{(0,1]}\in C_0((0,1])$ to $k$.  Then define $\phi:C_0((0,1])\otimes A\rightarrow \Z\otimes B \cong B$ to be the tensor product of this $^*$-homomorphism with $\hat\phi$.\footnote{The notation here is chosen as $1_\Z\otimes\hat{\phi}$ is a \emph{supporting map} for $\phi$; a concept we use repeatedly in the paper --- see Lemma \ref{lem:SupportingMap}.}

\begin{thm}\label{onecolourclass}
Let $A$ be a separable, unital and nuclear $\mathrm{C}^*$-algebra, and let $B$ be a simple, separable, unital and exact $\mathrm{C}^*$-algebra such that $B$ is $\Z$-stable and $T(B)$ is a Bauer simplex.  Consider two injective $^*$-homomorphisms $\hat\phi_1,\hat\phi_2:A\rightarrow B$. Let $k\in\Z$ be a positive contraction of full spectrum and form the maps $\phi_i:C_0((0,1])\otimes A\rightarrow \Z\otimes B \cong B$ as described above.  Then $\phi_1$ and $\phi_2$ are approximately unitarily equivalent if and only if  $\hat\phi_1$ and $\hat\phi_2$ agree on traces.
\end{thm}

One can decompose a $^*$-homomorphism $\phi:A\rightarrow B$ as the sum $h\otimes \phi+(1-h)\otimes \phi$ of two maps of the form considered by Theorem \ref{onecolourclass}.  This leads to a new notion of equivalence of morphisms in the spirit of Cuntz-Pedersen equivalence of positive elements from \cite{CP:JFA} and the 2-coloured equivalence of projections lemma of Matui and YS (\cite[Lemma 2.1]{MS:DMJ}). We say that two unital $^*$-homomorphisms $\phi_i:A\rightarrow B$ ($i=1,2$) are $2$-coloured equivalent, if there exist $w^{(0)},w^{(1)}$ in $B$ such that 
\begin{align}\phi_1(a)&=w^{(0)}\phi_2(a)w^{(0)}{}^*+w^{(1)}\phi_2(a)w^{(1)}{}^*,\text{ and }\nonumber\\
\phi_2(a)&=w^{(0)}{}^*\phi_1(a)w^{(0)}+w^{(1)}{}^*\phi_1(a)w^{(1)},\quad a\in A,\label{e1.1}
\end{align} and such that $w^{(i)}{}^*w^{(i)}$ commutes with the image of $\phi_2$ and $w^{(i)}w^{(i)}{}^*$ commutes with the image of $\phi_1$.  This forms part of a general definition of $n$-coloured equivalence between unital $^*$-homomorphisms (Definition \ref{def:ColourEquiv}) for which the $n=1$ case is precisely unitary equivalence. 

In general we need approximate coloured equivalence (i.e., $2$-coloured equivalence upon embedding the codomain into the ultrapower $B_\omega$). In this way the approximate unitary equivalence classification result of Theorem \ref{onecolourclass} gives rise to the following theorem, which classifies morphisms by traces up to $2$-coloured equivalence; more general versions which do not require simplicity of $A$ will be given in Section \ref{sec:2colour}.

\begin{thm}\label{thm:2colouredIntroVersion}
Let $A$ be a simple, separable, unital, nuclear $\mathrm{C}^{*}$-algebra, and let $B$ be a simple, separable, unital, and exact $\mathrm{C}^*$-algebra such that $B$ is $\Z$-stable and $T(B)$ is a Bauer simplex. Let $\phi_1,\phi_2:A \to B_\omega$ be unital $^*$-homomorphisms. Then $\phi_1$ and $\phi_2$ are approximately $2$-coloured equivalent if and only if $\tau \circ \phi_1 = \tau \circ \phi_2$ for every $\tau \in T(B_\omega)$.
\end{thm}

We view Theorem \ref{thm:2colouredIntroVersion} as a $2$-coloured $\mathrm{C}^*$-algebraic version of the classical fact that normal $^*$-homomorphisms $\M\rightarrow \mathcal N^\omega$ from a finite, injective, separably acting von Neumann algebra $\M$ into the ultrapower of a II$_1$ factor are classified, up to unitary equivalence, by their tracial data.

We can extend  $2$-coloured classification results to the situation where only one of $\phi_1$ or $\phi_2$ is a $^*$-homomorphism and the other an order zero map (albeit with a slightly weaker conclusion).  This enables us to calculate the nuclear dimension of broad classes of $\mathrm C^*$-algebras, which will imply Theorems \ref{thm:TW} and \ref{thm:Rigidity}.

In the stably finite setting, there is also a sharper version of Conjecture~\ref{Con:TW} which replaces nuclear dimension by decomposition rank (cf.\ \cite{KW:IJM}) in (i). It follows from \cite{MS:DMJ} and \cite{SWW:arXiv} that in the monotracial case the dividing line between decomposition rank and nuclear dimension is precisely quasidiagonality. Our methods carry this statement over to more general trace spaces;  \emph{quasidiagonality of all traces} is the property that distinguishes decomposition rank from nuclear dimension. The following summarizes our dimension computations in the stably finite case (we exclude the approximately finite dimensional (AF) algebras, as these are precisely the $\mathrm{C}^*$-algebras with nuclear dimension zero \cite[Remark 2.3(iii)]{WZ:Adv}).

\begin{thm}
\label{thm:dn1}
If $A$ is a non-AF, simple, separable, unital, nuclear, $\mathcal Z$-stable $\mathrm C^*$-algebra such that $T(A)$ is a Bauer simplex, then the nuclear dimension of $A$ is $1$.
Furthermore, if every $\tau \in T(A)$ is quasidiagonal, then the decomposition rank of $A$ is $1$.
\end{thm}

Here the optimal bound of $1$ on the nuclear dimension can be viewed as a $2$-coloured approximate finite dimensionality; one can approximate the nuclear $\mathrm{C}^*$-algebras of Theorem \ref{thm:dn1} by two order zero images of finite dimensional algebras. The two order zero maps arise from the two colours in Theorem \ref{thm:2colouredIntroVersion}.  

An alternative --- in fact almost orthogonal --- approach to the implication (ii) $\Longrightarrow$ (i) of the Toms-Winter conjecture is through the examination of $\mathrm{C}^*$-algebras arising as inductive limits of concrete building blocks.  This was initiated by AT and WW in \cite{TW:APDE}, who showed that $\Z$-stable approximately homogeneous (AH) $\mathrm{C}^*$-algebras have decomposition rank at most $2$, and very recently extended in \cite{ENST:preprint} (with the same estimate) to $\Z$-stable ASH algebras.  These results, and the dimension computations in the present paper complement each other nicely: here we require no inductive limit structure, and so our methods work under abstract axiomatic conditions; in particular we obtain nuclear dimension results in the absence of the Universal Coefficient Theorem (UCT) and quasidiagonality. In contrast, the results of \cite{TW:APDE,ENST:preprint} do not require simplicity or tracial state space restrictions, but only work for algebras with very concrete building blocks.

We end the paper by showing how our techniques can also be used to calculate the nuclear dimension of Kirchberg algebras without a UCT assumption. Aiming for the exact value of $1$ (and hence lowering the bound for Kirchberg algebras from \cite{MS:DMJ}) was inspired by recent work of Ruiz, Sims and S\o{}rensen (\cite{RSS:arXiv}) who obtain the optimal estimate for UCT Kirchberg algebras.

\begin{thm}
\label{thm:Kirchberg,boom!}
The nuclear dimension of every Kirchberg algebra is $1$.  
\end{thm}

\bigskip
In summary, our \emph{formal main results} are the dimension computations of Theorems \ref{thm:dn1} and \ref{thm:Kirchberg,boom!}, together with their implications for the Toms-Winter conjecture and the Elliott program, Theorems \ref{thm:TW} and \ref{thm:Rigidity}. 

The main \emph{conceptual} novelties are (a) the systematic use of von Neumann bundle techniques in connection with central sequence algebra methods, (b) the classification of order zero maps up to approximate unitary equivalence by tracial information (Theorem \ref{onecolourclass}) and, based on this, (c) the $n$-coloured classification of $^{*}$-homomorphisms, again in terms of tracial data (Theorem \ref{thm:2colouredIntroVersion}).

In the remainder of this introduction we will say a bit more about the strategies of the proofs and the architecture of the paper.

\bigskip
\subsection*{\sc From classification to nuclear dimension}\hfill \\

\noindent
Before outlining the proofs of Theorems \ref{onecolourclass} and \ref{thm:2colouredIntroVersion}, which are long and quite technical, let us briefly explain how we get from there to Theorem \ref{thm:dn1} (from which Theorem \ref{thm:TW} is immediate), following Connes' strategy for proving that an  injective II$_1$ factor $\M$ is hyperfinite  (\cite{C:Ann}). This contains 
three main ingredients:
\begin{enumerate}
\item\label{CS.1} $\M$ is McDuff, i.e., $\M\cong \M\vnotimes \R$;
\item\label{CS.2} there exists a unital embedding
$\theta:\M\hookrightarrow\R^\omega$, into the ultrapower of the hyperfinite
II$_1$ factor;
\item\label{CS.3} $\M$ has approximately inner flip, meaning $x \otimes y \mapsto y\otimes x$ on $\M\bar{\otimes} \M$ is approximately inner.
\end{enumerate}
Using (iii), one compares  the first factor
embedding $\M\hookrightarrow (\M\vnotimes \R)^\omega$ with $x\mapsto
1\otimes \theta(x)$ using the approximately inner flip.  This moves
the finite dimensional approximations from $\R^\omega$ back into $\M$
yielding hyperfiniteness.

$\mathrm C^*$-algebra versions of this argument appeared in work of Effros
and Rosenberg (\cite{ER:PJM}), and provide the strategy used in
\cite{MS:DMJ} and \cite{SWW:arXiv}.  In the $\mathrm C^*$-algebra context, (i)
is naturally replaced by $\Z$-stability which, unlike the von Neumann algebra 
situation, must be imposed as a hypothesis; the exotic examples of
\cite{R:Acta,T:Ann,V:JFA,V:JAMS} are not $\Z$-stable.  Very recently YS, SW and WW proved a version
of (ii), with order zero maps in place of $^{*}$-homomorphisms (\cite[Proposition 3.2]{SWW:arXiv}). Having approximately inner flip in the $\mathrm C^*$-context is far too restrictive (e.g., an AF algebra with this property must be UHF; see also \cite{T:arXiv}, which characterizes those classifiable $\mathrm{C}^*$-algebras with approximately inner flip), but Matui and YS took a major step 
toward circumventing this problem by proving that simple, nuclear, UHF-stable $\mathrm C^*$-algebras with unique trace
have a $2$-coloured approximately inner flip (this is, in fact, a special case of Theorem \ref{thm:2colouredIntroVersion}). With an additional $2$-coloured reduction technique to pass from UHF-stable to $\Z$-stable $\mathrm C^*$-algebras, these $\mathrm C^*$-analogues of (i)-(ii) were used in \cite{SWW:arXiv} to prove that simple, separable, unital, nuclear, $\Z$-stable $\mathrm C^*$-algebras with unique trace have nuclear dimension at most 3. 

To extend past the monotracial situation, a replacement for a $2$-coloured
approximately inner flip is required, as such a flip allows for at most one trace (see \cite[Proposition 4.3]{SWW:arXiv}). At least for Bauer simplices, the $2$-coloured classification theorems (Theorem \ref{thm:2colouredIntroVersion} and its technical extensions allowing one of the two maps to be c.p.c.\ order zero) provide such a tool. Indeed, with $A$ as in Theorem \ref{thm:dn1}, we are able to compare the first factor embedding
$A\hookrightarrow (A\otimes\Z)_\omega$ with suitable maps from $A$
into $(A\otimes \Z)_\omega$ which factor through ultrapowers of finite
dimensional algebras, and agree with the first factor embedding on
traces.  We produce these maps by gluing together maps which behave
well on individual extremal traces of $A$ over the boundary, with
the individual maps either arising from \cite{SWW:arXiv} or from the
quasidiagonality assumption in Theorem \ref{thm:dn1}.  In contrast to
\cite{SWW:arXiv} this technique enables both $2$-colouring arguments
to be performed simultaneously, leading to nuclear dimension $1$ rather than anything greater.

\bigskip
\subsection*{\sc Outline of the proof of Theorem \ref{onecolourclass}} \hfill \\

\noindent
We outline the proof of Theorem \ref{onecolourclass}, from which the $2$-coloured classification theorem (Theorem \ref{thm:2colouredIntroVersion}) follows, so let $A$ and $B$ be as in Theorem \ref{onecolourclass}. Write $B_\omega$ for the ultrapower of $B$ and consider two unital $^*$-homomorphisms $\hat{\phi}_1,\hat{\phi}_2:A\rightarrow B_\omega$ which agree on traces. For this exposition, we additionally assume that $A$ is simple and that $B$ is nuclear.

Identify $B_\omega$ with $(B\otimes\Z)_\omega$, and fix a positive contraction $k\in \Z$ with full spectrum and consider the c.p.c.\ order zero maps $\phi_i:=\hat{\phi}_i(\cdot)\otimes k$ (approximate unitary equivalence of these c.p.c.\ order zero maps is the same as approximate unitary equivalence of the $^*$-homomorphisms $\phi_1,\phi_2$ appearing in Theorem \ref{onecolourclass}, see Proposition \ref{prop:Ord0Structure}).   To establish the unitary equivalence of $\phi_1$ and $\phi_2$, we use an order zero version of a $2\times 2$ matrix trick invented by Connes (\cite{Connes:ASENS}), to reduce the problem of classifying 
maps to that of classifying positive elements, with the complexity transferred to a relative commutant sequence algebra.  Define a $^*$-homomorphism $\pi:A\rightarrow M_2(B_\omega)$ by
\begin{equation}
\pi(x)=\begin{pmatrix}\hat{\phi}_1(x)\otimes 1_\Z&0\\0&\hat{\phi}_2(x)\otimes 1_\Z\end{pmatrix},\quad x\in A,
\end{equation}
and form the relative commutant sequence algebra $C:=M_2(B_\omega)\cap \pi(A)'$, noting that
\begin{equation}
\label{eq:IntroDefhi}
h_1:=\begin{pmatrix}\phi_1(1_A)&0\\0&0\end{pmatrix},\ h_2:=\begin{pmatrix}0&0\\0&\phi_2(1_A)\end{pmatrix}\in C.
\end{equation}
The $2\times 2$ matrix trick will show that $\phi_1$ and $\phi_2$ are unitarily equivalent in $B_\omega$ whenever $h_1$ and $h_2$ are unitarily equivalent in $C$; given the structure of $C$ it suffices to show that $h_1$ and $h_2$ are approximately unitarily equivalent.

A key tool for establishing approximate unitary equivalence of two positive elements $h_1$ and $h_2$ in a $\mathrm C^*$-algebra $C$ is the Cuntz semigroup, dating back to \cite{Cuntz:MA}.
The positive contractions $h_1$ and $h_2$ induce homomorphisms between the Cuntz semigroups of $C_0((0,1])$ and of $C$.  Using a modification of an argument of Robert and Santiago from \cite{RS:JFA}, we can show that, for the relative commutant sequence algebras used in the $2\times 2$ matrix trick, two totally full positive contractions $h_1$ and $h_2$ are unitarily equivalent if and only if they induce the same map at the level of the Cuntz semigroup (here totally full means that $f(h_i)$ is full in $C$ for all nonzero $f\in C_0((0,\|h_i\|])_+$). To show that $h_1$ and $h_2$ induce the same Cuntz semigroup maps, we develop structural properties of $C$.   The ultrapower $M_2(B_\omega)$ is equipped with the trace kernel ideal $J$, which gives the ideal $J\cap C\lhd C$. When $A$ is simple, Matui and YS show (using the language we adopt in this paper, see Definition \ref{def:SI}) in \cite{MS:DMJ} that the $^*$-homomorphism $\pi$ has property (SI). This is a technical tool which enables the transfer of some structural properties from the quotient $C/(J\cap C)$ back to $C$. For example, in the presence of property (SI), all traces on $C$ factor through $C/(J\cap C)$.

When $B$ has a unique trace $\tau$, the algebra $C/(J\cap C)$ is a finite von Neumann algebra: in fact it is the relative commutant $\R^\omega \cap \bar{\pi}(A)'$, where $\R^\omega$ is the ultrapower of the hyperfinite II$_1$ factor, and $\bar{\pi}$ is the induced map obtained from the quotient map $M_2(B_\omega)/J\cong \R^\omega$. As such, $C/(J\cap C)$ has a wealth of structural properties. For example, $C/(J\cap C)$ has strict comparison of positive elements. Using property (SI), strict comparison passes back to $C$. We can also use the structure of $C/(J\cap C)$ to see that every trace on $C$ lies in the closed convex hull of those traces of the form $\tau(\pi(a)\cdot)$  for some $a\in A_+$. Thus these latter traces can be used as a test set to show that $\rho(f(h_1))=\rho(f(h_2))>0$ for all traces $\rho$ on $C$ and nonzero $f\in C_0((0,1])_+$.  Strict comparison is then used to see that $h_1$ and $h_2$ are totally full in $C$ and induce the same Cuntz semigroup map.  This strategy of examining tracial behaviour of $h_1$ and $h_2$ inside $C$ has its spiritual origins in Haagerup's proof of injectivity implies hyperfiniteness (see \cite[Proof of Theorem 4.2]{H:JFA}).

The sketch of the previous paragraph works equally well when $B$ has finitely many extremal traces, as in this case $C/(J\cap C)$ is again a finite von Neumann algebra. In general, however, $C/(J\cap C)$ is not a von Neumann algebra, but when the set of extremal traces of $B$ is compact, $C/(J\cap C)$ is a relative commutant inside an ultraproduct of $\mathrm W^*$-bundles, in the sense of Ozawa, of finite von Neumann algebras\footnote{In fact each fibre is the hyperfinite II$_1$ factor.} over this compact extremal tracial boundary.  The $\Z$-stability assumption on $B$ ensures that these bundles are trivial by \cite{O:JMSUT}, which enables us to extend certain von Neumann algebraic structural properties to these bundles, in Section \ref{sec:Bundles}. In this way we obtain strict comparison and a test set of traces for $C$ via a more general version of property (SI), from which Theorem \ref{onecolourclass} follows as in the unique trace case.

\bigskip
\subsection*{\sc Structure of the paper} \hfill \\

\noindent
While we described the strategy for maps $A\rightarrow B_\omega$, where $B_\omega$ is the ultrapower of a fixed finite, simple, separable, unital, exact and $\Z$-stable $\mathrm{C}^*$-algebra, the proofs work when $B_\omega=\prod_\omega B_n$ is the ultraproduct of any sequence $(B_n)_{n=1}^\infty$ of such algebras. Further exactness is only used in Theorems \ref{onecolourclass} and \ref{thm:2colouredIntroVersion} in order to access Haagerup's result that quasitraces are traces for exact algebras (\cite{H:QTrace}). We will work in the generality of ultraproducts of $\mathrm{C}^*$-algebras whose quasitraces are traces in the main body of the paper.

In order to use the $2$-coloured strategy to prove Theorem \ref{thm:dn1}, we must actually work in the situation where the second map is only an order zero map.\footnote{The meaning of agreement on traces is strengthened in this case, see Theorem \ref{thm7.8}.} This gives rise to additional technicalities, which take up a significant part of the next 5 sections of the paper.  For example, it will not in general be possible to construct a $^*$-homomorphism $\pi$ used to define the algebra $C$ above, and we instead work with an order zero map $\pi$ and modify the definition of $C$ accordingly. Such \emph{supporting order zero maps} are constructed in Section \ref{sect:2} which also collects a number of preliminary facts regarding order zero maps, traces, ultraproducts, approximations by invertibles, and $\Z$-stability.  Section \ref{sec:Matrix} contains the version of our $2\times 2$ matrix trick for the order zero relative commutant sequence algebras we use.  In Section \ref{sec:Bundles} we develop structural properties of ultrapowers of $\mathrm{W}^*$-bundles, which we transfer to the algebras $C$ in Section \ref{sec:StrictComp}, via a more flexible notion of property (SI) which allows for nonsimple domain algebras and order zero maps.  In Section \ref{sec:TotallyFullClass}, we give our version of the Robert-Santiago classification result and use this to obtain Theorem \ref{KeyLemmaP} (which is the technical version of Theorem \ref{onecolourclass}) from which Theorems \ref{thm:TW}, \ref{thm:Rigidity}, \ref {thm:2colouredIntroVersion}, \ref{thm:dn1} and \ref{thm:Kirchberg,boom!} will follow.

In Section \ref{sec:2colour} we prove Theorem \ref{thm:2colouredIntroVersion} (as Theorem \ref{thm:2ColourUniqueness}) and further develop the notion of coloured equivalence. Our covering dimension estimates are then given in Section \ref{sec:FiniteDim}, which proves Theorem \ref{thm:dn1} (as Theorem \ref{thm:FiniteDim}). Section \ref{sec:TWQD} revisits the notion of quasidiagonal traces, and shows that vast classes of simple, stably finite, nuclear $\mathrm C^*$-algebras have the property that all traces are quasidiagonal; in this section, we also prove Theorem \ref{thm:Rigidity} (as Corollary \ref{cor8.6}).  
The paper ends with Section \ref{sec:KirAlgs}, in which we prove Theorem \ref{thm:Kirchberg,boom!}
 (as Corollary \ref{DimKirchberg}), and also show how our $2$-coloured methods can be applied in the setting of Kirchberg algebras. Indeed, in the absence of traces the purely infinite analogue of Theorem \ref{onecolourclass} provides a classification result for suitable order zero maps (Theorem \ref{thm:TotFullCpcKirClass}).
 
\bigskip
\subsection*{\sc Acknowledgements} \hfill \\

\noindent
This project has benefited greatly from conversations with a number of colleagues; we are particularly indebted to Rob Archbold, George Elliott, Thierry Giordano, Eberhard Kirchberg, Hiroki Matui, Narutaka Ozawa, Leonel Robert, Mikael R{\o}rdam, Luis Santiago, Andrew Toms, and Joachim Zacharias for sharing their conceptual insights and for plenty of inspiring discussions. We thank the organisers and funders of the following meetings: BIRS workshop `Dynamics and $\mathrm C^*$-algebras: Amenability and Soficity'; GPOTS 2015; Scottish Operator Algebras Seminars. We also thank Etienne Blanchard, Jorge Castillejos-Lopez, Samuel Evington, and Ilijas Farah for their helpful comments on earlier versions of this paper. Finally, we thank the referee for their careful reading and helpful comments on the previous version of this paper.

\renewcommand*{\thetheorem}{\roman{theorem}}
\setcounter{theorem}{0}
\numberwithin{theorem}{section}

\clearpage\section{Preliminaries}\label{sect:2}

\noindent
Let $A$ be a $\mathrm C^*$-algebra. Denote the positive cone of $A$ by $A_+$, the set of contractions by $A^1$, and the set of positive contractions by $A_+^1$.
Let $a,b \in A$. For $\eps > 0$, we write $a \approx_\eps b$ to mean $\|a-b\|<\eps$.
When $b$ is self-adjoint we write $a \vartriangleleft b$ to mean that $ba=ab=a$. In the case that $A$ is unital, this is equivalent to $a\in\{1_A-b\}^\perp$; we will often use this latter notation, particularly when $b$ is a positive contraction.
By a \emph{hereditary subalgebra} of $A$, we mean a $\mathrm C^*$-subalgebra $B$ of $A$ satisfying $x \in B_+$ whenever $0 \leq x \leq y \in B_+$; when $a \in A_+$, the hereditary subalgebra generated by $a$ is denoted $\mathrm{her}(a)$ and is equal to $\overline{aAa}$.

We use two pieces of functional calculus repeatedly in the paper.  Given $\eps>0$ and a positive element $a$ in a $\mathrm{C}^*$-algebra, write $(a-\eps)_+$ for the element obtained by appplying the function $t\mapsto \max(0,t-\eps)$ to $a$.  The function\footnote{When we use $g_\eps$ to define $g_\eps (a)$, we only need to consider $g_\eps$ as a function in $C_0((0,\|a\|])$.} $g_\eps \in C_0((0,\infty])$ is defined by
\begin{equation}
\label{eq:gepsDef}
g_\eps|_{(0,\eps/2]} \equiv 0,\quad g_\eps|_{[\eps,\infty]}\equiv 1, \quad\text{and}\quad g_\eps|_{[\eps/2,\eps]}\text{ is linear.} \end{equation}
In this way $(a-\eps)_+\vartriangleleft g_\eps(a)$ for $a\geq 0$.

The Jiang-Su algebra, denoted by $\Z$, was introduced in \cite{JS:AJM} and is central to this paper.  It is a simple, unital $\mathrm C^*$-algebra with a unique trace and the same $K$-theory as $\mathbb C$. Moreover $\Z$ is strongly self-absorbing in the sense that the first factor embedding $\Z \to \Z \otimes \Z$ is approximately unitarily equivalent to an isomorphism (\cite[Theorems 3 and 4]{JS:AJM}).
A $\mathrm C^*$-algebra $A$ is \emph{$\Z$-stable} if it is isomorphic to $A \otimes \Z$.  This is a regularity property, which gives rise to a wealth of additional structural properties (including the main results of the present article).  We will collect a number of these of relevance to us in Lemma \ref{lem:Zfacts} below.

Another key notion is that of totally full $^*$-homomorphisms.
In \cite[Definition 2.8]{DE:PLMS}, these were called ``full embeddings'' (in the unital case).

\begin{defn}
\label{def:TotallyFull}
Let $A,B$ be $\mathrm C^*$-algebras, let $\pi:A \to B$ be a ${}^*$-homo\-morphism. We say that
$\pi$ is \emph{totally full} if for every nonzero element $a\in A$, $\pi(a)$ is full in $B$ (i.e., $\pi(a)$ generates $B$ as a closed two-sided ideal).
Likewise, a positive element $b \in B_+$ is said to be \emph{totally full} if $b\neq 0$ and the $^*$-homomorphism $C_0((0,\|b\|]) \to B$ sending $\mathrm{id}_{(0,\|b\|]}$ to $b$ is totally full.
\end{defn}

\bigskip

\subsection{\sc Order zero maps}\hfill  \\

\noindent
Over the last 10 years it has become apparent that it is very natural to work with the class of maps between $\mathrm{C}^*$-algebras which preserve the order structure and orthogonality. Building on the work of Wolff (\cite{W:AM}), the structure theory for these maps was developed in \cite{WZ:MJM}.

\begin{defn} [{\cite[Definition 1.3]{WZ:MJM}}]
Let $A,B$ be $\mathrm C^*$-algebras. A completely positive and contractive (c.p.c.) map $\phi:A \to B$ is said to be \emph{order zero} if, for every $a,b \in A_+$ with $ab=0$, one has $\phi(a)\phi(b)=0$.
\end{defn}

\begin{prop}[{\cite[Corollary 3.1]{WZ:MJM}}]\label{prop.Cone}
\label{prop:Ord0Structure}
Let $A,B$ be $\mathrm C^*$-algebras.
There is a one-to-one correspondence between c.p.c.\ order zero maps $\phi:A \to B$ and ${}^*$-homomorphisms $\pi:C_0((0,1]) \otimes A \to B$, where $\phi$ and $\pi$ are related by the commuting diagram
\begin{equation}
\xymatrix{
A \ar[rr]^-{a \mapsto \mathrm{id}_{(0,1]} \otimes a} \ar[drr]_-{\phi} && C_0((0,1]) \otimes A \ar[d]^{\pi} \\
&& B \, .
}
\end{equation}
\end{prop}

It follows that, when $A$ is unital, c.p.c.\ order zero maps $\phi:A\to B$ are characterized as the c.p.c.\ maps satisfying the \emph{order zero identity}
\begin{equation}
\label{eq:Ord0Ident}
\phi(ab)\phi(1_A)=\phi(a)\phi(b),\quad a,b\in A.
\end{equation}
From this relation one can read off the following well known fact.
\begin{prop}
\label{prop:Ord0Homo}
Let $A,B$ be $\mathrm C^*$-algebras such that $A$ is unital, and let $\phi:A \to B$ be a c.p.c.\ order zero map.
Then $\phi$ is a ${}^*$-homomorphism if and only if $\phi(1_A)$ is a projection.
\end{prop}

There is a continuous functional calculus for c.p.c.\ order zero maps (see \cite[Corollary 4.2]{WZ:MJM}): given a c.p.c.\ order zero map $\phi:A \to B$ and a positive contraction $f \in C_0((0,1])_+$, the c.p.c.\ order zero map $f(\phi):A\to B$ is defined as follows.
If $\pi:C_0((0,1]) \otimes A \to B$ is the ${}^*$-homomorphism satisfying $\phi(a) = \pi(\mathrm{id}_{(0,1]} \otimes a)$ then
\begin{equation}\label{OZFC} f(\phi)(a) := \pi(f \otimes a),\quad a\in A.
\end{equation}
Throughout this paper, when $\phi$ is an order zero map and $n\in\N$, $\phi^n$ refers to the functional calculus output, i.e., $\phi^n = f(\phi)$ where $f(t)=t^n$.

\bigskip

\subsection{\sc Traces and Cuntz comparison}\hfill  \\

\noindent
Let $A$ be a $\mathrm C^*$-algebra. In this paper we only consider bounded traces, so for us a \emph{trace} on $A$ is a state $\tau:A \to \mathbb C$ such that $\tau(ab)=\tau(ba)$ for every $a,b\in A$. (We use the term \emph{tracial functional} for positive bounded functionals with the trace property.) The set of all traces on $A$ is denoted $T(A)$. When $A$ is unital, this set is convex and compact under the weak-$^*$ topology, so has an extreme boundary, $\partial_eT(A)$; in fact, $T(A)$ is a Choquet simplex (\cite[Theorem II.6.8.11]{B:Encyc}).

Let $\tau:A \to \mathbb C$ be a trace.
Then, by tensoring with the non-normalized tracial functional on $M_k$, $\tau$ induces a tracial functional (also denoted $\tau$) on $M_k(A)$ for every $k\in \N$.
Also, it produces a function $d_\tau:M_k(A)_+ \to [0,\infty)$ defined by
\begin{equation} d_\tau(a) := \lim_{n\to\infty} \tau(a^{1/n}). \end{equation}

Let $A$ be a $\mathrm C^*$-algebra and let $a,b \in M_k(A)_+$ for some $k\in \N$.
Cuntz subequivalence, which has its origins in \cite{C:MS,Cuntz:MA}, is the relation $\preceq$ defined by $a \preceq b$ if there exists a sequence $(x_n)_{n=1}^\infty$ in $M_k(A)$ such that
\begin{equation} \lim_{n\to\infty} \|a-x_n^*bx_n\| = 0.
\end{equation}
Say that $a$ and $b$ are \emph{Cuntz equivalent}, if $a \preceq b$ and $b \preceq a$.
The Cuntz semigroup $W(A)$ (first studied in this form in \cite{R:JFA2}) is the quotient of $\bigcup_{n=1}^\infty M_n(A)_+$ by Cuntz equivalence,
where each $M_n(A)_+$ is treated as a subset of $M_{n+1}(A)_+$ by embedding it in the top-left corner. The class of $a \in M_n(A)_+$ in $W(A)$ is denoted $[a]$.
The set $W(A)$ inherits a semigroup structure defined by
\begin{equation}
[a] + [b] = \left[ \begin{pmatrix} a&0 \\ 0&b \end{pmatrix} \right].
\end{equation}
The Cuntz semigroup $W(A)$ is additionally endowed with an order induced by the preorder $\preceq$.

Recall from \cite[Theorem II.2.2]{BH:JFA} that for $\tau\in T(A)$, and for $a,b \in M_k(A)_+$, with $a \preceq b$ we have
\begin{equation} d_\tau(a) \leq d_\tau(b).\end{equation}
See \cite{APT:Contemp} for more about the Cuntz semigroup.

\begin{defn}[{cf.\ \cite[Section 6]{B:LMSnotes}}]
\label{defn:StrictComp}
The $\mathrm C^*$-algebra $A$ has \emph{strict comparison (of positive elements, with respect to bounded traces)}, if 
\begin{align}
\label{eq:StrictComp}
(\forall \tau\in T(A),\ d_\tau(a) < d_\tau(b))\Longrightarrow a\preceq b
\end{align}
for $k\in\N$ and $a,b\in M_k(A)_+$.
\end{defn}

Note that strict comparison, as in Definition \ref{defn:StrictComp}, is a property of the classical Cuntz semigroup $W(A)$ rather than the complete version $\Cu(A):=W(A\otimes\mathcal K)$ developed in \cite{CEI:Crelle}. Also, we emphasize that although we will be applying strict comparison to nonsimple $\mathrm C^*$-algebras, the above definition can be used to obtain Cuntz comparison \textit{only} when the element $b$ is full---for otherwise, $d_\tau(b)=0$ can occur for some trace $\tau$. Blackadar originally considered this property in the case that $A$ is a simple $\mathrm C^*$-algebra (\cite{B:LMSnotes}).
In the nonsimple case, a variation called ``almost-unperforation of the Cuntz semigroup'' is often the more suitable property, as it says something about Cuntz comparison of nonfull elements, and is therefore strictly stronger than the property we have defined. Nonetheless, the definition above is most useful for the arguments in this article.

Functionals on the Cuntz semigroup arise from $2$-quasitraces on the $\mathrm{C}^*$-algebra (\cite[Proposition 4.2]{ERS:AJM}).  It remains an open problem as to whether all $2$-quasitraces are traces; famously this is the case for exact $\mathrm{C}^*$-algebras by the work of Haagerup (\cite{H:QTrace}).  We will often need as a hypothesis on a $\mathrm C^*$-algebra that all its $2$-quasitraces are traces, and we write $QT(A)=T(A)$ to indicate that this condition holds for a $\mathrm C^*$-algebra $A$ (although the exact definition of a $2$-quasitrace will not be needed).  Note too that we do not ask that $d_\tau(a) < d_\tau(b)$ for all lower semicontinuous $2$-quasitraces in Definition \ref{defn:StrictComp}, as the definition we give here works best for our arguments (in retrospect, for the $\mathrm C^*$-algebras that we show this concept applies to, all quasitraces turn out to be traces, see Remark \ref{rmk:QTT}).

\begin{remark}\label{EQTT}
In \cite[Corollary 4.6]{R:IJM}, R\o{}rdam shows that simple, separable, unital, exact, $\Z$-stable $\mathrm{C}^*$-algebras have strict comparison (of positive elements by bounded traces).  Examining this proof, exactness is only used to access Haagerup's result (\cite{H:QTrace}) (see also the discussion in footnote \footref{CUWSC} below). As such every simple, separable, unital, $\Z$-stable $\mathrm{C}^*$-algebra $A$ with $QT(A)=T(A)$ has strict comparison in the sense of Definition \ref{defn:StrictComp}.
\end{remark}

\begin{remark}
\label{rmk:StrictCompFullHered}
If $A$ has strict comparison with respect to bounded traces, and $B$ is a full hereditary subalgebra of $A$, then $B$ also has strict comparison with respect to bounded traces.
This is because every trace $\tau\in T(A)$ restricts to a nonzero bounded tracial functional on $B$ (and hence a trace after renormalizing), so that if $a,b \in B_+$ satisfy \eqref{eq:StrictComp} for traces in $T(B)$ then \eqref{eq:StrictComp} holds for traces in $T(A)$, whence $a \preceq b$ in $A$ and therefore in $B$ (by \cite[Lemma 2.2(iii)]{KR:AJM}).
\end{remark}

\bigskip

\subsection{\sc Ultraproducts and the reindexing argument}\hfill  \\

\label{sec:Reindexing}

\noindent
Throughout this article, $\omega$ will always denote a free ultrafilter on $\N$, which we regard as fixed.
Let $(B_n)_{n=1}^\infty$ be a sequence of $\mathrm C^*$-algebras.
The bounded sequence algebra is defined to be
\begin{equation} \prod_{n=1}^\infty B_n := \{(b_n)_{n=1}^\infty \mid b_n \in B_n \text{ and } \|(b_n)_{n=1}^\infty\|:= \sup_n \|b_n\| < \infty\}. \end{equation}
The ultraproduct $\mathrm C^*$-algebra is defined to be
\begin{equation}
\prod_\omega B_n := \prod_{n=1}^\infty B_n \big/\{ (b_n)_{n=1}^\infty \mid \lim_{n\to\omega} \|b_n\| = 0\}.
\end{equation}
We will often also use $B_\omega$ to denote the ultraproduct $\prod_\omega B_n$.

Suppose now that for all $n$ we have  $T(B_n)\neq \emptyset$. The \emph{trace-kernel ideal} is defined to be
\begin{equation}
\label{eq:JBdef}
J_{B_\omega} := \{(b_n)_{n=1}^\infty \in B_\omega \mid \lim_{n\to \omega} \max_{\tau \in T(B_n)} \tau(b_n^*b_n) = 0\}. \end{equation}
Let $T_\omega(B_\omega)$ be the collection of \emph{limit traces} on $B_\omega$, i.e., those $\tau\in T(B_\omega)$ of the form $\tau((b_n)_{n=1}^\infty)= \lim_{n\to\omega} \tau_n(b_n)$ for some sequence of traces $\tau_n\in T(B_n)$.  Then $J_{B_\omega}$ consists of $b \in B_\omega$ for which
\begin{equation}
\label{eq:JBeqDef} \tau(b^*b) = 0, \quad \tau \in T_\omega(B_\omega); \end{equation}
equivalently,
\begin{equation}
\label{eq:JBeqDef2} \tau(|b|)=0,\quad \tau \in T_\omega(B_\omega), \end{equation}
where this last equivalence follows from $\tau(|b|)^2 \leq \tau(b^*b)\leq \tau(|b|)\|b\|$ (cf.\ \cite[Definition 4.3]{KR:Crelle}).
Here is a simple fact regarding the trace-kernel quotient, using this last characterization.

\begin{lemma}
\label{lem:ProjectionModJ}
Let $(B_n)_{n=1}^\infty$ be a sequence of unital $\mathrm C^*$-algebras such that $T(B_n)\neq \emptyset$ for all $n$, and set $B_\omega := \prod_\omega B_n$.
Let $a,e \in (B_\omega)_+$ be contractions such that $a \vartriangleleft e$ and $\tau(e) = d_\tau(a)$ for all $\tau \in T_\omega(B_\omega)$.
Then the image of $e$ in $B_\omega/J_{B_\omega}$ is a projection.
\end{lemma}

\begin{proof}
For $\tau \in T_\omega(B_\omega)$, we have
\begin{equation}
\tau(e^2) \leq \tau(e) = \lim_n \tau(a^{1/n}) 
\stackrel{a\vartriangleleft e}{=}
 \lim_n \tau(ea^{1/n}e) \leq \tau(e^2).
\end{equation}
Since $e-e^2 \geq 0$, it follows that $e-e^2 \in J_{B_\omega}$, as required.
\end{proof}

It is of particular relevance to this paper to know when the limit traces are weak$^*$-dense in $T(B_\omega)$, as in particular $J_{B_\omega}$ will then be the set of those $b\in B_\omega$ for which \eqref{eq:JBeqDef} holds for all $\tau \in T(B_\omega)$.  This density fails in general (\cite{BF:HMJ}), but does hold for $\mathrm{C}^*$-algebras with suitable regularity properties. Indeed, Ozawa establishes such a density result (\cite[Theorem 8]{O:JMSUT}) when each $B_n$ is $\Z$-stable and exact. This was generalised by Ng and Robert in \cite[Theorem 1.2]{NR:arXiv} to the case where each $B_n$ is unital and has strict comparison of full positive elements by bounded traces.\footnote{Note that Ng and Robert use the complete Cuntz semigroup $\Cu(A):=W(A\otimes\mathcal K)$ to define strict comparison, whereas in this paper we work primarily with the incomplete version $W(A)$.}  By results essentially due to R\o{}rdam (\cite[Section 4]{R:IJM}) this happens when each $B_n$ is additionally simple, $\Z$-stable and has $QT(B_n)=T(B_n)$.\footnote{\label{CUWSC}There is a detail required here to take care of the difference between $W(A)$, which R\o{}rdam works with in \cite{R:IJM}, and $\Cu(A)$, used in \cite{NR:arXiv}. R\o{}rdam shows in \cite[Theorem 4.5]{R:IJM} that any $\Z$-stable $\mathrm{C}^*$-algebra $B$ has almost unperforated Cuntz semigroup $W(B)$. Applying this to $B:=A\otimes\mathcal K$, it follows that $\Cu(A)$ is almost unperforated whenever $A$ is $\Z$-stable. If additionally $A$ is simple, then $\Cu(A)$ has strict comparison by its functionals, namely those induced by lower semicontinuous $2$-quasitraces, by \cite[Proposition 6.2]{ERS:AJM}.  Thus for simple, unital, $\Z$-stable $A$ with $QT(A)=T(A)$, $\Cu(A)$ has strict comparison by bounded traces on $A$ (extended to lower semicontinuous tracial functionals on $A\otimes\mathcal K$).}  Combining these results gives the following set of conditions ensuring that limit traces are weak$^*$-dense, which we will use repeatedly (noting that if one replaces the condition $QT(B_n)=T(B_n)$ by exactness, this reduces exactly to Ozawa's theorem (\cite[Theorem 8]{O:JMSUT})).

\begin{prop}[R\o{}rdam, Ng-Robert]\label{NoSillyTraces}
Let $(B_n)_{n=1}^\infty$ be a sequence of simple, separable, unital and $\mathcal Z$-stable $\mathrm{C}^*$-algebras such that $QT(B_n)=T(B_n)$ for all $n$.  Then $T_\omega(\prod_\omega B_n)$ is weak$^*$-dense in $T(\prod_\omega B_n)$. 
\end{prop}

A key property of the ultraproduct is that one can use an argument variously called ``Kirchberg's $\eps$-test'', the ``diagonal sequence argument'', the ``reindexing argument'', or ``saturation''. In this paper we primarily work with Kirchberg's formulation from \cite{K:Abel}, which is the most suitable for our purposes.

\begin{lemma}[{Kirchberg's $\eps$-test, \cite[Lemma A.1]{K:Abel}}]\label{epstest}
Let $X_1,X_2,\dots$ be a sequence of nonempty sets, and for each $k,n\in\N$, let $f^{(k)}_n:X_n\rightarrow [0,\infty)$ be a function.
Define $f^{(k)}_\omega:\prod_{n=1}^\infty X_n\to[0,\infty]$ by $f^{(k)}_\omega((s_n)_{n=1}^\infty)=\lim_{n\rightarrow\omega}f^{(k)}_n(s_n)$ for $(s_n)\in\prod_{n=1}^\infty X_n$.  Suppose that for all $m\in\N$ and $\eps>0$, there exists $(s_n)_{n=1}^\infty\in \prod_{n=1}^\infty X_n$ with $f^{(k)}_\omega((s_n))<\eps$ for $k=1,\dots,m$.  Then there exists $(t_n)_{n=1}^\infty\in \prod_{n=1}^\infty X_n$ such that $f^{(k)}_\omega((t_n))=0$ for all $k\in\N$.
\end{lemma}

One of our most crucial applications of this argument is the existence of supporting order zero maps given in Lemma \ref{lem:SupportingMap}, for which we record a test for detecting c.p.c.\ order zero maps.
First we set up a way of using the $\eps$-test with contractive maps out of a separable $\mathrm C^*$-algebra.

\begin{lemma}\label{lem:ContractiveTest}
Let $A,B_n$ be $\mathrm C^*$-algebras such that $A$ is separable and unital, and write $B_\omega := \prod_\omega B_n$.
Fix a countable dense $\mathbb Q[i]$-$^*$-subalgebra $A_0$ of $A$.
Let $X_n$ denote the set of $^*$-linear maps $A_0\rightarrow B_n$.
Then there exist functions $g^{(k)}_n:X_n\rightarrow[0,\infty)$ indexed by $n\in\N$ and $k\in I$, for some countable index set $I$, such that for $(\phi_n)_{n=1}^\infty \in \prod_{n=1}^\infty X_n$, $(\phi_n)_{n=1}^\infty$ induces a contractive $^*$-linear map $A_0 \to B_\omega$ (which therefore extends by continuity to a contractive $^*$-linear map $A \to B_\omega$) if and only if $\lim_{n\to\omega} g^{(k)}_n(\phi_n) = 0$ for all $k \in I$.
\end{lemma}

\begin{proof}
Enumerate $A_0$ as $(a_k)_{k=1}^\infty$.
A sequence $(\phi_n)_{n=1}^\infty \in \prod_{n=1}^\infty X_n$, represents a contractive map $A_0 \to B_\omega$ if and only if $\lim_{n\rightarrow\omega} \|\phi_n(a_k)\| \leq \|a_k\|$ for all $k$.
Therefore, the conclusion holds by defining $g^{(k)}_n$ by 
\begin{equation}
 g^{(k)}_n(\phi_n) := \max\{\|\phi_n(a_k)\|-\|a_k\|,0\},\quad\phi_n\in X_n.\qedhere
\end{equation}
\end{proof}

\begin{lemma}\label{orderzerotest}
Let $A,B_n$ be $\mathrm C^*$-algebras such that $A$ is separable and unital, and write $B_\omega := \prod_\omega B_n$.
Fix a countable dense $\mathbb Q[i]$-$^*$-subalgebra $A_0$ of $A$.
Let $X_n$ denote the set of $^*$-linear maps $A_0\rightarrow B_n$.
Then there exist functions $f^{(k)}_n:X_n\rightarrow[0,\infty)$ indexed by $n\in\N$ and $k\in I$, for some countable index set $I$, such that the following holds: a sequence $(\phi_n)_{n=1}^\infty \in \prod_{n=1}^\infty X_n$ induces a contractive $^*$-linear map $A_0 \to B_\omega$ which extends by continuity to a c.p.c.\ order zero map $\phi:A \to B_\omega$ if and only if $\lim_{n\to\omega} f^{(k)}_n(\phi_n) = 0$ for all $k \in I$.
\end{lemma}

\begin{proof}
First take $g^{(k)}_n$ to be a family of functions provided by Lemma \ref{lem:ContractiveTest}, indexed by $k \in \mathbb N$.
For each $m\in\N$, enumerate all elements of $M_m(A_0)_+$ as $(a^{(m)}_k)_{k=1}^\infty$ (note that this sequence is dense in $M_m(A)_+$), and define
\begin{equation}
f^{(m,k)}_n(\phi_n)=\inf_{b\in M_m(B_n)_+}\|\phi_n^{(m)}(a^{(m)}_k)-b\|,\quad \phi_n\in X_n,
\end{equation} where $\phi^{(m)}_n:M_m(A_0)\rightarrow M_m(B_n)$ is the amplification of $\phi_n:A_0\rightarrow B_n$. 
Let $(\phi_n)_{n=1}^\infty\in \prod_{n=1}^\infty X_n$ extend by continuity to a contractive ${}^*$-linear map $A\rightarrow B_\omega$.
Then $(\phi_n)_{n=1}^\infty$ represents a completely positive (and therefore c.p.c.) map if and only if $\lim_{n\rightarrow\omega}f^{(m,k)}_n(\phi_n)=0$ for all $k,m\in\N$.\footnote{As $(\phi_n)_{n=1}^\infty$ induces a contractive map $A\rightarrow B_\omega$, the amplifications $(\phi_n^{(m)})_{n=1}^\infty$ induce continuous maps $M_m(A)\rightarrow M_m(B_\omega)$ for each $m\in\N$. Thus it suffices to test for $m$-positivity on a dense set from $M_m(A)_+$.}

To test for the additional condition of being order zero, enumerate the unit ball of $A_0$ as $(a'_k)_{k=1}^\infty$ (so that this sequence is dense in the unit ball of $A$), and define $h^{(k,l)}_n:X_n\rightarrow[0,\infty)$ by 
\begin{equation}h^{(k,l)}_n(\phi_n):=\|\phi_n(a'_ka'_l)\phi_n(1_A)-\phi_n(a'_k)\phi_n(a'_l)\|,\quad \phi_n\in X_n.
\end{equation}
Then, for a sequence $(\phi_n)_{n=1}^\infty\in\prod_{n=1}^\infty X_n$ which represents a c.p.c.\ map $\phi:A\rightarrow B_\omega$, the order zero identity (\ref{eq:Ord0Ident}) shows that $\phi$ is order zero if and only if $\lim_{n\rightarrow\omega}h^{(k,l)}_n(\phi_n)=0$ for all $k,l\in\N$.  Thus the functions $g^{(k)}_n$, $f^{(m,k)}_n$ and $h^{(k,l)}_n$ provide the required testing functions.
\end{proof}

\begin{remark}
\label{rmk:orderzerotestplus}
The previous lemma is easily strengthened to the following (a form that will be needed in Lemma  \ref{lem:GoodTraceMaps}).

Let $A$ be a separable unital $\mathrm C^*$-algebra, let $A_0$ be a countable dense $\mathbb Q[i]$-$^*$-subalgebra $A_0$ of $A$, and let $\mathcal C$ be a set of $\mathrm C^*$-algebras.  Let $X_n$ denote the set of $^*$-linear maps from $A_0$ to any $\mathrm C^*$-algebra in $\mathcal C$. Then there exist functions $f^{(k)}_n:X_n\rightarrow[0,\infty)$ indexed by $n\in\N$ and $k\in I$, for some countable index set $I$, such that for a sequence $(\phi_n)_{n=1}^\infty\in \prod_{n=1}^\infty X_n$ (with $\phi_n:A_0 \to B_n$ and $B_n \in \mathcal C$), $(\phi_n)_{n=1}^\infty$ induces a contractive map which extends by continuity to a c.p.c.\ order zero map $A \to \prod_\omega B_n$ if and only if $\lim_{n\rightarrow\omega}f^{(k)}_n(\phi_n)=0$ for all $k\in I$. 

Indeed, this version follows from Lemma \ref{orderzerotest} by taking $\bigoplus_{C \in \mathcal C} C$ for each $B_n$ in the lemma.
\end{remark}

\begin{lemma}[{cf.\ \cite[Lemma 2.2]{SWW:arXiv}}]
\label{lem:SupportingMap}
Let $A,B_n$ be unital $\mathrm C^*$-algebras such that $A$ is separable, set $B_\omega := \prod_\omega B_n$, and suppose that $S\subset B_\omega$ is separable and self-adjoint.
Let $\phi:A \to B_\omega\cap S'$ be a c.p.c.\ order zero map.
Then there exists a c.p.c.\ order zero map $\hat\phi:A \to B_\omega\cap S'$ such that
\begin{align}
\label{eq:Supporting}
\phi(ab)&=\hat{\phi}(a)\phi(b)=\phi(a)\hat{\phi}(b),\quad a,b\in A.
\end{align}
If the map $\tau \mapsto d_\tau(\phi(1_A))$ from $T(B_\omega)$ to $[0,1]\subset\mathbb R$ is continuous (with respect to the weak$^*$-topology on $T(B_\omega)$) then we can, in addition, arrange that
\begin{align}
\label{eq:SupportingTrace}
\tau(\hat{\phi}(a)) &= \lim_{m\to\infty} \tau(\phi^{1/m}(a)), \quad a\in A_+,\ \tau \in T_\omega(B_\omega),
\end{align}
where order zero map functional calculus (\ref{OZFC}) is used to interpret $\phi^{1/m}$. In this case, the induced map $\bar{\hat\phi}:A \to B_\omega/J_{B_\omega}$ is a $^*$-homomorphism.
\end{lemma}

A c.p.c.\ order zero map $\hat\phi$ satisfying (\ref{eq:Supporting}) is called a \emph{supporting order zero map} of $\phi$, and throughout,
we will use the notation $\hat{\phi}$ (or $\hat\psi, \hat\pi$, etc.) exclusively to denote a supporting order zero map of $\phi$ (respectively $\psi,\pi$).
Note that the requirement that $\tau\mapsto d_\tau(\phi(1_A))$ is continuous on $T_\omega(B_\omega)$ is certainly necessary for (\ref{eq:SupportingTrace}) to hold.

Supporting order zero maps can be used to recapture the functional calculus for order zero maps from (\ref{OZFC}): if $\phi,\hat{\phi}:A\rightarrow B_\omega$ are c.p.c.\ order zero maps satisfying (\ref{eq:Supporting}), and $f\in C_0((0,1])_+$, then (as noted after \cite[Lemma 2.2]{SWW:arXiv}), $f(\phi)(x)=\hat{\phi}(x)f(\phi(1_A))$ for all $x\in A$.

\begin{proof}[Proof of Lemma \ref{lem:SupportingMap}]
Lift $\phi$ to a representative sequence of $^*$-linear maps $\phi_n:A\rightarrow B_n$.\footnote{When $A$ is nuclear we can use the Choi-Effros lifting theorem to lift $\phi$ to a representative sequence of c.p.c.\ maps, thus (by setting $X_n$ equal to the set of c.p.c.\ maps $A \rightarrow B_n$) avoiding Lemma \ref{lem:ContractiveTest} and the completely positive part of Lemma \ref{orderzerotest}.  The only application of Lemma \ref{lem:SupportingMap} to which this short cut does not apply is Remark \ref{rmk:relSI-simple}, which shows that the definition we give of property (SI) in terms of maps extends the original definition for $\mathrm{C}^*$-algebras.}  We will use the $\eps$-test, with $X_n$ equal to the set of $^*$-linear maps $A\rightarrow B_n$.  For this we fix  dense sequences $(a_i)_{i=1}^\infty$ in the unit ball of $A$ and $(s^{(i)})_{i=1}^\infty$ in $S$, and lift each $s^{(i)}$ to a bounded sequence $(s^{(i)}_n)_{n=1}^\infty$ with each $s^{(i)}_n\in B_n$.  Consider the countable collection of functions $f^{(0,i,j)}_n,f^{(1,i,j)}:X_n\rightarrow [0,\infty)$ given by 
\begin{align}
f^{(0,i,j)}_n(\psi_n)&:=\|\phi_n(a_ia_j) - \psi_n(a_i)\phi_n(a_j)\|\nonumber\\
&\quad+ \|\phi_n(a_ia_j) - \phi_n(a_i)\psi_n(a_j)\|,\quad \psi_n\in X_n, \nonumber\\
f_n^{(1,i,j)}(\psi_n)&:=\|[\psi_n(a_i),s_n^{(j)}]\|,\quad \psi_n\in X_n.
\end{align}
Fix $\eps>0$ and define $\psi:=g_\eps(\phi):A \to B_\omega\cap S'$ (where $g_\eps$ is as in (\ref{eq:gepsDef})) so that $\psi$ is a c.p.c.\ order zero map.
Let $\pi:C_0((0,1])\otimes A \to B_\omega\cap S'$ be the ${}^*$-homomorphism induced by $\phi$ as in Proposition \ref{prop:Ord0Structure}. For any contractions $a,b\in A_+$, as $g_\eps\cdot \mathrm{id}_{(0,1]} \approx_\eps\mathrm{id}_{(0,1]}$, we have
\begin{align}
\notag
\phi(ab) &= \pi(\mathrm{id}_{(0,1]}  \otimes (ab)) \\
\notag
&\approx_\eps \pi\left((g_\eps\cdot \mathrm{id}_{(0,1]} ) \otimes (ab)\right) \\
\notag
&= \pi(g_\eps \otimes a)\pi(\mathrm{id}_{(0,1]}  \otimes b)=\pi(\mathrm{id}_{(0,1]} \otimes a)\pi(g_\eps \otimes b) \\
&= \psi(a)\phi(b)=\phi(a)\psi(b).
\end{align}
 Thus taking any representative sequence $(\psi_n)_{n=1}^\infty\in\prod_{n=1}^\infty X_n$ of $\psi$, we have $\lim_{n\rightarrow\omega}f^{(0,i,j)}_n(\psi_n)\leq\eps$ for all $i,j\in\N$. Further, as $\psi(A)$ commutes with $S$, we also have $\lim_{n\rightarrow\omega}f^{(1,i,j)}_n(\psi_n)=0$ for all $i,j\in\N$.  Using the functions $f^{(0,i,j)}_n,f^{(1,i,j)}$, together with the functions from Lemma \ref{orderzerotest} (with respect to a dense $\mathbb Q[i]$-$^*$-subalgebra $A_0$ of $A$), in Kirchberg's $\eps$-test (Lemma \ref{epstest}), it follows that there exists a c.p.c.\ order zero map $\hat{\phi}:A\rightarrow B_\omega\cap S'$ satisfying (\ref{eq:Supporting}).
 
Now suppose additionally that $\tau \to d_\tau(\phi(1_A))$ from $T(B_\omega)$ to $[0,1]$ is continuous (with respect to the weak$^*$-topology on $T(B_\omega)$). Set
\begin{equation}\label{defgammam} \gamma_m := \max_{\tau \in T(B_\omega)} d_\tau(\phi(1_A))- \tau(\phi(1_A)^{1/m}) \geq 0, \end{equation}
so that $\lim_{m\to\infty}\gamma_m=0$ by Dini's theorem (this is where we use the continuity of $\tau\mapsto d_\tau(\phi(1_A))$).
Consider the additional countable family of functions $h_n^{(m)}:X_n\rightarrow [0,\infty)$, given by
\begin{equation}\label{e2.29} h_n^{(m)}(\psi_n):=\max\left(\max_{\tau_n \in T(B_n)}\big( \tau_n(\psi_n(1_A))-\tau_n(\phi_n(1_A)^{1/m})\big)-\gamma_m,0\right) 
\end{equation}
for $\psi_n\in X_n$. 
For each $\eps>0$ and $\tau\in T(B_\omega)$, 
\begin{equation}
\tau(g_\eps(\phi)(1_A))\leq d_\tau(\phi(1_A))\leq \tau(\phi(1_A)^{1/m})+\gamma_m,\quad m\in\N,
\end{equation}
using (\ref{defgammam}). Thus if $(\psi_n)_{n=1}^\infty$ is an element of $\prod_{n=1}^\infty X_n$ representing $g_\eps(\phi)$, then
\begin{equation}
\lim_{n\rightarrow\omega}h^{(m)}_n(\psi_n)=0,\quad m\in\N.
\end{equation}
Since the maps $g_\eps(\phi)$ are precisely those used in the above, in the case when $\tau\to d_\tau(\phi(1_A))$ is continuous, we can combine the functions $h^{(m)}_n$ with those used in the previous paragraph, so that Kirchberg's $\eps$-test (Lemma \ref{epstest}) provides a c.p.c.\ order zero map $\hat{\phi}:A\rightarrow B_\omega\cap S'$ satisfying (\ref{eq:Supporting})  and represented by a sequence $(\psi_n)_{n=1}^\infty\in\prod_{n=1}^\infty X_n$ with $\lim_{n\rightarrow\omega}h_n^{(m)}(\psi_m)=0$ for all $m\in\N$. It follows that
\begin{equation}
\label{eq:SupportingConde}
\max\left(\max_{\tau \in T_\omega(B_\omega)} \big(\tau(\hat{\phi}(1_A))-\tau(\phi(1_A)^{1/m})\big)-\gamma_m,0\right)=0,\quad m\in\N,\end{equation}
and so
\begin{equation}
\tau(\hat{\phi}(1_A))\leq \tau(\phi(1_A)^{1/m})+\gamma_m,\quad m\in\N,\ \tau\in T_\omega(B_\omega).
\end{equation}
Taking limits, since $\lim_{m\rightarrow\infty}\gamma_m=0$, we have
\begin{equation}
\tau(\hat{\phi}(1_A))\leq\lim_{m\rightarrow\infty}(\tau(\phi(1_A)^{1/m})+\gamma_m)=d_\tau(\phi(1_A)),\quad \tau\in T_\omega(B_\omega).
\end{equation}

Fix $\tau\in T_\omega(B_\omega)$ momentarily. For any $a \in A_+$, using the fact that supporting order zero maps can be used to recapture the functional calculus (as noted in the paragraph preceding this proof), \eqref{eq:Supporting} implies that 
\begin{equation}
\phi^{1/m}(a)=\hat{\phi}(a)(\phi(1_A))^{1/m}=(\hat{\phi}(a))^{1/2}(\phi(1_A))^{1/m}(\hat{\phi}(a))^{1/2}\leq \hat{\phi}(a),
\end{equation}
for each $m\in\N$. Thus, $a \mapsto \tau(\hat\phi(a)) - \lim_{m\to\infty} \tau(\phi^{1/m}(a))$ is a positive functional, and since it evaluates to $0$ at $1_A$, it must be zero.
Consequently, we find that \eqref{eq:SupportingTrace} holds.
The last part of the lemma follows from Proposition \ref{prop:Ord0Homo} since, by Lemma \ref{lem:ProjectionModJ}, $\bar{\hat\phi}(1_A)$ is a projection.
\end{proof}

\begin{remark}
The delicate point in the proof of the previous lemma is that even for limit traces $\tau=\lim_{n\rightarrow\omega}\tau_n$, it is not the case that $d_\tau((a_n)_{n=1}^\infty)=\lim_{n\rightarrow\omega}d_{\tau_n}(a_n)$, hence the need to introduce $\gamma_m$.   Indeed $\lim_{n\to\omega} d_{\tau_n}(a_n)$ can depend on the choice of a representing sequence $(a_n)_{n=1}^\infty$; for example, $(1_{B_n}/n)_{n=1}^\infty$ and $(0)_n$ represent the same element of $B_\omega$, yet $\lim_{n\rightarrow\omega} d_{\tau_n}(1_{B_n}/n)=1$ and $\lim_{n\rightarrow\omega} d_{\tau_n}(0)=0$.
\end{remark}

With a few small exceptions, our remaining applications of Kirchberg's $\eps$-test are more routine, taking $X_n$ to be (finite products) of bounded subsets (primarily the unit ball) of a $\mathrm{C}^*$-algebra $B_n$ and defining the functions $f^{(k)}_n$ through norm and trace conditions; an expository account of how to use the $\eps$-test in this way is in \cite[Section 4.1]{T:MA}. We record three such applications of the reindexing argument here for use in the sequel.

\begin{lemma}
\label{lem:ActAsUnit}
Let $(B_n)_{n=1}^\infty$ be a sequence of $\mathrm C^*$-algebras and set $B_\omega := \prod_\omega B_n$.
Let $S_1,S_2$ be separable self-adjoint subsets of $B_\omega$, and let $T$ be a separable subset of $B_\omega \cap S_1' \cap S_2^\perp$.
Then there exists a contraction $e \in (B_\omega \cap S_1' \cap S_2^\perp)_{+}$ that acts as a unit on $T$, i.e., such that $et=te=t$ for every $t\in T$.
\end{lemma}

\begin{proof}
We use the $\eps$-test with $X_n$ equal to the collection of positive contractions in $B_n$. Since $S_1$ and $S_2$ are separable, there exists a countable collection of functions $f^{(k)}_n:X_n\rightarrow [0,\infty)$, such that a positive contraction $x\in B_\omega$ represented by $(x_n)_{n=1}^\infty\in\prod_{n=1}^\infty X_n$ lies in $B_\omega\cap S_1'\cap S_2^\perp$ if and only if $\lim_{n\rightarrow\omega}f^{(k)}_n(x_n)=0$ for all $k$. Take a dense sequence $(t^{(k)})_{k=1}^\infty$ in $T$, fix representatives $(t^{(k)}_n)_{n=1}^\infty$ of $t^{(k)}$ and define $g^{(k)}_n(x_n):=\max(\|x_nt^{(k)}_n-t^{(k)}_n\|,\|t^{(k)}_nx_n-t^{(k)}_n\|)$. In this way a contraction $x\in B_\omega$ represented by $(x_n)_{n=1}^\infty\in\prod_{n=1}^\infty X_n$ acts as a unit on $T$ if and only if $\lim_{n\rightarrow\omega}g^{(k)}_n(x_n)=0$ for all $k\in\N$.

Given $m\in\N$ and $\eps>0$, use an approximate unit to find a positive contraction $x\in C^*(T)\subset B_\omega\cap S_1'\cap S_2^\perp$ with $\|xt^{(k)}-t^{(k)}\|\leq\eps$ and $\|t^{(k)}x-t^{(k)}\|\leq\eps$ for $k=1,\dots,m$. Then a representative $(x_n)_{n=1}^\infty\in\prod_{n=1}^\infty X_n$ for $x$ will satisfy $\lim_{n\rightarrow\omega}f^{(k)}_n(x_n)=0$ for all $k\in\N$, and $\lim_{n\rightarrow\omega}g^{(k)}_n(x_n)\leq\eps$ for $k=1,\dots,m$.  Thus Kirchberg's $\eps$-test (Lemma \ref{epstest}) provides the required $e$.
\end{proof}

\begin{lemma}\label{NewEpsLemma}
Let $(B_n)_{n=1}^\infty$ be a sequence of $\mathrm C^*$-algebras and set $B_\omega := \prod_\omega B_n$.
Let $S_1,S_2$ be separable self-adjoint subsets of $B_\omega$, and set $C:=B_\omega\cap S_1'\cap S_2^\perp$. Write $C^\sim$ for the unitization of $C$ if $C$ is non-unital, and $C^\sim=C$ otherwise.
\begin{enumerate}
\item Let $h_1,h_2\in C_+$. Then $h_1$ and $h_2$ are unitarily equivalent via a unitary from $C^\sim$ if and only if they are approximately unitarily equivalent, i.e., for any $\eps>0$ there exists a unitary $u\in C^\sim$ with $uh_1u^*\approx_\eps h_2$.
\item Let $a\in C$.  Then there exists a unitary $u\in C^\sim$ with $a=u|a|$ if and only if for each $\eps>0$ there exists a unitary $u\in C^\sim$ with $a\approx_\eps u|a|$.
\item Let $h_1,h_2 \in C_+$. Then $h_1$ and $h_2$ are Murray-von Neumann equivalent if and only if they are approximately Murray-von Neumann equivalent, i.e., for any $\eps>0$ there exists $x\in C$ with $xx^*\approx_\eps h_1$ and $x^*x\approx_\eps h_2$.

\end{enumerate}
\end{lemma}
\begin{proof}
We give the proof when $C$ is non-unital (the unital case is easier). Taking $X_n$ to be the set of elements of $B_n$ of norm at most $2$, as in Lemma \ref{lem:ActAsUnit}, there is a countable collection of functions $f^{(k)}_n:X_n\rightarrow B_n$ with the property that for $x\in B_\omega$ with norm at most $2$ and with the representative $(x_n)_{n=1}^\infty \in \prod_{n=1}^\infty X_n$, we have $x\in B_\omega\cap S_1'\cap S_2^\perp$ if and only if $\lim_{n\rightarrow\omega}f^{(k)}_n(x_n)=0$ for all $k\in\N$.  

(i):  Take representative sequences $(h_{1,n})_{n=1}^\infty$ and $(h_{2,n})_{n=1}^\infty$ from $\prod_{n=1}^\infty B_n$ for $h_1$ and $h_2$ respectively.  For $x\in C$ and $\lambda\in\mathbb C$, $x+\lambda 1_{C^\sim}$ is a unitary if and only if $|\lambda|=1$ and $xx^*+\lambda x^*+\bar{\lambda}x=x^*x+\lambda x^*+\bar{\lambda}x=0$.  For each $\lambda\in\mathbb C$ with $|\lambda|=1$, consider functions $g_n^{(\lambda)}:X_n\rightarrow[0,\infty)$ given by
\begin{align}
g_n^{(\lambda)}(x_n)&:=\max \big(\|x_nx_n^*+\bar{\lambda}x_n+\lambda x_n^*\|,\|x_n^*x_n+\bar{\lambda}x_n+\lambda x_n^*\|,\nonumber\\
&\qquad \|(x_n+\lambda1_{C^\sim})h_{1,n}(x_n+\lambda1_{C^\sim})^*-h_{2,n}\|\big)
\end{align}
and define $g_n:X_n\rightarrow [0,\infty)$ by 
\begin{equation}
g_n(x_n):=\inf_{\lambda\in\mathbb C,\ |\lambda|=1}g_n^{(\lambda)}(x_n).
\end{equation}
Now suppose $h_1$ and $h_2$ are approximately unitarily equivalent in $C^\sim$. Fix $\eps>0$ and take a unitary $u=x+\lambda 1_{C^\sim}$ with $x \in C$ and $\lambda \in \mathbb C$ with $|\lambda|=1$, such that $uh_1u^*\approx_{\eps}h_2$.
Since $\|x\| \leq 2$, $x$ is represented by a sequence $(x_n)_{n=1}^\infty \in \prod_{n=1}^\infty X_n$.
Then $\lim_{n\rightarrow\omega}f^{(k)}_n(x_n)=0$, and $\lim_{n\rightarrow\omega}g^{(\lambda)}_n(x_n)\leq\eps$, so that $\lim_{n\rightarrow\omega}g_n(x_n)\leq\eps$. Thus, by Kirchberg's $\eps$-test (Lemma \ref{epstest}), there exists a sequence $(x_n)_{n=1}^\infty$ in $\prod_{n=1}^\infty X_n$ representing an element $x\in B_\omega$ with $\lim_{n\rightarrow\omega}f^{(k)}_n(x_n)=0$ (i.e., $x\in C$), and $\lim_{n\rightarrow\omega}g_n(x_n)=0$.  For each $n$, pick $\lambda_n\in\mathbb C$ with $|\lambda_n|=1$ that minimizes $g_n^{(\lambda_n)}(x_n)$; then set $\lambda:=\lim_{n\rightarrow\omega}\lambda_n$, so that $|\lambda|=1$.  Then $\lim_{n\rightarrow\omega}g^{(\lambda)}_n(x_n)=0$, so $u=x+\lambda 1_{C^\sim}\in C^\sim$ witnesses the unitary equivalence of $h_1$ and $h_2$.

(ii): This is very similar to (i). Take a representative sequence $(a_n)_{n=1}^\infty$ for $a$ and define functions $g_n^{(\lambda)}:X_n\rightarrow[0,\infty)$ for $|\lambda|=1$ by
\begin{align}
g_n^{(\lambda)}(x_n):=\max&\big(\|x_nx_n^*+\bar{\lambda}x_n+\lambda x_n^*\|,\|x_n^*x_n+\bar{\lambda}x_n+\lambda x_n^*\|,\nonumber\\
&\qquad\|a_n-(x_n+\lambda)|a_n|\|\big)
\end{align}
and $g_n:X_n\to[0,\infty)$ by $g_n(x_n)=\inf\{g_n^{(\lambda)}(x_n):\lambda\in\mathbb C,\ |\lambda|=1\}$ for $x_n\in X_n$. Then, one can argue just as in (i) to see that if for each $\eps>0$ there is a unitary $u\in C^\sim$ with $a\approx_\eps u|a|$, then there is a unitary $u\in C^\sim$ with $a=u|a|$.

(iii): This is a straightforward application of the $\eps$-test as it does not make reference to the unitization. We may assume that $h_1$ and $h_2$ are contractions. Given representative sequences $(h_{1,n})_{n=1}^\infty$ and $(h_{2,n})_{n=1}^\infty$ in $\prod_{n=1}^\infty X_n$ for $h_1$ and $h_2$ respectively, the result follows from applying Kirchberg's $\eps$-test (Lemma \ref{epstest}) to the functions $f^{(k)}_n$ as above which test for membership of $C$, and the functions $g_n:X_n\rightarrow[0,\infty)$ given by $g_n(x_n):=\max(\|x_nx_n^*-h_{1,n}\|,\|x_n^*x_n-h_{2,n}\|)$.
\end{proof}

\begin{lemma}
\label{lem:LargeTraceSubordinate}
Let $(B_n)_{n=1}^\infty$ be a sequence of unital $\mathrm C^*$-algebras with $T(B_n)\neq\emptyset$ for each $n\in\N$. Write $B_\omega := \prod_\omega B_n$.  Let $S_0$ be a countable self-adjoint subset of $(B_\omega)_+$ and let $T$ be a separable self-adjoint subset of $B_\omega$.
If $x,f \in (B_\omega \cap S_0'\cap T')_+$ are contractions with $x\vartriangleleft f$ and with the property that for all $a \in S_0$ there exists $\gamma_a \geq 0$ such that $\tau(af^m) \geq \gamma_a$ for all $m\in \N,\ \tau \in T_\omega(B_\omega)$, then there exists a contraction $f' \in (B_\omega \cap S_0'\cap T')_+$ such that $x\vartriangleleft f' \vartriangleleft f$ and $\tau(a(f')^m) \geq \gamma_a$ for all $m\in\N,\ \tau \in T_\omega(B_\omega)$, and $a\in S_0$.

If each $B_n$ is simple, separable, $\mathcal Z$-stable and has $QT(B_n)=T(B_n)$, then the above statement holds with $T(B_\omega)$ in place of $T_\omega(B_\omega)$.
\end{lemma}

\begin{proof}
We use the $\eps$-test with $X_n$ equal to the set of positive contractions in $B_n$. By countability of $S_0$ and separability of $T$, there are functions $g^{(k)}_n:X_n\rightarrow [0,\infty)$ such that a positive contraction $y\in B_\omega$ represented by $(y_n)_{n=1}^\infty\in\prod_{n=1}^\infty X_n$ satisfies $y\in B_\omega \cap S_0'\cap T'$ and $x\vartriangleleft y$ if and only if $\lim_{n\rightarrow\omega}g^{(k)}_n(y_n)=0$ for all $k\in\N$.
Enumerate $S_0$ as $\{a^{(1)},a^{(2)},\dots\}$.
Fix a representative sequence $(f_n)_{n=1}^\infty$ of $f$ and for each $k$, fix a representative sequence $(a^{(k)}_n)_{n=1}^\infty$ of $a^{(k)}$.  Define 
\begin{align}
\notag
h^{(0)}_n(y_n)&:=\|f_ny_n-y_n\| \quad \text{and} \\
h^{(m,k)}_n(y_n)&:=\max\big(\gamma_{a^{(k)}}-\min_{\tau\in T(B_n)}\tau(a^{(k)}_ny^m_n),0\big),\quad y_n\in X_n.
\end{align}
In this way a positive contraction $y\in B_\omega$ represented by $(y_n)_{n=1}^\infty\in\prod_{n=1}^\infty X_n$ has $y\vartriangleleft f$ and $\tau(ay^m)\geq\gamma_a$ for all $a\in S_0$, $m\in\N$ and $\tau\in T_\omega(B_\omega)$ if and only if $\lim_{n\rightarrow\omega}h^{(0)}(y_n)=0$ and $\lim_{n\rightarrow\omega}h^{(m,k)}_n(y_n)=0$ for all $m,k\in\N$.

Fix $\eps>0$, find $r\in\mathbb N$ such that $t^r-t^{r+1}<\eps$ for all $t \in [0,1]$, and set $y=f^r$, a positive contraction in $B_\omega\cap S_0'\cap T'$ with $x\vartriangleleft y$. Since \begin{equation} fy = f^{r+1} \approx_\eps f^r = y, \end{equation}
a representative sequence $(y_n)_{n=1}^\infty$ in $\prod_{n=1}^\infty X_n$ for $y$ will satisfy the estimate $\lim_{n\rightarrow\omega}h^{(0)}(y_n)\leq\eps$. As $\tau(ay^m) = \tau(af^{rm}) \geq \gamma_a$ for all $a\in S_0$, $m\in\N$ and $\tau\in T_\omega(B_\omega)$, we have $\lim_{n\rightarrow\omega}h^{(m,k)}_n(y_n)=0$.  Kirchberg's $\eps$-test (Lemma \ref{epstest}) then provides the required positive contraction $f'$.

The last statement follows by the density of $T_\omega(B_\omega)$ in $T(B_\omega)$ from Proposition \ref{NoSillyTraces}.
\end{proof}

Via a special case of the previous result, we adapt the crucial relative commutant surjectivity results from \cite[Section 4]{KR:Crelle} to handle the hereditary subalgebras of relative commutant sequence algebras used in this paper.

\begin{lemma}
\label{lem:RelCommSurjectivity}
Let $(B_n)_{n=1}^\infty$ be a sequence of separable, unital $\mathrm C^*$-algebras with $T(B_n)\neq\emptyset$ for each $n\in\N$. Set $B_\omega:=\prod_\omega B_n$.
Let $A$ be a separable, unital $\mathrm C^*$-algebra and let $\pi:A\rightarrow B_\omega$ be a  c.p.c.\ order zero map such that $\pi(1_A)$ is full and the induced map $\bar{\pi}:A \to B_\omega/J_{B_\omega}$ is a $^*$-homomorphism.
Define $C:=B_\omega \cap \pi(A)' \cap \{1_{B_\omega}-\pi(1_A)\}^\perp$.
Let $S \subseteq C$ be a countable self-adjoint subset and let $\bar{S}$ denote the image of $S$ in $B_\omega/J_{B_\omega}$.
\begin{enumerate}
\item[(i)] Then the image of $C\cap S'$ in $B_\omega/J_{B_\omega}$ is precisely
\begin{eqnarray}
\lefteqn{ \bar{\pi}(1_A)\left((B_\omega/J_{B_\omega}) \cap \bar{\pi}(A)' \cap
\bar{S}'\right) } \nonumber \\
& = & (B_\omega/J_{B_\omega})\cap
\bar{\pi}(A)'\cap \bar{S}'\cap \{1_{B_\omega/J_{B_\omega}}-\bar{\pi}(1_A)\}^\perp,
\end{eqnarray}
a $\mathrm C^*$-subalgebra of $B_\omega / J_{B_\omega}$ with unit $\bar{\pi}(1_A)$.
\item[(ii)] Let $\tau\in T_\omega(B_\omega)$ be a limit trace and $a\in A_+$ and form the tracial functional $\rho:=\tau(\pi(a)\cdot)$ on $C$.
Then  $\|\rho\|=\tau(\pi(a))$.  If each $B_n$ is additionally  simple, $\Z$-stable and has $QT(B_n)=T(B_n)$, then this holds for all traces $\tau\in T(B_\omega)$.
\end{enumerate}
\end{lemma}

\begin{proof}
(i): Kirchberg's $\eps$-test (Lemma \ref{epstest}) can be used to see that $J_{B_\omega}$ is a $\sigma$-ideal of $B_\omega$ (in the case that $B_n=B$ for some fixed $B$, this is done in \cite[Proposition 4.6]{KR:Crelle}, and the general case is proved in exactly the same fashion). Then \cite[Proposition 4.5(iii)]{KR:Crelle} shows that the image of $B_\omega \cap \pi(A)' \cap S'$ in $B_\omega/J_{B_\omega}$ is $(B_\omega/J_{B_\omega}) \cap \bar{\pi}(A)' \cap \bar{S}'$. 

The image of a hereditary $\mathrm{C}^*$-subalgebra $A_0$ of a $\mathrm{C}^*$-algebra $A$ under a surjective $^*$-homomorphism $\theta:A\rightarrow D$ is again hereditary.\footnote{For positive elements $x\in A_0$ and $d\in D$ with $d\leq \theta(x)$, we have $d\in \overline{\theta(x)\theta(A)\theta(x)}\subseteq \theta(A_0)$, so $\theta(A_0)$ is a hereditary subalgebra of $D$.} Since $C\cap S'$ is a hereditary subalgebra of $B_\omega \cap \pi(A)' \cap S'$, the image of $C\cap S'$ in $B_\omega/J_{B_\omega}$ is a hereditary subalgebra of $(B_\omega/J_{B_\omega}) \cap \bar{\pi}(A)' \cap \bar{S}'$, contained in $\bar{\pi}(1_A)\left((B_\omega/J_{B_\omega}) \cap \bar{\pi}(A)' \cap \bar{S}'\right)$.  We need to show that $\bar{\pi}(1_A)$ lies in this image.

As $\bar{\pi}$ is a $^*$-homomorphism, $\tau(\pi(1_A)^m)=\tau(\pi(1_A))$ for all $m\in\N$ and all $\tau\in T_\omega(B_\omega)$. Taking $T:=S\cup \pi(A)$, $S_0:=\{\pi(1_A)\}$, $f:=\pi(1_A)$ and $x:=0$ in Lemma \ref{lem:LargeTraceSubordinate}, we obtain a positive contraction $e\in B_\omega\cap \pi(A)'\cap S'$ with $e\vartriangleleft \pi(1_A)$ (so that $e\in C\cap S'$ and $\pi(1_A)-e \geq 0$) and $\tau(\pi(1_A))\geq \tau(e)\geq\tau(\pi(1_A))$ for all $\tau\in T_\omega(B_\omega)$.  Then, for $\tau\in T_\omega(B_\omega)$, $\tau(\pi(1_A)-e)=0$, so by \eqref{eq:JBeqDef2},  $e\equiv\pi(1_A)$ modulo $J_{B_\omega}$, and thus $\bar\pi(1_A)$ is equal to the image of $e\in C\cap S'$ in $B_\omega/J_{B_\omega}$.

(ii): Given $\tau\in T_\omega(B_\omega)$ and $a \in A_+$, form $\rho:=\tau(\pi(a)\cdot)$.
Then we certainly have $\|\rho\|\leq\tau(\pi(a))$. Conversely, for the positive contraction $e$ as above, since $\pi(1_A)\equiv e$ modulo $J_{B_\omega}$,
\begin{equation}
|\tau(\pi(a)(\pi(1_A)-e))|\leq \|\pi(a)\|\tau(\pi(1_A)-e)=0
\end{equation}
so that $\rho(e)=\tau(\pi(a)e)=\tau(\pi(a)\pi(1_A))=\tau(\pi(a))$, as $\bar{\pi}$ is a $^*$-homomorphism.  Thus $\|\rho\|=\tau(\pi(a))$.  

When each $B_n$ is simple $\Z$-stable and has $QT(B_n)=T(B_n)$, the weak$^*$-density of $T_\omega(B_\omega)$ in $T(B_\omega)$ from Proposition \ref{NoSillyTraces} gives the last statement.\end{proof}

A unital $\mathrm C^*$-algebra is said to have \emph{stable rank one} if its invertible elements form a dense subset (\cite[Definition 1.4]{R:PLMS}).  As in \cite{L:Book} stable rank one passes to ultraproducts via the following well-known lemma (a special case of \cite[Theorem 5]{P:JOT}), which we record for later use.

\begin{lemma}
\label{lem:PolarDecomp}
Let $A$ be a unital $\mathrm C^*$-algebra and let $a\in A$.
Then $a$ is a norm-limit of invertible elements if and only if, for any $\eps>0$, there exists a unitary $u \in A$ such that
 $a \approx_\eps u|a|$.
\end{lemma}

\begin{lemma}[{cf.\ \cite[Lemma 19.2.2(1)]{L:Book}}]
\label{lem:sr1Ultrapower}
Let $(B_n)_{n=1}^\infty$ be a sequence of unital $\mathrm C^*$-algebras with stable rank one.
Then $B_\omega:=\prod_\omega B_n$ also has stable rank one.
\end{lemma}

Ultraproducts of $\Z$-stable $\mathrm{C}^*$-algebras satisfy a number of known regularity properties; the lemma below collects some of these for later use.
\begin{lemma}
\label{lem:Zfacts}
Let $(B_n)_{n=1}^\infty$ be a sequence of separable, $\Z$-stable $\mathrm C^*$-algebras and set $B_\omega:=\prod_\omega B_n$. Then:
\begin{enumerate}[(i)]
\item\label{lem:Zfacts2} If $S\subset B_\omega$ is separable, then there exist isomorphisms $\phi_n:B_n \to B_n \otimes \Z$ such that the induced isomorphism $B_\omega \to \prod_\omega (B_n \otimes \Z)$ takes $x\in S$ to $x \otimes 1_{\mathcal Z} \in (\prod_\omega B_n) \otimes \Z\subset \prod_\omega (B_n\otimes\Z)$. 
\item\label{lem:Zfacts1} If $S \subset B_\omega$ is separable and self-adjoint, and each $B_n$ is unital, then $B_\omega\cap S'$ contains a unital copy of $\Z$.
\item\label{lem:Zfacts4} If each $B_n$ is simple, unital, and finite, then $B_\omega$ has stable rank one.
\item\label{lem:Zfacts6} If $S \subset B_\omega$ is separable and self-adjoint, and $b \in (B_\omega \cap S')_+$, then for any $n\in \N$ there exists $c\in (B_\omega\cap S')_+$ with $c \leq b$ such that $n[c] \leq [b] \leq (n+1)[c]$ in $W(B_\omega\cap S')$.
\end{enumerate}
\end{lemma}

\begin{proof} For (i), the strongly self-absorbing property of $\Z$ (see \cite[Theorem 2.2]{TW:TAMS} and \cite[Theorem 7.2.2]{R:Book}) provides isomorphisms $\psi_n:B_n\rightarrow B_n\otimes\Z$ which are approximately unitarily equivalent to $\mathrm{id}_{B_n}\otimes 1_\Z$.  Take a dense sequence $(s^{(k)})_{k=1}^\infty$ in $S$ and lift each $s^{(k)}$ to a bounded sequence $(s_n^{(k)})_{n=1}^\infty$.  Take unitaries $u_n$ in $B_n$ (or its unitisation if $B_n$ is non-unital) such that $\|\Ad(u_n)\circ\psi_n(s^{(k)}_n)-s^{(k)}_n\otimes 1_{\Z}\|<\frac{1}{n}$ for $k=1,\dots,n$.  Then $\phi_n=\Ad(u_n)\circ\psi_n$ provides the required isomorphisms.

 Then (\ref{lem:Zfacts1}) follows immediately since the canonical embedding of $\Z$ into $\prod_\omega(B_n\otimes\Z)$ commutes with the image of $B_\omega$ in $\prod_\omega(B_n\otimes\Z)$.
Part (\ref{lem:Zfacts4}) is a consequence of \cite[Theorem 6.7]{R:IJM} and Lemma \ref{lem:sr1Ultrapower}.

For (\ref{lem:Zfacts6}), let $\phi:\mathcal Z \to B_\omega \cap (S \cup \{b\})'$ be a unital embedding by (\ref{lem:Zfacts1}).
By \cite[Lemma 4.2]{R:IJM}, let $e_n \in \mathcal Z_+$ be a positive contraction such that $n[e_n] \leq [1_{\mathcal Z}] \leq (n+1)[e_n]$.
Then define $c:=b\phi(e_n)$.
\end{proof}

Finally in this section, we record that strict comparison passes to ultraproducts, and in fact, only limit traces are needed there to test comparison. Results of this nature are known and standard, though we are not aware of this particular form in the literature.

\begin{lemma}
\label{lem:StrictCompLimTraces}
Let $(B_n)_{n=1}^\infty$ be a sequence of unital $\mathrm C^*$-algebras and set $B_\omega:=\prod_\omega B_n$.
Suppose that each $B_n$ has strict comparison with respect to bounded traces (as in Definition \ref{defn:StrictComp}).
Then $B_\omega$ has strict comparison of positive elements with respect to limit traces, in the following sense:
If $a,b \in M_k(B_\omega)_+$ for some $k\in\mathbb N$ satisfy $d_\tau(a) < d_\tau(b)$ for all $\tau$ in the weak$^{*}$-closure of $T_\omega(B_\omega)$, then $a \preceq b$.
\end{lemma}

\begin{proof}
Take contractions $a,b\in M_k(B_\omega)_+$ with $d_\tau(a)<d_\tau(b)$ for all $\tau\in \overline{T_\omega(B_\omega)}$ and fix $\eps>0$.  As $\tau\mapsto\tau(g_\eps(a))$ is continuous on the compact set $\overline{T_\omega(B_\omega)}$ and $\tau\mapsto d_\tau(b)$ is the increasing pointwise supremum of the continuous functions $\tau\mapsto\tau(g_\delta(b))$, one can find $\delta>0$ such that $\tau(g_\eps(a))<\tau(g_\delta(b))$ for all $\tau\in \overline{T_\omega(B_\omega)}$.
Given lifts $(a_n)_{n=1}^\infty$ and $(b_n)_{n=1}^\infty$ of $a$ and $b$ to sequences of positive contractions, we have
\begin{equation}
I:= \{n\in\mathbb N:\mbox{for all }\tau\in T(B_n),\ \tau(g_\eps(a_n))<\tau(g_\delta(b_n))\}\in\omega. 
\end{equation}
For $n\in I$, 
since $d_\tau((a_n-\eps)_+) \leq \tau(g_\eps(a_n))$, 
strict comparison in $B_n$ provides $v_n\in M_k(B_n)$ with $\|v_n^*g_\delta(b_n)v_n-(a_n-\eps)_+\|<\frac{1}{n}$. Set $w_n:=g_\delta(b_n)^{1/2}v_n$ for $n\in I$ and $w_n=1$ for $n\notin I$ so that $(w_n)_{n=1}^\infty$ is bounded and therefore represents an element $w\in B_\omega$. Thus $g_{\delta/2}(b)w=w$ and $w^*w=(a-\eps)_+$.  Therefore $(a-\eps)_+\preceq g_{\delta/2}(b)$. Since $\eps$ was arbitrary, $a\preceq b$.
\end{proof}

\clearpage\section{A \texorpdfstring{$2\times 2$}{2 x 2} matrix trick}\label{sec:Matrix}

\noindent
In Lemma \ref{lem:CombinedUnitaryEquiv} of this section, we provide our order zero version of a $2\times 2$ matrix trick due to Connes, used to reduce unitary equivalence of maps to unitary equivalence of positive elements in a suitable relative commutant sequence algebra.
We set this up in such a way to handle both the stably finite and purely infinite cases simultaneously.  We need some additional facts about approximating elements by invertibles, which we will use again in Section \ref{sec:TotallyFullClass}.  The strategy for approximating elements with invertibles by examining zero divisors (used in the next two results) originates in \cite{R:JFA2}; there, R\o{}rdam showed that simple, unital, stably finite, UHF-stable $\mathrm C^*$-algebras have stable rank one.  Our initial preparatory lemma uses Robert's version of these methods from \cite{Robert:arXiv}.

\begin{lemma}
\label{lem:Robertsr1}
Let $(B_n)_{n=1}^\infty$ be a sequence of unital, $\mathcal Z$-stable $\mathrm C^*$-algebras and set $B_\omega :=\prod_\omega B_n$.
Let $S \subset B_\omega$ be separable and self-adjoint, and let $d \in (B_\omega \cap S')_+$ be a contraction.
Suppose that $x,f \in C:=B_\omega \cap S' \cap \{1_{B_\omega}-d\}^\perp$ are such that $xf=fx=0$, $f \geq 0$ and $f$ is full in $C$.
Then $x$ is approximated by invertibles in the unitization of $C$.
\end{lemma}

\begin{proof}
By Lemma \ref{lem:ActAsUnit} (with $T:=\{x^*x,xx^*\},S_1:=S,S_2:=\{1_{B_\omega}-d,f\}$), we may find a contraction $e \in C_+$ such that $xx^*,x^*x\vartriangleleft e$ and $ef=0$ so that by taking polar decomposition, $ex=xe=x$.
We may find a separable subalgebra $C_0$ of $C$ containing $x,e$, and $f$, such that $f$ is full in $C_0$.
Since $B_\omega \cap S' \cap C_0' \cap \{d\}'$ contains a unital copy of $\mathcal Z$ (Lemma \ref{lem:Zfacts}(\ref{lem:Zfacts1})), there exists a map $\phi:C_0 \otimes \mathcal Z \to C$ which restricts to the identity on $C_0 \otimes 1_{\mathcal Z} = C_0$; namely, $\phi(c \otimes z) := c\psi(z)$ for $c \in C_0,z \in \mathcal Z$, where $\psi:\mathcal Z \to B_\omega \cap S' \cap C_0' \cap \{d\}'$ is a fixed unital ${}^*$-homomorphism.
By \cite[Lemma 2.1]{Robert:arXiv}, $x\otimes 1_{\mathcal Z}$ is a product of two nilpotent operators in $C_0\otimes\mathcal Z$, and so pushing forward along $\phi$, $x$ is a product of two nilpotent operators in $C$.  For a nilpotent operator $y\in C$ and $\eps>0$, the operator $y+\eps1_{C^\sim}$ is invertible in $C^\sim$ (with inverse $-\sum_{k=1}^N(-\eps)^{-k}y^{k-1}$, where $N\in\N$ satisfies $y^N=0$).  Therefore $x$ lies in the closure of the invertible operators in the unitization of $C$.
\end{proof}

The next lemma improves the previous one, by replacing the positive element $f$ by an element $s$ not required to be positive.
Note that if $s$ is a full element in a $\mathrm C^*$-algebra $C$, then so too is $|s|$, since upon taking the polar decomposition $s=u|s|$, for any $\eps>0$, we have $ug_\eps(|s|)\in C$ and $s=\lim_{\eps\to 0} ug_\eps(|s|)|s|$.

\begin{lemma}
\label{lem:GapInvertibles2}
Let $(B_n)_{n=1}^\infty$ be a sequence of unital, $\mathcal Z$-stable $\mathrm C^*$-algebras and set $B_\omega :=\prod_\omega B_n$.
Let $S \subset B_\omega$ be separable and self-adjoint, and let $d \in (B_\omega \cap S')_+$ be a contraction.
Suppose that $x,s \in C:=B_\omega \cap S' \cap \{1_{B_\omega}-d\}^\perp$ are such that $xs=sx=0$ and $s$ is full in $C$.
Then $x$ is approximated by invertibles in the unitization of $C$.
\end{lemma}

\begin{proof}
By Lemma \ref{lem:ActAsUnit}, let $e\in C_+$ be such that $s\vartriangleleft e$. 
By Lemma \ref{lem:Zfacts}(\ref{lem:Zfacts1}), there exists a unital ${}^*$-homomorphism $\psi:\mathcal Z \to B_\omega \cap S' \cap \{x,x^*,s,s^*,e\}'$.
Let $z_1,z_2 \in \mathcal Z_+$ be nonzero orthogonal elements.
Define $s' := \psi(z_1)s$ and $f:=\psi(z_2)|s|$.
Since $\mathcal Z$ is simple and $s$ and $|s|$ are full in $C$, both $f$ and $s'$ are full in $C$.
Also $f \geq 0$ and $fs' = s'f = 0$, so by Lemma \ref{lem:Robertsr1} (with $s'$ in place of $x$), $s'$ is approximated by invertibles in $C^\sim$.

By Lemma \ref{lem:PolarDecomp} and Lemma \ref{NewEpsLemma}(ii) $s'$ has a unitary polar decomposition $s'=u|s'|$ in $C^\sim$. 
Since 
\begin{equation}
s'x=\psi(z_1)sx=0
\end{equation}
and similarly $xs'=0$, it follows that
\begin{equation}
(xu)|s'| = |s'|(xu) = 0,
\end{equation}
so that by Lemma \ref{lem:Robertsr1} (with $|s'|$ in place of $f$ and $xu$ in place of $x$, noting that $C^\sim$ multiplies $C$ into $C$ so that $xu\in C$), $xu$ is approximated by invertibles in $C^\sim$.
Finally, since $u$ is unitary, it follows that $x$ is itself approximated by invertibles in $C^\sim$.
\end{proof}

Now we can present our $2\times 2$ matrix trick, in the context of order zero maps into ultrapowers.
\begin{lemma}\label{lem:CombinedUnitaryEquiv}
Let $A$ be a separable, unital $\mathrm C^*$-algebra, let $(B_n)_{n=1}^\infty$ be a sequence of separable $\mathrm C^*$-algebras and write $B_\omega=\prod_\omega B_n$. Let $\phi_1,\phi_2:A\rightarrow B_\omega$ be c.p.c.\ order zero maps and $\hat{\phi}_1,\hat{\phi}_2:A\rightarrow B_\omega$ be supporting order zero maps (as in \eqref{eq:Supporting}).
 Suppose that either:
\begin{enumerate}[(i)]
\item $B_\omega$ has stable rank one; or
\item each $B_n$ is a Kirchberg algebra, and $\phi_i(1_A)$ is totally full in $B_\omega$ for $i=1,2$.
\end{enumerate}
Let $\pi:A\rightarrow M_2(B_\omega)$ be given by
\begin{equation}
\pi(a) := \begin{pmatrix} \hat\phi_1(a) & 0 \\ 0 & \hat\phi_2(a) \end{pmatrix}, \quad a\in A.
\end{equation}
If
\begin{equation}
\begin{pmatrix}\phi_1(1_A)&0\\0&0\end{pmatrix}\text{ and }\begin{pmatrix}0&0\\0&\phi_2(1_A)\end{pmatrix}
\end{equation}
are unitarily equivalent in the unitization of $C:=M_2(B_\omega)\cap \pi(A)'\cap \{1_{M_2(B_\omega)}-\pi(1_A)\}^\perp$, then $\phi_1$ and $\phi_2$ are unitarily equivalent.
\end{lemma}
\begin{proof}
Let
\begin{equation}
u=\begin{pmatrix} u_{11}&u_{12} \\ u_{21}&u_{22} \end{pmatrix} \in C^\sim
\end{equation}
be a unitary such that
\begin{equation}\label{eq:UnitaryDecompSetup}
u\begin{pmatrix} \phi_1(1_A) & 0\\0 & 0 \end{pmatrix}=  \begin{pmatrix} 0 & 0\\0 & \phi_2(1_A) \end{pmatrix}u.
\end{equation}

Then, since $u$ is unitary,
\begin{equation} 
\label{eq:UnitaryDecompSum1}
u_{11}^*u_{11} + u_{21}^*u_{21} = 1_{B_\omega}. 
\end{equation}
By \eqref{eq:UnitaryDecompSetup},
\begin{equation}\label{e3.8}
u_{11}\phi_1(1_A)u_{11}^* = 0, \text{ and } u_{21}\phi_1(1_A)= \phi_2(1_A)u_{21}.
\end{equation}
It follows that $u_{11}\phi_1(1_A)^{1/2}=0$, and so
\begin{equation}
\label{eq:UnitaryDecompConjugate1}
\phi_1(1_A) \stackrel{\eqref{eq:UnitaryDecompSum1}}= u_{21}^*u_{21}\phi_1(1_A) \stackrel{(\ref{e3.8})}= u_{21}^*\phi_2(1_A)u_{21}.
\end{equation}
Symmetrically, we obtain (among other relations)
\begin{equation}
\label{eq:UnitaryDecompAlmostUnitary}
\phi_2(1_A) = u_{21}u_{21}^*\phi_2(1_A).
\end{equation}
As $u$ commutes with $\pi(A)$, one checks that
\begin{equation}
\label{eq:UnitaryDecompConjugate2}
u_{21}\hat{\phi}_1(a) = \hat{\phi}_2(a)u_{21},\quad a\in A.
\end{equation}

We claim that $u_{21}^*\phi_2(1_A)=\phi_1(1_A)u_{21}^*$ is approximable by invertibles in $B_\omega$.  In case (i) this is immediate.  In case (ii), take $\eps>0$, so that $u_{21}^*(\phi_2(1_A)-\eps)_+$ has a nonzero right zero-divisor in $B_\omega$ of the form $e_r=h(\phi_2(1_A))$ for some nonzero positive $h\in C_0((0,1])$ supported on $[0,\eps]$ (this is nonzero by the assumption that $\phi_2(1_A)$ is totally full).  Also 
\begin{equation}
u_{21}^*(\phi_2(1_A)-\eps)_+=(\phi_1(1_A)-\eps)_+u_{21}^*,
\end{equation} so that $e_l=h(\phi_1(1_A))$ is a nonzero (as $\phi_1(1_A)$ is totally full) left zero-divisor of $u_{21}^*(\phi_2(1_A)-\eps)_{+}$ in $B_\omega$.  Since $B_\omega$ is simple and purely infinite (see \cite[Remark 2.4]{K:Abel}, for example), there exists $t\in B_\omega$ with $te_lt^*=e_r$.  Thus $s:=te_l^{1/2}\neq 0$ satisfies $su_{21}^*(\phi_2(1_A)-\eps)_+=0$ and
\begin{equation}
u_{21}^*(\phi_2(1_A)-\eps)_+ss^*(\phi_2(1_A)-\eps)_+u_{21}=0,
\end{equation}
so $u_{21}^*(\phi_2(1_A)-\eps)_+s=0$.  By Lemma \ref{lem:GapInvertibles2} (with $S:=\{1_{B_\omega}\}$ and $d:=1_{B_\omega}$), $u_{21}^*(\phi_2(1_A)-\eps)_+$ is approximated by invertibles in $B_\omega$.
Since $\eps>0$ was arbitrary in the last calculation, $u_{21}^*\phi_2(1_A)$ is also approximated by invertibles in $B_\omega$.
This establishes our claim.

By Lemma \ref{lem:PolarDecomp} and Lemma \ref{NewEpsLemma}(ii), there is a unitary $w\in B_\omega$ with
\begin{equation}
u_{21}^*\phi_2(1_A)= w\,\left|u_{21}^*\phi_2(1_A)\right|.
\end{equation}
By \eqref{eq:UnitaryDecompAlmostUnitary}, 
\begin{equation}
\left|u_{21}^*\phi_2(1_A)\right| = (\phi_2(1_A)u_{21}u_{21}^*\phi_2(1_A))^{1/2} = \phi_2(1_A).
\end{equation}
Thus,
\begin{equation}
\label{eq:UnitaryDecompPDecomp}
u_{21}^*\phi_2(1_A) = w\phi_2(1_A).
\end{equation}
It follows that $u_{21}^*bu_{21}=wbw^*$ for all $b \in \mathrm{her}(\phi_2(1_A))$, and therefore,
\begin{align}
u_{21}^*\phi_2(a)u_{21} &= w\phi_2(a)w^*, \quad a\in A.
\label{eq:UnitaryDecompNew.1}
\end{align}

We now compute, for $a\in A$,
\begin{align}
\phi_1(a) &\stackrel{\eqref{eq:Supporting}}{=} \hat{\phi}_1(a)\phi_1(1_A)\nonumber \\
&\stackrel{\eqref{eq:UnitaryDecompConjugate1}}{=} \hat{\phi}_1(a)u_{21}^*\phi_2(1_A)u_{21}\nonumber \\
&\stackrel{\eqref{eq:UnitaryDecompConjugate2}}{=} u_{21}^*\hat{\phi}_2(a)\phi_2(1_A)u_{21}\nonumber \\
&\stackrel{\eqref{eq:Supporting}}{=} u_{21}^*\phi_2(a)u_{21}\nonumber\\
&\stackrel{\eqref{eq:UnitaryDecompNew.1}}{=} w\phi_2(a)w^*,
\end{align}
as required.
\end{proof}

\clearpage\section{Ultrapowers of trivial {$\mathrm{W}^*$}-bundles}
\label{sec:Bundles}

\noindent
In this section we examine trivial $\mathrm{W}^*$-bundles, which arise from strict completions of $\Z$-stable $\mathrm C^*$-algebras.
Our objective is to obtain structural results for relative commutant sequence algebras of these bundles, such as strict comparison of positive elements, which we later lift back to the $\mathrm C^*$-algebra level.

\bigskip

\subsection{\sc Continuous $\mathrm{W}^{*}$-bundles}\hfill  \\

\label{sec:Bundles1}

\noindent
Tracial continuous $\mathrm{W}^*$-bundles were introduced by Ozawa in \cite[Section 5]{O:JMSUT} as follows.\footnote{
In \cite{O:JMSUT}, Ozawa works with a standing assumption that the base space $K$ is metrizable; we have generalized the definition to the non-metrizable case, since we will need this level of generality to treat ultraproducts of $\mathrm W^*$-bundles in Section \ref{sec:Bundles2}. In particular the ultracoproduct $\prod_\omega K_n$ of metrisable compact spaces $(K_n)_{n=1}^\infty$ need not be metrisable.} Let $K$ be a compact Hausdorff space.

\begin{defn}[{\cite[Section 5]{O:JMSUT}}]
\label{def:W*bundle}
A $\mathrm C^*$-algebra $\M$ is a \emph{tracial continuous $\mathrm{W}^*$-bundle over $K$} if:
\begin{itemize}
\item $C(K)$ is contained in the centre of $\M$;
\item there is a faithful tracial conditional expectation $E:\M\rightarrow C(K)$, defining a norm $\|x\|_{2,u}:=\|E(x^*x)\|^{1/2}$, called the uniform $2$-norm;
\item the unit ball of $\M$ is complete in $\|\cdot\|_{2,u}$.
\end{itemize}
For each $\lambda\in K$, let $\tau_\lambda$ be the trace $\mathrm{ev}_\lambda\circ E$ on $\M$, with associated GNS-representation $\pi_\lambda$ and associated $2$-norm $\|x\|_{2,\lambda} := \mathrm{ev}_\lambda \circ E(x^*x)^{1/2}$.
The $\mathrm W^*$-bundle axioms imply that each fibre $\pi_\lambda(\M)$ is a finite von Neumann algebra: $\pi_\lambda(\M)''=\pi_\lambda(\M)$ (this is \cite[Theorem 11]{O:JMSUT} --- the proof does not use the standing metrizability assumption in \cite{O:JMSUT}).  We say that $\M$ has \emph{factorial fibres} when each $\pi_\lambda(\M)$ is a factor.
\end{defn}

The following formula is obvious but quite useful:
\begin{equation}
\label{eq:W*Norm}
\|x\|_{2,u} = \max_{\lambda \in K} \|x\|_{2,\lambda}.
\end{equation}

Completing $\M$ in $\|\cdot\|_{2,u}$ gives a Hilbert $C(K)$-module $L^2(\M)$. The \emph{strict} topology on $\M$ is given by the action of $\M$ as adjointable operators on this module; precisely a net $(x_i)$ in $\M$ converges strictly to $x\in \M$ if and only if $\|x_i\eta-x\eta\|_{L^2(\M)},\|x_i^*\eta-x^*\eta\|_{L^2(\M)}\rightarrow 0$ for all $\eta\in L^2(\M)$.\footnote{As set out in \cite[Chapter 8]{Lan:Book}, there are two natural ``strict'' topologies on the adjointable operators on a Hilbert module; the one defined here, and the strict topology arising from the adjointable operators as the multiplier algebra of the compact operators on the module.  These topologies agree on the unit ball of the adjointable operators (\cite[Proposition 8.1]{Lan:Book}). Thus, via the Kaplansky density theorem for the strict topology (\cite[Proposition 1.4]{Lan:Book}), one obtains the same closure of a unital $\mathrm{C}^*$-subalgebra of the adjointable operators in either strict topology. Using this, all statements below are equally valid for either strict topology.} Just as the strong$^*$-topology on the unit ball of a II$_1$ factor in standard form is induced by the $2$-norm, it is easy to see that on the unit ball of $\M$, the $\|\cdot\|_{2,u}$-topology and the strict topology agree.
Consequently, a $\mathrm W^*$-bundle $\M$ is $\|\cdot\|_{2,u}$-separable if and only if it is separable in the strict topology;\footnote{A strictly-dense sequence will automatically be $\|\cdot\|_{2,u}$-dense; conversely, if $\M$ is $\|\cdot\|_{2,u}$-separable then so are the operator norm balls of any radius (as the $\|\cdot\|_{2,u}$-topology is metrizable). This implies that these balls, and therefore $\M$ (as the countable union of the balls of radius $n$) is strictly separable.} this is the natural notion of separability for $\mathrm W^*$-bundles. Note also that if $\mathcal M$ is strictly separable, then $C(K)$ is norm-separable, so that $K$ is metrizable.

Tracial continuous $\mathrm W^*$-bundles (hereafter referred to as $\mathrm{W}^*$-bundles) form a category, with morphisms defined as follows.
Given $\mathrm{W}^*$-bundles $\M$ and $\mathcal N$ over $K$ and $L$ respectively with conditional expectations $E_\M$ and $E_{\mathcal N}$, a morphism $\theta:\M\rightarrow\mathcal N$ in the category of $\mathrm{W}^*$-bundles is given by a $^*$-homomorphism from $\M$ into $\mathcal N$ (also denoted $\theta$) and a continuous map $\sigma:L\rightarrow K$, inducing the $^*$-homomorphism $\tilde{\sigma}:C(K)\rightarrow C(L)$ such that
\begin{equation}
\xymatrix{{\M}\ar[r]^{\theta}\ar[d]_{E_{\M}}&{\mathcal N}\ar[d]^{E_{\mathcal N}}\\C(K)\ar[r]_{\tilde{\sigma}}&C(L)}
\end{equation}
commutes.

There are three key constructions of $\mathrm{W}^*$-bundles used in this paper: trivial bundles, bundles arising from the strict completion of a stably finite $\mathrm C^*$-algebra, and ultraproducts.
We recall the first two of these now from \cite{O:JMSUT}.

\begin{example}
Let $\mathcal N$ be a finite von Neumann algebra with faithful trace $\tau_{\mathcal N}$, used to define $\|\cdot\|_2$, and let $K$ be a compact Hausdorff space.
The \emph{trivial $\mathrm{W}^*$-bundle over $K$ with fibre $\mathcal N$} is the $\mathrm C^*$-algebra $C_\sigma(K,\mathcal N)$ of norm bounded and $\|\cdot\|_2$-continuous functions from $K$ to $\mathcal N$, with the canonical conditional expectation onto $C(K)$, $E(f)(\lambda)=\tau_{\mathcal N}(f(\lambda))$ for $f\in C_\sigma(K,\mathcal N)$ and $\lambda\in K$. Note that by a partition of unity argument, $\mathrm C^*(\mathcal N,C_\sigma(K,\mathbb C1_{\mathcal N}))$ is $\|\cdot\|_{2,u}$-dense in $C_\sigma(K,\mathcal N)$.
\end{example}
\begin{example}[{Ozawa}]\label{StrictCompletion}
Let $B$ be a separable, unital $\mathrm C^*$-algebra for which $T(B)\neq \emptyset$. Let $Tr:B^{**}_f\rightarrow Z(B^{**}_f)$ be the centre valued trace on the finite part of the bidual of $B$, and let $\overline{B}^{\mathrm{st}}$ be the closure of $B$ in $B^{**}_f$ with respect to the strict topology on the Hilbert $Z(B^{**}_f)$-module $(B^{**}_f,Tr)$.  

Ozawa maps the bounded Borel functions on $\partial_e T(B)$ into $Z(B^{**}_f)$ and uses this to describe $\overline{B}^{\mathrm{st}}$ and $\overline{B}^{\mathrm{st}}\cap Z(B^{**}_f)$. When the extreme boundary $\partial_e T(B)$ is compact, this work canonically identifies $\overline{B}^{\mathrm{st}}\cap Z(B^{**}_f)$ with $C(\partial_e T(B))$ and the center valued trace $Tr$ restricts to a conditional expectation $E:\overline{B}^{\mathrm{st}}\rightarrow C(\partial_e T(B))$ (\cite[Theorem 3]{O:JMSUT}), giving $\overline{B}^{\mathrm{st}}$ the structure of a  $\mathrm{W}^*$-bundle over $\partial_e T(B)$, which is determined on the strictly dense subalgebra $B$ by
\begin{equation}\label{StrictCompletion:E}
E(b)(\tau)=\tau(b),\quad b\in B,\ \tau\in\partial_e T(B).
\end{equation}

For $x \in B \subseteq \overline{B}^{\mathrm{st}}$, we have $\|x\|_{2,u}=\sup\{\tau(x^*x)^{1/2}:\tau\in T(B)\}$.
Alternatively to this strict completion picture,
$\overline{B}^{\mathrm{st}}$ can be thought of as the $\mathrm{C}^*$-algebra obtained by adjoining to $B$ the limit points of all norm-bounded, $\|\cdot\|_{2,u}$-Cauchy sequences; see \cite[Page 351-2]{O:JMSUT}.\footnote{Formally, as in \cite{O:JMSUT}, one considers the quotient $\mathrm{C}^*$-algebra of the norm bounded  $\|\cdot\|_{2,u}$-Cauchy sequences, modulo the ideal of  $\|\cdot\|_{2,u}$-null sequences.}

The fibre at $\tau \in \partial_e T(B)$ is given by $\pi_\tau(B)''$ (\cite[Theorem 3]{O:JMSUT}), where $\pi_\tau$ is the GNS-representation of $B$ induced by $\tau$. In particular, if $B$ is additionally nuclear and has no finite dimensional quotients, then all these fibres are copies of the hyperfinite II$_1$ factor $\R$.
\end{example}

\bigskip

\subsection{\sc Tensor products, ultraproducts and McDuff bundles}\hfill  \\
\label{sec:Bundles2}

\noindent
Here we set out the theory of tensor products and ultraproducts of $\mathrm W^*$-bundles and use these to discuss McDuff bundles and Ozawa's trivialization theorem.
\begin{defn}
Given $\mathrm W^*$-bundles $\M$ and $\mathcal N$ over spaces $K$ and $L$ respectively with conditional expectations $E_\M$ and $E_{\mathcal N}$, the continuous $\mathrm W^*$-bundle $\M\vnotimes\mathcal N$ over $K\times L$ is defined as follows. Set $E:=E_\M\otimes E_{\mathcal N}:\M\otimes\mathcal N\rightarrow C(K)\otimes C(L)\cong C(K\times L)$ be the conditional expectation defined on the minimal tensor product of $\M$ and $\mathcal N$. This is faithful (as can be deduced from Kirchberg's slice lemma, \cite[Lemma 4.1.9]{R:Book}) and tracial and so induces a uniform $2$-norm.   We obtain $\M\vnotimes \mathcal N$ by taking the strict completion of $\M\otimes \mathcal N$ in the Hilbert $C(K\times L)$-module associated to $(\M\otimes\mathcal N,E)$. 
\end{defn}

Note that when $\M$ and $\mathcal N$ are finite von Neumann algebras with faithful traces, which we can view as $\mathrm{W}^*$-bundles over the one point space, then the $\mathrm{W}^*$-bundle tensor product described above agrees with the classical von Neumann tensor product, justifying our notation.  Note too that one obtains the trivial bundle $C_\sigma(K,\M)$ as the tensor product of $C(K)$ (viewed as a trivial bundle over $K$ with fibres $\mathbb C$) and $\M$ viewed as a bundle over the one point space.
We also wish to note that the tensor product operation commutes with taking the uniform $2$-closure of a $\mathrm{C}^*$-algebra. To do this, we need the following folklore proposition; we supply a proof as we have not found a reference.
\begin{prop}
Let $A$ and $B$ be unital, stably finite $\mathrm C^*$-algebras with extreme tracial boundaries $\partial_e T(A)$ and $\partial_e T(B)$ respectively.  Then every extremal trace on $A\otimes B$ is of the form $\tau_A\otimes \tau_B$ for $\tau_A\in\partial_e T(A)$ and $\tau_B\in\partial_e T(B)$. This leads to a canonical homeomorphism $\partial_e T(A)\times \partial_eT(B)\cong \partial_e T(A\otimes B)$.
\end{prop}
\begin{proof}
Firstly recall that a trace $\tau$ on a unital $\mathrm C^*$-algebra $C$ is extremal if and only if $\pi_\tau(C)''$ is a factor.  Take an extremal trace $\tau$ on $A\otimes B$, and let $\tau_A$ and $\tau_B$ denote the traces on $A$ and $B$ obtained by restricting $\tau$ to $A\otimes 1_B\cong A$ and $1_A\otimes B\cong B$ respectively.  Then $\pi_\tau(A\otimes 1_B)''$ must be a factor, as any nontrivial central element would also commute with $\pi_\tau(1_A\otimes B)$. It is easy to check (using \cite[Proposition 3.3.7]{Ped:book}, for example) that $\pi_\tau(A\otimes 1_B)''\cong \pi_{\tau_A}(A)''$, and hence $\pi_{\tau_A}(A)''$ is also a factor. Thus $\tau_A$ is extremal. For $a\in A$ and a contraction $b\in B_+$ with $0<\tau_B(b)<1$, we have
\begin{equation}
\tau_A(a)=\tau_B(b)\frac{\tau(a\otimes b)}{\tau_B(b)}+(1-\tau_B(b))\frac{\tau(a\otimes (1_B-b))}{1-\tau_B(b)}.
\end{equation}
Since $\tau_A$ is extremal, $\tau_A(a)\tau_B(b)=\tau(a\otimes b)$.  Therefore $\tau=\tau_A\otimes\tau_B$.  In this way we see that $\partial_e T(A\otimes B)=\partial_e T(A)\times\partial_e T(B)$, as a set. Working with elementary tensors, this identification is also a homeomorphism with respect to the weak$^*$-topologies.
\end{proof}

\begin{prop}\label{TensorBundle}
Let $A$ and $B$ be simple, separable, unital, stably finite $\mathrm C^*$-algebras with compact extreme tracial boundaries $\partial_e T(A)$ and $\partial_e T(B)$ respectively.  Then $\overline{A\otimes B}^{\mathrm{st}}\cong \overline{A}^{\mathrm{st}}\vnotimes\overline{B}^{\mathrm{st}}$, as $\mathrm{W}^*$-bundles over $\partial_e T(A)\times \partial_e T(B)\cong \partial_e T(A\otimes B)$.
\end{prop}
\begin{proof}
Identify $\partial_e T(A\otimes B)$ with $\partial_e T(A)\times \partial_e T(B)$ as in the previous proposition. By construction $A\otimes B$ is strictly dense in $\overline{A}^{\mathrm{st}}\vnotimes \overline{B}^{\mathrm{st}}$ and the expectations $E_A:\overline{A}^{\mathrm{st}}\rightarrow C(\partial_e T(A))$ and $E_B:\overline{B}^{\mathrm{st}}\rightarrow C(\partial_e T(B))$ satisfy
\begin{align}
E_A(a)(\tau_A)&=\tau_A(a),\quad a\in A,\ \tau_A\in\partial_e T(A),\\
E_B(b)(\tau_B)&=\tau_B(b),\quad b\in B,\ \tau_B\in\partial_e T(B).
\end{align}
Thus $E_A\otimes E_B$ satisfies the defining relation (\ref{StrictCompletion:E}) for the expectation on the bundle $\overline{A\otimes B}^{\mathrm{st}}$ from Example \ref{StrictCompletion}.
\end{proof}

\begin{defn}\label{W*Ultrapower}
Given a sequence of $\mathrm{W}^*$-bundles $(\M_n)_{n=1}^\infty$ over a sequence of compact Hausdorff spaces $(K_n)_{n=1}^\infty$, with uniform $2$-norms $\|\cdot\|_{2,u}^{(n)}$, and a free ultrafilter $\omega$ on $\N$, define the underlying $\mathrm{C}^{*}$-algebra of the $\mathrm{W}^{*}$-bundle ultraproduct to be  
\begin{equation}
\prod^\omega \M_n:=\prod_{n=1}^\infty \M_n/\big\{(x_n)_{n=1}^\infty\in\prod_{n=1}^\infty \M_n:\lim_{n\rightarrow\omega}\|x_n\|_{2,u}^{(n)}=0 \big\}.
\end{equation}
In addition to the operator norm, the ultraproduct inherits a uniform $2$-norm,  $\|\cdot\|^\omega_{2,u}$ defined by $\|(x_n)_{n=1}^\infty\|^\omega_{2,u} := \lim_{n\rightarrow\omega} \|x_n\|_{2,u}^{(n)}$.
\end{defn}

\begin{remark}
The definition above extends the usual definition of ultraproducts of finite von Neumann algebras with fixed faithful tracial states, regarding these as $\mathrm{W}^*$-bundles over the one-point space.
\end{remark}

The ultraproduct construction of Definition \ref{W*Ultrapower} remains in the category of  $\mathrm{W}^*$-bundles. For a sequence $(K_n)_{n=1}^\infty$ of compact metrizable spaces, let $\sum_\omega K_n$ denote the ultracoproduct of $(K_n)_{n=1}^\infty$, i.e., the compact Hausdorff space satisfying $C(\sum_\omega K_n) = \prod_\omega C(K_n)$ (this has also been described, for instance, in \cite{Bankston:JSL}).
We shall often denote this ultracoproduct by $K^\omega$, particularly when $K_n=K$ for all $n$. The set theoretic ultraproduct $\prod_\omega K_n$ (defined to be $\prod_{n=1}^\infty K_n$ modulo the equivalence relation given by $(\lambda_n)_{n=1}^\infty\sim (\mu_n)_{n=1}^\infty$ if and only if $\{n:\lambda_n=\mu_n\}\in\omega$) is identified with the subset of $\sum_\omega K_n$ consisting of those characters $\lambda\in (\prod_\omega C(K))^*$ defined using a sequence $(\lambda_n)_{n=1}^\infty$ from $\prod_{n=1}^\infty K_n$ by $g(\lambda)=\lim_{n\rightarrow\omega}g_n(\lambda_n)$, where $(g_n)_{n=1}^\infty$ represents $g\in \prod_\omega C(K_n)$.  If $g\in \prod_\omega C(K_n)$ is nonzero, and is represented by $(g_n)_{n=1}^\infty$, we can find a sequence $(\lambda_n)_{n=1}^\infty$ in $\prod_{n=1}^\infty K_n$ with $\lim_{n\rightarrow\omega}|g_n(\lambda_n)|=\|g\|$. In this way, $\prod_\omega K_n$ is dense in $\sum_\omega K_n$.

\begin{prop}\label{UltraproductBundles}
Given a sequence $(\M_n)_{n=1}^\infty$ of $\mathrm{W}^*$-bundles over $K_n$ and a free ultrafilter $\omega$ on $\N$, the ultraproduct
$\prod^\omega \M_n$ is a $\mathrm{W}^*$-bundle over $\sum_\omega K_n$, via the conditional expectation $E^\omega:\prod^\omega \M_n \to \prod_\omega C(K_n)$ induced by the maps $E_n:\M_n \to C(K_n)$.
Indeed, we have
\begin{equation} \|E^\omega(x^*x)\|^{1/2} = \|x\|^\omega_{2,u},\quad x\in \prod^\omega\M_n. \end{equation}
\end{prop}
The proof is essentially the same as the classical proof that the ultrapower of a finite von Neumann algebra is a finite von Neumann algebra; see \cite[Theorem A.3.5]{SS:Book}.
\begin{proof}
Given a bounded sequence $(x_n)_{n=1}^\infty$ with $x_n\in \M_n$, one has 
\begin{align}
\lim_{n\rightarrow \omega}\|E_n(x_n)\|=\lim_{n\rightarrow\omega}\|E_n(x_n^*)E_n(x_n)\|^{1/2}&\leq\lim_{n\rightarrow\omega} \|E_n(x_n^*x_n)\|^{1/2}\nonumber\\
&= \lim_{n\rightarrow\omega}\|x_n\|_{2,u}^{(n)}\leq\lim_{n\rightarrow\omega}\|x_n\|.
\end{align}
In particular, when $x_n=E_n(x_n)\in C(K_n)\subset Z(\M_n)$ for each $n$, we have that $\lim_{n\rightarrow\omega}\|x_n\|=\lim_{n\rightarrow\omega}\|x_n\|_{2,u}^{(n)}$ so that the $\mathrm{C^*}$-ultraproduct $\prod_\omega C(K_n)$ canonically embeds into $Z(\prod^\omega\M_n)$.  The inequality $\lim_{n\rightarrow\omega}\|E_n(x_n)\|\leq\lim_{n\rightarrow\omega} \|x_n\|_{2,u}^{(n)}$ shows that $E^\omega$ is well defined. 
The map $E^\omega$ is clearly a tracial conditional expectation onto $\prod_\omega C(K_n)$, and as $\lim_{n\rightarrow\omega} \|E_n(x_n^*x_n)\|^{1/2}= \lim_{n\rightarrow\omega} \|x_n\|_{2,u}^{(n)}$, $E^\omega$ is faithful.

It remains to verify $\|\cdot\|_{2,u}^\omega$-completeness of the unit ball. Suppose that $(x_i)_{i=1}^\infty$ is a $\|\cdot\|^\omega_{2,u}$-Cauchy sequence of contractions.
Passing to a subsequence we may assume that $\|x_i-x_j\|_{2,u}^\omega < 1/j$ for all $i > j$.
We can inductively lift $x_1,x_2,\dots$ to sequences of contractions $(x_{1,n})_{n=1}^\infty,(x_{2,n})_{n=1}^\infty,\dots$ so that $\|x_{i,n}-x_{j,n}\|^{(n)}_{2,u}<1/j$ for all $n\in \mathbb N$ and all $i > j$. Indeed, if lifts have been obtained for $x_1,\dots,x_i$, then start with any lift $(\hat{x}_{i+1,n})_{n=1}^\infty$ of $x_{i+1}$ to a sequence of contractions. Then $I_{i+1}:=\{n\in\N:\|\hat{x}_{i+1,n}-x_{j,n}\|^{(n)}_{2,u}<1/j,\ \forall j=1,\dots,i\}\in\omega$. Defining
\begin{equation}
x_{i+1,n}=\begin{cases}\hat{x}_{i+1,n},&n\in I_{i+1};\\x_{i,n},&n\notin I_{i+1},\end{cases}
\end{equation}
we obtain the desired lift.  Thus, for each $n$ as the unit ball of $\M_n$ is $\|\cdot\|_{2,u}^{(n)}$ complete, there is some $y_n$ in the unit ball of $\M_n$ such that $\|y_n-x_{n,j}\|_{2,u}^{(n)}\leq 1/j$ for all $j$.  Thus, with $y$ the element of $\prod^\omega \M_n$ represented by $(y_n)_{n=1}^\infty$, $\|y-x_j\|_{2,u}^\omega\leq 1/j$ for all $j$, proving the required completeness.
\end{proof}

A key ingredient in the application of $\mathrm{W}^*$-techniques to the study of the structure of simple, nuclear $\mathrm C^*$-algebras is the surjectivity result \cite[Theorem 3.3]{KR:Crelle} concerning the canonical map between central sequence algebras, and the extension of this result in \cite{MS:DMJ} to relative commutant sequence algebras. This is crucial in arguments where properties of the von Neumann algebraic (or $\mathrm W^*$-bundle) central sequence algebras are transferred to the $\mathrm{C}^*$-algebraic level. The next lemma provides the version of this fact we need in this paper, in the setting of sequences of algebras.

\begin{lemma}
\label{lem:CentralSurjectivity}
Let $(B_n)_{n=1}^\infty$ be a sequence of separable, unital	 $\mathrm C^*$-algebras such that $\partial_e T(B_n)$ is nonempty and compact for each $n$.  Let $\M_n:=\overline{B_n}^{\mathrm{st}}$ be the $\mathrm{W}^*$-bundle associated to $B_n$ by Ozawa as in Example \ref{StrictCompletion}.  Write $B_\omega:=\prod_\omega B_n$, $\M^\omega:=\prod^\omega \M_n$ and let $J_{B_\omega}$ be the trace-kernel ideal from (\ref{eq:JBdef}). Then we have the canonical identification
\begin{equation}\label{lem:CentralSurjectivity.Easy}
B_\omega/J_{B_\omega}\cong\M^\omega.
\end{equation}

Further, suppose we are given a separable, unital $\mathrm C^*$-algebra $A$ and a  c.p.c.\ order zero map $\pi:A\rightarrow B_\omega$ such that $\pi(1_A)$ is full and the induced map $\bar{\pi}:A \to B_\omega/J_{B_\omega}=\M^\omega$ is a $^*$-homomorphism. Define $C:=B_\omega \cap \pi(A)' \cap \{1_{B_\omega}-\pi(1_A)\}^\perp$. Let $S \subset C$ be countable and self-adjoint. Let $\bar{S}$ denote the image of $S$ in $\M^\omega$. Then the image of $C\cap S'$ in $\M^\omega=B_\omega/J_{B_\omega}$ is precisely 
\begin{equation}
\bar{\pi}(1_A)\left((B_\omega/J_{B_\omega}) \cap \bar{\pi}(A)' \cap \bar{S}'\right)=\M^\omega\cap \bar{\pi}(A)'\cap \bar{S}'\cap \{1_{\M^\omega}-\bar{\pi}(1_A)\}^\perp.
\end{equation}
\end{lemma}

\begin{proof} The inclusions $B_n\rightarrow\M_n$ induce a well defined $^*$-homomorphism $\Theta:B_\omega\rightarrow \M^\omega$, with kernel $J_{B_\omega}=\{(b_n)_{n=1}^\infty\in B_\omega:\lim_{n\rightarrow\omega}(\|b_n\|_{2,u}^{(n)})^2=0\}$.  Given a sequence $(x_n)_{n=1}^\infty$ of contractions in $\M_n$ representing a contraction $x$ in $\M^\omega$, applying the strict topology version of Kaplansky's density theorem (see \cite[Proposition 1.4]{Lan:Book}) to $B_n\subset \overline{B_n}^{\mathrm{st}}=\M_n$, and using that the $\|\cdot\|_{2,u}^{(n)}$-topology agrees with the strict topology on the unit ball of $\M_n$,\footnote{Or using the alternative definition of $\overline{B_n}^{\mathrm{st}}$ in terms of adjoining limit points of norm-bounded, $\|\cdot\|_{2,u}^{(n)}$-Cauchy sequences.}
we obtain a contraction $y_n\in B_n$ with $\|y_n-x_n\|_{2,u}^{(n)}\leq\frac{1}{n}$.  In this way $(y_n)_{n=1}^\infty$ also represents $x$, and hence $\Theta$ is surjective, proving (\ref{lem:CentralSurjectivity.Easy})

The second part follows by Lemma \ref{lem:RelCommSurjectivity} (i).
\end{proof}

In \cite{M:PLMS}, McDuff characterized those separably-acting II$_1$ factors $\M$ which absorb the hyperfinite II$_1$ factor $\R$ tensorially, as those for which $\R$ embeds in $\M^\omega\cap\M'$ (or equivalently those for which some matrix algebra $M_k$ for $k\geq 2$ embeds into $\M^\omega\cap\M'$). These factors, now called the \emph{McDuff factors}, played a crucial role in Connes' work on injective factors.  The analogous characterization holds in the category of $\mathrm W^*$-bundles. Given a tracial continuous $\mathrm{W}^*$-bundle $\M$, note that the canonical embedding $\M\hookrightarrow \M^\omega$ is a $\mathrm{W}^*$-bundle embedding, enabling us to form the central sequence tracial $\mathrm{W}^*$-bundle $\M^\omega\cap\M'$. 

\begin{prop}\label{DefMcDuff}
Let $\M$ be a strictly-separable $\mathrm{W}^*$-bundle.  The following are equivalent:
\begin{enumerate}
\item[$(1)$] $\M\cong \M\vnotimes\R$ as $\mathrm W^*$-bundles;
\item[$(2)$] $\R$ embeds unitally into $\M^\omega\cap\M'$;
\item[$(3)$] there exists $k\geq 2$ such that $M_k$ embeds unitally into $\M^\omega\cap\M'$;
\item[$(3')$] $M_k$ embeds unitally into $\M^\omega\cap\M'$ for all $k\geq 2$;
\item[$(4)$] there exists $k\geq 2$ such that for any $\|\cdot\|_{2,u}^\omega$-separable self-adjoint subset $S$ of $\M^\omega$, $M_k$ embeds unitally into $\M^\omega\cap S'$.
\item[$(4')$] for any $\|\cdot\|_{2,u}^\omega$-separable self-adjoint subset $S$ of $\M^\omega$ and any $k\geq 2$, $M_k$ embeds unitally into $\M^\omega\cap S'$.
\end{enumerate}
\end{prop}
\begin{defn}
We call strictly separable $\mathrm{W}^*$-bundles satisfying the conditions of Proposition \ref{DefMcDuff} \emph{McDuff}.
\end{defn}

\begin{proof}[Proof of Proposition \ref{DefMcDuff}]
The implication $(1)\Longrightarrow (3')$ follows from the fact that there exist unital embeddings $M_k\rightarrow\R^\omega\cap\R'$. The implications $(3)\Longrightarrow (4)$ and $(3')\Longrightarrow(4')$ follow from a standard reindexing argument.\footnote{Given a $\|\cdot\|_{2,u}^\omega$-dense sequence $(s^{(m)})_{m=1}^\infty$ in $S$, lift each $s^{(m)}$ to a uniformly bounded representative sequence $(s^{(m)}_n)_{n=1}^\infty$ from $\M$.  Then, given a unital embedding $\theta:M_k\rightarrow \M^\omega\cap\M'$, lift it to a sequence of u.c.p.\ maps $\theta_n:M_k\rightarrow\M$. Fix a finite generating set of unitaries $F$ for $M_k$.   For each $n$, there exists some $l_n$ such that $\|[\theta_{l_n}(u),s^{(m)}_n]\|_{2,u}\leq\frac{1}{n}$ for $m\leq n$ and $\|\theta_{l_n}(u)\theta_{l_n}(u^*)-1_{\M}\|_{2,u}<\frac{1}{n}$ for $u\in F$.  Then the map induced by $(\theta_{l_n})_{n=1}^\infty$ gives a unital embedding $M_k\hookrightarrow \M^\omega \cap S'$.}  
The implication $(4)\Longrightarrow (2)$ is proven by using $(4)$ to inductively construct compatible unital embeddings $M_k^{\otimes m}\rightarrow \M^\omega\cap \M'$ for each $m$, which gives rise to a unital embedding $\theta:M_{k^\infty}\rightarrow\M^\omega\cap \M'$.  By uniqueness of the trace on $M_{k^\infty}$ this embedding has $\|\theta(x)-\theta(y)\|_{2,u}^\omega=\|x-y\|_{2,\tau_{M_k{^\infty}}}$, and hence, using that the unit ball of $\M$ is $\|\cdot\|_{2,u}^\omega$-complete, it extends to the required unital embedding $\R\hookrightarrow \M^\omega\cap \M'$.

The implication $(2)\Longrightarrow(1)$ is proven using the standard techniques for tensorial absorption of strongly self-absorbing algebras, cf.\ \cite[Section 7.2]{R:Book} as follows.\footnote{The language of continuous model theory puts this implication, McDuff's original theorem for II$_1$ factors, and the tensorial absorption results for strongly self-absorbing algebras in the same framework, cf.\ \cite[Lemma 6.2]{F:ICM}.} 

We claim that there is a sequence of unitaries $(u_n)_{n=1}^\infty$ in $(\M\vnotimes\R)^\omega\cap (\M\otimes 1_\R)'$ with
\begin{equation}
\lim_{n \to \infty} \inf \{\|u_nxu_n^*-z\|^\omega_{2,u} \mid z \in ((\M\otimes 1_\R)^\omega)^1\}=0,\quad x\in (\M\vnotimes\R)^1.
\end{equation}
Assuming this claim is true, fix $\|\cdot\|_{2,u}$-dense sequences $(x_n)_{n=1}^\infty$ and $(y_n)_{n=1}^\infty$ of the unit balls $\M^1$ and $(\M\vnotimes\R)^1$ respectively. We can then inductively find unitaries $(v_n)_{n=1}^\infty$ in $\M\vnotimes\R$ and elements $z_{j,n}$ in $\M^1$ such that
\begin{align}
\|v_n^*(v_{n-1}^*\cdots v_1^*y_jv_1\cdots v_{n-1})v_n-z_{j,n}\otimes 1_\R\|_{2,u}&\leq\frac{1}{n},\quad j\leq n\notag\\
\|[v_n,(x_j\otimes 1_\R)]\|_{2,u}&<2^{-n},\quad j\leq n\notag\\
\|[v_n,(z_{j,m}\otimes 1_\R)]\|_{2,u}&<2^{-n},\quad j\leq n,\ m\leq n-1.\label{DefMcDuff.E1:SAW}
\end{align}
Then, for each $j\in\N$,  $(v_1\cdots v_n(x_j\otimes 1_{\R})v_n^*\cdots v_1^*)_{n=j}^\infty$ is a $\|\cdot\|_{2,u}$-Cauchy sequence in the unit ball of $(\M\vnotimes \R)$, and thus converges to some element of $(\M\vnotimes \R)^1$, which we denote $\psi(x_j)$.
This defines a function $\psi:\{x_1,x_2,\dots\} \to (\M\vnotimes \R)^1$ which is uniformly continuous in $\|\cdot\|_{2,u}$.
Since $(\M\vnotimes \R)^1$ is $\|\cdot\|_{2,u}$-complete, $\psi$ extends to a continuous function on $\M^1$, and then in turn to a bounded linear map on $\M$, also denoted $\psi$, which is an injective (since it is $\|\cdot\|_{2,u}$-isometric) ${}^*$-homomorphism (here we use the fact that multiplication is jointly $\|\cdot\|_{2,u}$-continuous on bounded sets).   The first and third conditions of (\ref{DefMcDuff.E1:SAW}) ensure that $\psi$ is surjective.
Moreover, for each $j$, we have
\begin{equation}
E_{\M\vnotimes\R} \circ \psi (x_j) = \lim_{n\to\infty} (E_\M \otimes \tau_{\R})(v_1\cdots v_n(x_j \otimes 1_\R)v_n^*\cdots v_1^*) = E_\M(x_j),
\end{equation}
and by $\|\cdot\|_{2,u}$-continuity, it follows that $E_{\M\vnotimes\R} \circ \psi = E_\M$ on the unit ball, and hence on all of $\M$ by linearity, so that $\psi$ is an isomorphism of $\mathrm W^*$-bundles.

Now, to prove the claim, let $\theta:\R\hookrightarrow \M^\omega\cap \M'$ be the unital embedding given by (2), which we view as taking values in $(\M\vnotimes1_\R)^\omega \subset (\M\vnotimes\R)^\omega$. Let $\beta:\R\rightarrow (\M\vnotimes\R)^\omega$ be the second factor embedding, so that $\theta(\R)$ and $\beta(\R)$ commute.
By the following argument, we see $\theta\otimes\beta$ defines a $\|\cdot\|_2$-$\|\cdot\|^\omega_{2,u}$-isometric ${}^*$-homomorphism $\R\vnotimes\R \to (\M\vnotimes\R)^\omega \cap (\M\vnotimes1_\R)'$.
Identify $M_{2^\infty}$ with a strong$^*$-dense $\mathrm C^*$-subalgebra of $\R$, so that the unit ball of $M_{2^\infty}$ is $\|\cdot\|_2$-dense in the unit ball of $\R$.
Then by nuclearity of $M_{2^\infty}$, we have a ${}^*$-homomorphism
\begin{equation} (\theta|_{M_{2^\infty}} \otimes \beta|_{M_{2^\infty}}):M_{2^\infty} \otimes M_{2^\infty} \to (\M\vnotimes\R)^\omega. 
\end{equation}
By uniqueness of the trace on $M_{2^\infty}\otimes M_{2^\infty}$, this map is $\|\cdot\|_2$-$\|\cdot\|^\omega_{2,u}$-isometric, so its restriction to the unit ball of $M_{2^\infty} \otimes M_{2^\infty}$ extends to a map $(\R\vnotimes\R)^1 \to ((\M\vnotimes\R)^\omega)^1$; moreover, the image of this map commutes with $\M\vnotimes1_\R$.
By linearity, this defines a map $\theta\otimes\beta:\R\vnotimes\R \to (\M\vnotimes\R)^\omega\cap (\M\vnotimes1_\R)'$.

As $\R$ has approximately inner (half-)flip, there is a sequence of unitaries $(u_n)_{n=1}^\infty$ in $(\theta\otimes\beta)(\R\vnotimes\R) \subset (\M\vnotimes\R)^\omega \cap (\M\otimes 1_{\R})'$ with $\|u_n(1_\M\otimes y)u_n^*-\theta(y)\|^\omega_{2,u}\rightarrow 0$ for $y\in\R$.  Thus
\begin{equation}
\label{eq:McDuffaif}
\|u_n(x\otimes y)u_n^*-(x\otimes 1_\R)\theta(y)\|^\omega_{2,u}\rightarrow 0,\quad x\in \M,\ y\in\R.
\end{equation}

The assignment $x \otimes y \to (x\otimes 1_\R)\theta(y)$ defines a ${}^*$-homomorphism $\eta:\M \otimes M_{2^\infty} \to (\M\vnotimes1_{\R})^\omega$, and this map is $\|\cdot\|_{2,u}$-$\|\cdot\|^\omega_{2,u}$-isometric, whence it extends to a $\|\cdot\|_{2,u}$-$\|\cdot\|^\omega_{2,u}$-isometric ${}^*$-homomorphism $\M \vnotimes \R \to (\M\vnotimes1_{\R})^\omega$.
Hence, by \eqref{eq:McDuffaif}, for $x\in(\M\vnotimes\R)^1$,
\begin{equation}
u_nxu_n^* \to \eta(x) \in ((\M\vnotimes1_\R)^\omega)^1,
\end{equation}
converging in $\|\cdot\|^\omega_{2,u}$, as required.
\end{proof}
 
 \begin{remark}\label{McDuffProductEmbedding}
 For later use we note that if $(\M_n)_{n=1}^\infty$ is a sequence of strictly separable, McDuff $\mathrm{W}^*$-bundles over the compact spaces $K_n$, and we write $\M^\omega:=\prod^\omega\M_n$, then a reindexing argument shows that for any separable self-adjoint subset $S$ of $\M^\omega$ and $k\in\N$, there is a unital embedding $M_k\hookrightarrow \M^\omega\cap S'$.
 \end{remark}
 
In \cite[Section 5]{O:JMSUT} Ozawa studies when a strictly separable (equivalently, $\|\cdot\|_{2,u}$-separable) $\mathrm{W}^*$-bundle $\M$ over a compact metrizable\footnote{As $\M$ is strictly separable, metrizability of $K$ is redundant here.} space $K$, all of whose fibres are copies of $\R$, is isomorphic  to the trivial bundle $C_\sigma(K,\R)$. In \cite[Theorem 15]{O:JMSUT}, he characterizes these bundles as those which are McDuff, as defined above. This theorem is the $\mathrm{W}^*$-bundle version of Connes' equivalence of property $\Gamma$, McDuff and hyperfiniteness for separably acting injective II$_1$ factors (\cite{C:Ann}). A key consequence of Connes' work is that all separably acting injective II$_1$ factors are hyperfinite, but the corresponding question is open for tracial continuous $\mathrm{W}^*$-bundles.

\begin{question}
Does there exist a nontrivial strictly-separable bundle $\M$ over a compact metrizable space $K$ all of whose fibres are copies of $\R$?  
\end{question}

There are two key conditions where triviality is known to hold: when $K$ has finite covering dimension (this is due to Ozawa, \cite[Corollary 12]{O:JMSUT}), and when $\M=\overline{B}^\mathrm{st}$ for a separable, unital, nuclear $\Z$-stable $\mathrm C^*$-algebra $B$ such that  $T(B)$ is a Bauer simplex. In the latter case, $\overline{B}^\mathrm{st}=\overline{B\otimes\Z}^\mathrm{st}\cong \overline{B}^\mathrm{st}\vnotimes\R$, by Proposition \ref{TensorBundle} (since $\overline{\Z}^\mathrm{st}\cong \R$), and so $\M$ is trivial by Ozawa's theorem (\cite[Theorem 15]{O:JMSUT}). We summarize this in the following theorem.

\begin{thm}\label{OzawaBundles}
Let $B$ be a separable, unital, nuclear $\mathrm C^*$-algebra with no finite dimensional quotients such that $\partial_e T(B)$ nonempty and compact, and suppose that either:
\begin{enumerate}
\item \textrm{$($\cite[Corollary 12]{O:JMSUT}$)$} $\partial_e T(B)$ has finite covering dimension; or
\item $B$ is $\Z$-stable.\footnote{In this case the no finite dimensional quotient condition is of course automatic.}
\end{enumerate}
Then $\overline{B}^\mathrm{st}$ is isomorphic as a $\mathrm{W}^*$-bundle to $C_\sigma(\partial_e T(B),\R)$.
\end{thm}

\bigskip

\subsection{\sc Strict comparison of relative commutant sequence algebras for McDuff bundles}\hfill  \\

\label{sec4.3}

\noindent
In this subsection we show that relative commutants of  $\|\cdot\|_{2,u}$-separable subalgebras of ultraproducts of strictly separable, McDuff $\mathrm{W}^*$-bundles have strict comparison of positive elements (Lemma \ref{lem:W*StrictComp}), which we later apply to obtain strict comparison results for relative commutant $\mathrm C^*$-algebras (Theorem \ref{thm:StrictCompTraceSurj}).
The strategy is to use a partition of unity argument similar to that of \cite[Lemma 14]{O:JMSUT} to extend this local structural result (all fibres, being von Neumann algebras, have strict comparison) to a global result; we use this argument again later (Proposition \ref{lem:TraceImpliesUE} and Lemma \ref{lem:GoodTraceMaps}) and so set up a general technical statement (Lemma \ref{lem:W*RealizingTypes}) which can be used in all occurrences.  We begin with two preparatory lemmas, producing the required partitions of unity.

\begin{lemma}\label{W*SC:Lem1}
Let $\M$ be a strictly separable, McDuff $\mathrm{W}^*$-bundle over a compact metrizable space $K$.  Fix $x_1,\dots,x_r\in\M$ and $\eps>0$. Given a partition of unity $f_1,\dots,f_l\in C(K)_+$, there exist projections $p_1,\dots,p_l\in \M$ with sum $1$ such that
\begin{enumerate}[(i)]
\item\label{W*SC:Lem1.1} $\|[p_i,x_j]\|_{2,u}<\eps$ for $i=1,\dots,l$ and $j=1,\dots,r$;\label{Partition.C1}
\item $\tau_\lambda(p_i)=f_i(\lambda)$, for $i=1,\dots,l$ and $\lambda\in K$;\label{Partition.New}
\item\label{W*SC:Lem1.3} $|\tau_\lambda(p_ix_j)-f_i(\lambda)\tau_\lambda(x_j)|<\eps$ for $i=1,\dots,l$, $j=1,\dots,r$, and $\lambda\in K$.\label{Partition.C2}
\end{enumerate}
\end{lemma}
\begin{proof}
We work in $\M\vnotimes \R$, so suppose $x_1,\dots,x_r\in \M\vnotimes\R$ and $\eps>0$ are given. By hyperfiniteness we can choose a matrix subalgebra $M_k$ of $\R$ and $y_1,\dots,y_r \in \M\vnotimes M_k\subset \M\vnotimes\R$ such that $\|x_j-y_j\|_{2,u}<\frac{\eps}{2}$ for each $j$.  Set $\mathcal S:=M_k'\cap \R$, which is another copy of the hyperfinite II$_1$ factor.  This gives us a $\mathrm W^*$-bundle embedding $\theta:C_\sigma(K,\mathcal S)\hookrightarrow (\M\vnotimes\R)\cap \{y_1,\dots,y_r\}'$ such that 
\begin{equation}
\tau_\lambda(y_i\theta(f))=\tau_\lambda(y_i)\tau_\lambda(f),\quad i=1,\dots,r,\ f\in C_\sigma(K,\mathcal S),\ \lambda\in K,
\end{equation}
arising from the canonical embeddings $C(K)\hookrightarrow \M$ and $\mathcal S\hookrightarrow\R$. Choose a maximal abelian self-adjoint subalgebra $\mathcal A$ in $\mathcal S$. Then $\mathcal A$ can be identified with $L^\infty[0,1]$ in such a way that the restriction of the trace on $\mathcal S$ to $\mathcal A$ is given by integration with respect to Lebesgue measure on $[0,1]$ (see \cite[Theorem 3.5.2]{SS:Book} for example).  With this identification, for  $t\in [0,1]$ let $e_t\in \mathcal A$ be the characteristic function of the interval $[0,t]$ so that $(e_t)_{t\in[0,1]}$ is a monotonically increasing family of projections in $\mathcal A$.  In this way $e_{f_1}$ defines an element of $C_\sigma(K,\mathcal S)$ with $e_{f_1}(\lambda)=e_{f_1(\lambda)}$ for $\lambda\in K$.  Define $p_{1}=\theta(e_{f_1})$. For $i>1$, define $p_{i}$ similarly by
\begin{equation}
p_{i}(\lambda):=\theta(e_{\sum_{j=1}^{i}f_j}-e_{\sum_{j=1}^{i-1}f_j}).
\end{equation}
This defines pairwise orthogonal projections $p_{1},\dots,p_{l}\in \M\cap\{y_1,\dots,y_r\}'$ which sum to $1_\M$, and satisfy $\tau_\lambda(p_{i})=f_i(\lambda)$ for each $\lambda\in K$ and $i=1,\dots,l$ (verifying (\ref{Partition.New})) and $\tau_\lambda(p_iy_j)=f_i(\lambda)\tau_\lambda(y_j)$.   Conditions (\ref{W*SC:Lem1.1}) and (\ref{W*SC:Lem1.3}) follow from the estimate $\|x_j-y_j\|_{2,u}<\frac{\eps}{2}$.
\end{proof}

\begin{lemma}\label{W*SC:Lem2}
Let $(\M_n)_{n=1}^\infty$ be a sequence of  strictly separable, McDuff $\mathrm{W}^*$-bundles over the sequence of compact metrizable spaces $(K_n)_{n=1}^\infty$.  Write $\M^\omega:=\prod^\omega \M_n$ and $K^\omega:=\sum_\omega K_n$, and let $S$ be a $\|\cdot\|_{2,u}^\omega$-separable and self-adjoint subset of $\M^\omega$ containing $1_{\M^\omega}$.  

Given a partition of unity $f_1,\dots,f_l\in C(K^\omega)_+\cong (\prod_\omega C(K_n))_+$, there exist pairwise orthogonal projections $p_1,\dots,p_l\in \M^\omega\cap S'$ with sum $1$ such that
\begin{equation}\label{eq:Partition2Trace}
\tau_\lambda(p_ix)=f_i(\lambda)\tau_\lambda(x),\quad x\in S,\ \lambda\in K^\omega.
\end{equation}
In particular $\tau_\lambda(p_i)=f_i(\lambda)$. 
\end{lemma}

\begin{proof}
Let $(x_j)_{j=1}^\infty$ be a $\|\cdot\|_{2,u}^\omega$-dense sequence in $S$  and let $x_j$ be represented by a uniformly bounded sequence $(x_j^{(n)})_{n=1}^\infty$ with $x_j^{(n)}\in \M_n$.  Lift each $f_i$ to a sequence of representatives $(f_i^{(n)})_{n=1}^\infty$ in $C(K_n)_+$ so that $f_1^{(n)},\dots,f_l^{(n)}$ is a partition of unity in $C(K_n)_+$ for each $n$ (for example by using Choi-Effros to lift the u.c.p.\ map $\mathbb C^l\rightarrow \prod_\omega C(K_n)$ given by $(a_1,\dots,a_l)\mapsto \sum_{i=1}^la_if_i$ to $\prod_{n=1}^\infty C(K_n)$).

By Lemma \ref{W*SC:Lem1}, for each $n$, there are pairwise orthogonal projections $p_1^{(n)},\dots,p_l^{(n)}$ in $\M_n$ with 
\begin{align}
&\|[p_i^{(n)},x_j^{(n)}]\|_{2,u}^{(n)}<\frac{1}{n}\label{L3.14:3.17},\\
&\tau_{\lambda_n}(p_i^{(n)})=f_i^{(n)}(\lambda_n),\label{L3.14:3.18}\\
&|\tau_{\lambda_n}(p_i^{(n)}x_j^{(n)})-f_i^{(n)}(\lambda_n)\tau_{\lambda_n}(x^{(n)}_j)|<\frac{1}{n},\label{L3.14:3.19}
\end{align}
for $i=1,\dots,l$, $j=1,\dots,n$ and $\lambda_n\in K_n$.  Let $p_i$ be the projection in $\M^\omega$ represented by $(p_i^{(n)})_{n=1}^\infty$, so that $p_1,\dots,p_l$ are pairwise orthogonal projections which commute with $S$ by (\ref{L3.14:3.17}).  These satisfy $\tau_\lambda(p_i)=f_i(\lambda)$ for $\lambda\in \prod_\omega K_n$ by (\ref{L3.14:3.18}) and thus for all $\lambda\in \sum_\omega K_n$ by the density of $\prod_\omega K_n$ in $\sum_\omega K_n$.  Similarly, (\ref{L3.14:3.19}) gives $\tau_\lambda (p_ix)=f_i(\lambda)\tau_\lambda(x)$ for all $x\in S$ and $\lambda\in \sum_\omega K_n$.
\end{proof}

We now come to the technical statement which allows us to transfer structural properties from fibres of an ultraproduct of McDuff bundles to the entire bundle. In the statement below we have arranged that the number of variables in the polynomials $h_n$ corresponds to $n$ only to simplify notation.

\begin{lemma}
\label{lem:W*RealizingTypes}
Let $(\M_n)_{n=1}^\infty$ be a sequence of strictly separable, McDuff $\mathrm{W}^*$-bundles over compact spaces $K_n$. Write $\M^\omega:=\prod^\omega \M_n$,  $K^\omega:=\sum_\omega K_n$ and let $E^\omega: \mathcal{M}^{\omega} \to C(K^{\omega})$ be the expectation from Proposition \ref{UltraproductBundles}.
For each $n \in \mathbb{N}$ let $h_n(x_1,\dots,x_n,z_1,\dots,z_n)$ be a $^*$-polynomial in $2n$ noncommuting variables.
Let $(a_i)_{i=1}^\infty$ be a sequence from $\M^\omega$.
Suppose that, for every $\eps > 0$, $k \in \N$, and $\lambda \in K^\omega$, there exist contractions $y^\lambda_1,\dots,y^\lambda_k \in \pi_\lambda(\M^\omega)$ such that
\begin{equation}
\label{eq:W*RealizingTypesFibre}
|\tau_{\pi_\lambda(\M^\omega)}(h_m(\pi_\lambda(a_1),\dots,\pi_\lambda(a_m),y^\lambda_1,\dots,y^\lambda_m))| < \eps,\quad m=1,\dots,k.
\end{equation}
Then there exist contractions $y_i \in \M^\omega$ for $i\in\N$ such that
\begin{equation}\label{ne.3.22}
E^\omega(h_m(a_1,\dots,a_m,y_1,\dots,y_m)) =0,\quad m\in\N.
\end{equation}
\end{lemma}

Recall that while $\tau_\lambda$ is the trace on $\M^\omega$ given by composing the conditional expectation $E^\omega:\M^\omega\to C(K^\omega)$ with evaluation at $\lambda$, the trace $\tau_{\pi_\lambda(\M^\omega)}$ appearing in the above statement is the trace on the fibre $\pi_\lambda(\M^\omega)$, so that
\begin{equation}
\xymatrix{
\M^\omega \ar[r]^{\pi_\lambda} \ar[dr]_{\tau_\lambda} & \pi_\lambda(\M^\omega) \ar[d]^{\tau_{\pi_\lambda(\M^\omega)}} \\
& \mathbb C
} \end{equation}
commutes.
\begin{proof}
Using Kirchberg's $\eps$-test (Lemma \ref{epstest}) with $X_n$ equal to the product of $n$ copies of the unit ball of $\M_n$, it suffices to find for each $k\in\N$ and $\eps>0$, contractions $y_1,\dots,y_k\in\M^\omega$ such that
\begin{equation}\label{e.3.24}
|\tau_\lambda(h_m(a_1,\dots,a_m,y_1,\dots,y_m))| <\eps,\quad m=1,\dots,k,\ \lambda \in K^\omega.
\end{equation}
Once this is done, take bounded representative sequences $(a_{i,n})_{n=1}^\infty$ of each $a_i$, and define $f^{(m)}_n:X_n\rightarrow [0,\infty)$ by
\begin{align}
&f^{(m)}_n(y_{1,n},y_{2,n},\dots,y_{n,n})\\
&:=\begin{cases}\sup_{\lambda_n\in K_n}|\tau_{\lambda_n}(h_m(a_{1,n},\dots,a_{m,n},y_{1,n},\dots,y_{m,n}))|,&m\leq n;\\0,&m>n.\end{cases}\nonumber
\end{align}
Lift the contractions $y_i$ in $\M^\omega$ satisfying (\ref{e.3.24}) to sequences $(y_{i,n})_{n=1}^\infty$ of contractions, so that (\ref{e.3.24}) ensures that
\begin{equation}
\lim_{n\rightarrow\omega}f^{(m)}_n(y_{1,n},\dots,y_{m,n},0,\dots,0)\leq\eps,\quad m\leq k.
\end{equation}
Thus the $\eps$-test gives us contractions $(\tilde{y}_i)_{i=1}^\infty\in\M^\omega$ with contractive lifts $(\tilde{y}_{i,n})_{n=1}^\infty$, such that
\begin{equation}
\lim_{n\rightarrow\omega}f^{(m)}_n(\tilde{y}_{1,n},\dots,\tilde{y}_{n,n})=0,\quad m\in\N.
\end{equation}
From this we have 
\begin{equation}\label{ne.3.28}
|\tau_\lambda(h_m(a_1,\dots,a_m,\tilde{y}_1,\dots,\tilde{y}_m))|=0,\quad m\in\N,\ \lambda\in \prod_\omega K_n,
\end{equation}
whence the density of $\prod_\omega K_n$ in $\sum_\omega K_n$ gives (\ref{ne.3.28}) for all $\lambda\in K^\omega$, which is precisely 
(\ref{ne.3.22}).

We now verify (\ref{e.3.24}), so fix $k\in\N$ and $\eps>0$. For each $\lambda\in K^\omega$, by hypothesis there are contractions $y_1^\lambda,\dots,y_k^\lambda\in\pi_\lambda(\M^\omega)$ such that \eqref{eq:W*RealizingTypesFibre} holds.  Lift these to contractions $\tilde{y}_1^\lambda,\dots,\tilde{y}_k^\lambda\in \M^\omega$ and so there is an open neighbourhood $U_\lambda$ of $\lambda$ in  $K^\omega$ such that 
\begin{equation}\label{L3.15:3.33}
|\tau_\mu(h_m(a_1,\dots,a_m,\tilde{y}^\lambda_1,\dots,\tilde{y}^\lambda_m))| < \eps,\quad m=1,\dots,k,\ \mu \in U_\lambda.
\end{equation}
By compactness, find $\lambda_{1},\dots,\lambda_{l}\in K^\omega$ so that $\bigcup_{i=1}^l U_{\lambda_{i}}$ covers $K^\omega$, and then let $f_{1},\dots,f_{l}\in C(K^\omega)_+$ be a partition of unity subordinate to this cover.
Set 
\begin{eqnarray}
\nonumber
T & := &
\{a_1,\dots,a_m,a_1^*,\dots,a_m^*\} \nonumber \\
& & \cup \{\tilde{y}^{\lambda_i}_j,\tilde{y}^{\lambda_i}_j{}^*\mid i =1,\dots,l,\ j=1,\dots,k\} \nonumber \\
& &\cup \{h_m(a_1,\dots,a_m,\tilde{y}^{\lambda_i}_1,\dots,\tilde{y}^{\lambda_i}_m) \mid i=1,\dots,l,\ m=1,\dots,k\}\nonumber\\
& &\cup \{h_m(a_1,\dots,a_m,\tilde{y}^{\lambda_i}_1,\dots,\tilde{y}^{\lambda_i}_m)^* \mid i=1,\dots,l,\ m=1,\dots,k\},
\end{eqnarray}
and then use Lemma \ref{W*SC:Lem2} to find pairwise orthogonal projections $p_{1},\dots,p_{l}\in\M^\omega \cap T'$ which sum to $1_{\M^\omega}$ and satisfy \eqref{eq:Partition2Trace} for all $x \in T$ and $\lambda \in K^\omega$.
For $j=1,\dots,k$, set $y_j := \sum_{i=1}^l p_i\tilde{y}^{\lambda_i}_j$.
Since the $p_i$ are pairwise orthogonal and commute with the $\tilde{y}^{\lambda_i}_j$, $y_j$ is a contraction.  Further, since $p_1,\dots,p_l$ is an orthogonal partition of unity commuting with $T$, we have
\begin{align}
\label{eq:W*RealizingTypesStep1}
\notag
h_m(a_1,\dots,a_m,y_1,\dots,y_m) &= \sum_{i=1}^l p_i h_m(a_1,\dots,a_m,y_1,\dots,y_m) \\
\notag
&= \sum_{i=1}^l p_ih_m(a_1,\dots,a_m,p_iy_1,\dots,p_iy_m) \\
&= \sum_{i=1}^l p_ih_m(a_1,\dots,a_m,\tilde{y}^{\lambda_i}_1,\dots,\tilde{y}^{\lambda_i}_m).
\end{align}

For $\lambda \in K^\omega$, as the support of $f_i$ is contained in $U_{\lambda_i}$, if $f_i(\lambda) \neq 0$, then (\ref{L3.15:3.33}) ensures that
\begin{equation}
\label{eq:W*RealizingTypesStep2}
|\tau_\lambda(h_m(a_1,\dots,a_m,\tilde{y}^{\lambda_i}_1,\dots,\tilde{y}^{\lambda_i}_m))| < \eps,\quad m=1,\dots,k.
\end{equation}
Therefore, for $\lambda \in K^\omega$,
\begin{eqnarray}
\lefteqn{ |\tau_\lambda(h_m(a_1,\dots,a_m,y_1,\dots,y_m))| } \notag \\ 
&\stackrel{\eqref{eq:W*RealizingTypesStep1}}\leq& \sum_{i=1}^l |\tau_\lambda(p_ih_m(a_1,\dots,a_m,\tilde{y}^{\lambda_i}_1,\dots,\tilde{y}^{\lambda_i}_m))| \notag \\
&\stackrel{\eqref{eq:Partition2Trace}}=& \sum_{i=1}^l f_i(\lambda) |\tau_\lambda(h_m(a_1,\dots,a_m,\tilde{y}^{\lambda_i}_1,\dots,\tilde{y}^{\lambda_i}_m))| \notag \\
&\stackrel{\eqref{eq:W*RealizingTypesStep2}}<& \sum_{i=1}^l f_i(\lambda) \eps = \eps,
\end{eqnarray}
as required.
\end{proof}
\begin{remark}\label{CommRemark}
Note that, for example, given a $\|\cdot\|_{2,u}^\omega$-separable and self-adjoint subset $S\subset \M^\omega$, we can encode the relation $y\in S'$ by the conditions $(\|[y,s_n]\|_{2,u}^\omega)^2=\|E^\omega((ys_n-s_ny)^*(ys_n-s_ny))\|=0$ for a dense sequence $(s_n)_{n=1}^\infty$ from $S$.  In this way we can use Lemma \ref{lem:W*RealizingTypes} to construct elements inside relative commutants, as in the following strict comparison result.
\end{remark}

\begin{lemma}
\label{lem:W*StrictComp}
Let $(\M_n)_{n=1}^\infty$ be a sequence of strictly separable, McDuff $\mathrm{W}^*$-bundles over compact spaces $K_n$. Write $\M^\omega:=\prod^\omega \M_n$ and $K^\omega:=\sum_\omega K_n$.  Let $S$ be a $\|\cdot\|^\omega_{2,u}$-separable and self-adjoint subset of $\M^\omega$, and let $p$ be a projection in the centre of $\M^\omega\cap S'$ (as happens when $p\in S\cap S'$). Then $p(\M^\omega\cap S')$ has strict comparison of positive elements, as in Definition \ref{defn:StrictComp}.
\end{lemma}

\begin{proof}
Note that $M_k(p(\M^\omega\cap S'))$ is isomorphic to 
\begin{equation}(p \otimes 1_k)\Big(\big(\prod^\omega \M_n \otimes M_k\big)\cap (S \otimes 1_k)'\Big),
\end{equation} so it suffices to take $k=1$ and positive contractions $a,b\in p(\M^\omega\cap S')$ with $d_\tau(a) < d_\tau(b)$ for all $\tau \in T(p(\M^\omega\cap S'))$ and show $a\preceq b$ in $\M^\omega\cap S'$.

Let $\eps > 0$. With $g_\eps$ as defined in \eqref{eq:gepsDef}, we have
\begin{equation}\label{lem:W*comp.newe}
d_\tau((a-\eps)_+)\leq \tau(g_\eps(a))\leq d_\tau(a)<d_\tau(b),\quad \tau\in T(p(\M^\omega\cap S')).
\end{equation}
As the map $\tau\mapsto \tau(g_\eps(a))$ is continuous and $\tau\mapsto d_\tau(b)$ is the increasing (as $\dl \to 0$) pointwise supremum of continuous functions $\tau \mapsto \tau(g_\dl(b))$, there is $\delta > 0$ such that
\begin{equation}\label{lem:W*comp.e1}
\tau(g_\eps(a))<\tau(g_\delta(b)),\quad \tau\in T(p(\M^\omega\cap S')).
\end{equation}

We shall now use Lemma \ref{lem:W*RealizingTypes} to show that there exists $y \in p(\M^\omega \cap S')$ such that
\begin{equation}
g_{\dl/2}(b)y=y \quad \text{and} \quad y^*y=(a-2\eps)_+.
\end{equation}
We encode the condition $y\in S'$ as in Remark \ref{CommRemark}, and the condition that $py=yp=y$ as $(\|py-y\|_2^\omega)^2=(\|yp-y\|_2^\omega)^2=0$, so that Lemma \ref{lem:W*RealizingTypes} shows that it is sufficient to find, for each $\lambda\in K^\omega$, a contraction $y_\lambda \in \mathcal N_\lambda:=\pi_\lambda(p)\big(\pi_\lambda(\M^\omega) \cap \pi_\lambda(S)'\big)$ such that
\begin{equation}
\pi_\lambda(g_{\dl/2}(b))y_\lambda=y_\lambda \quad \text{and} \quad y_\lambda^*y_\lambda=\pi_\lambda((a-2\eps)_+).
\end{equation}

Fix $\lambda\in K^\omega$ so that $\pi_\lambda(a),\pi_\lambda(b) \in \mathcal N_\lambda$.  There are two cases to consider. Firstly, when $\pi_\lambda(p)=0$, then $\pi_\lambda(a)=\pi_\lambda(b)=0$ and we can take $y_\lambda=0$ also.  Secondly, when $\pi_\lambda(p)\neq 0$, (\ref{lem:W*comp.newe}) and (\ref{lem:W*comp.e1}) give
\begin{equation} d_\tau(\pi_\lambda((a-\eps)_+)) < d_\tau(\pi_\lambda(g_\delta(b)))\quad \tau\in T(\mathcal N_\lambda), \end{equation}
as $\pi_\lambda(p)\neq 0$. Since $\pi_\lambda(\M^\omega)$ is a finite von Neumann algebra (\cite[Theorem 11]{O:JMSUT}), so too is $\mathcal N_\lambda$, which therefore has strict comparison (this is easily seen using the Borel functional calculus and comparison of projections in finite von Neumann algebras), so that $\pi_\lambda((a-\eps)_+)\preceq \pi_\lambda(g_\delta(b))$ in $\mathcal N_\lambda$. Then by \cite[Proposition 2.4]{R:JFA2} there exists $r_\lambda \in \mathcal N_\lambda$ and $\eta_\lambda>0$ such that 
\begin{equation}
\label{L3.17:3.34}
\pi_\lambda((a-2\eps)_+)=r_\lambda^*(\pi_\lambda(g_\delta(b))-\eta_\lambda)_+r_\lambda.
\end{equation}
Set $y_\lambda=((\pi_\lambda(g_\delta(b))-\eta_\lambda)_+)^{1/2}r_\lambda$, so that
\begin{equation}
\pi_\lambda(g_{\delta/2}(b)) y_\lambda = y_\lambda \quad \text{and}\quad y_\lambda^*y_\lambda = \pi_\lambda((a-2\eps)_+), \end{equation}
as required (since $y_\lambda$ is automatically contractive from (\ref{L3.17:3.34})).
\end{proof}

\bigskip

\subsection{\sc Traces on a relative commutant}\hfill  \\

\noindent
In this section, we describe a generating collection of traces on the relative commutant of a separable nuclear $\mathrm{C}^*$-subalgebra in an ultraproduct of McDuff $\mathrm{W}^*$-bundles.  These results are bundle versions of the fact that the central sequence algebra $\R^\omega\cap \R'$ is a II$_1$ factor (and hence has unique trace); see \cite[Theorem XIV.4.18]{Takesaki.3}.  Recall that a trace is normal if it is continuous with respect to strong operator topology on bounded sets of $\M$ (this is independent of the choice of faithful representation). We use $T(\M)$ for the collection of all traces on the finite von Neumann algebra $\M$ (not just the normal traces). Since every trace on $\M$ is of the form $\phi\circ E$ where $E$ is the centre-valued trace on $\M$ (which is normal) and $\phi$ is a state on the centre of $\M$, normal traces are dense in the set of all traces; see \cite[Proposition 8.3.10]{KR:2}. 

\begin{lemma}
\label{lem:vnCommTraces}
Let $\M$ be a finite von Neumann algebra with faithful normal trace $\tau_{\M}$, used to define $\|\cdot\|_2$ on $\M$.
Let $A$ be a separable, unital, nuclear $\mathrm C^*$-algebra and $\phi:A \to \M$ a unital $^*$-homomorphism.
Set $\mathcal N:=\M \cap \phi(A)'$.
Define $T_0$ to be the set of all traces on $\mathcal N$ of the form $\tau(\phi(a)\cdot)$ where $\tau \in T(\M)$ is a normal trace and $a \in A_{+}$ has $\tau(\phi(a))=1$. Suppose that $z \in \mathcal N$ is a contraction for which $\rho(z)=0$ for all $\rho \in T_0$.
Then for any finite subset $\mathcal F$ of $A$ and $\eps > 0$, there exist elements $x_1,\dots,x_{10},y_1,\dots,y_{10} \in \M$ of norm at most 12, such that
\begin{equation}
\label{eq:vnCommTraces10Comm}
 z = \sum_{i=1}^{10} [x_i,y_i]
\end{equation}
and $\|[x_i,\phi(a)]\|_2, \|[y_i,\phi(a)]\|_2 < \eps$ for all $a \in \mathcal F$. 

If additionally $(\M,\tau_\M)$ arises as an ultraproduct of tracial von Neumann algebras, then the weak$^*$-closed convex hull of $T_0$ is $T(\mathcal N)$.
\end{lemma}

\begin{proof}
Define $S_0$ to be the set of all traces on $\mathcal N$ of the form $\tau(b\cdot)$ where $\tau \in T(\M)$ and $0 \le b \in \M \cap \phi(A)''$, such that $\tau(b)=1$. Let us first show that $S_0$ is the closure of $T_0$.
First given a normal trace $\tau\in T(\M)$ and $b\in \M\cap\phi(A)''$ with $\tau(b)=1$, take a bounded sequence $(a_n)_{n=1}^\infty$ in $\phi(A)$ with $a_n\rightarrow b$ in strong-operator topology by Kaplansky's density theorem. Then $\tau(a_n)\rightarrow \tau(b)=1$, so renormalizing, we can assume that $\tau(a_n)=1$ for all $n$. Further, for each $x\in \M$, $a_nx\rightarrow bx$ in strong operator topology so that $\tau(a_n x)\to\tau(b x)$.
Thus $\tau(b\cdot)$ lies in the weak$^*$-closure of $T_0$.  As the normal traces on $\M$ are weak$^*$-dense in the collection of all traces on $\M$, $T_0$ is weak$^*$-dense in $S_0$. Therefore, our hypothesis on $z$ ensures that $\rho(z)=0$ for all $\rho \in S_0$.

Since $A$ is nuclear, $\M\cap \phi(A)''$ is hyperfinite (by Connes' theorem \cite{C:Ann}) and so we may find a finite dimensional subalgebra $B$ of $\M\cap \phi(A)''$ such that for all $a\in\F$  there exists $b \in B$ such that $\|\phi(a)-b\|_2 < \eps/10$. We shall arrange that $x_i,y_i \in \M \cap B'$; it will then follow from the estimates $\|x_i\|,\|y_i\|\leq 12$, that $\|[x_i,\phi(a)]\|_2, \|[y_i,\phi(a)]\|_2 < \eps$ for $a \in \mathcal F$.

Let $B=\bigoplus_{k=1}^l M_{r_k}$, and let $p_k$ be a minimal projection in $M_{r_k}$, so that $\M \cap B'$ is isomorphic to $p_1\M p_1 \oplus \cdots \oplus p_l\M p_l$ via an isomorphism  $\psi:p_1\M p_1 \oplus \cdots \oplus p_l\M p_l\rightarrow \M\cap B'$
 of the form 
 \begin{equation}
 \psi(x_1,\dots,x_l)=\sum_{k=1}^l\sum_{i=1}^{r_k}e_{i1}^{(k)}x_le_{1i}^{(k)}
 \end{equation} 
 where $(e_{ij}^{(k)})_{i,j=1}^{r_k}$ is a set of matrix units for $M_{r_k}$ with $e_{11}^{(k)}=p_k$. We have that $z \in \M \cap B'$, and by hypothesis, $\tau(p_jz) = 0$ for every $\tau \in T(\M)$, since $\tau(p_j\cdot)$ is a scalar multiple of a trace in $S_0$. 
By \cite[Theorem 3.2]{FackdelaHarpe}, there exist $x_{1,j},\dots,x_{10,j},y_{1,j},\dots,y_{10,j} \in p_j\M p_j$ of norm at most $12$, such that
\begin{equation} p_jz = \sum_{i=1}^{10} [x_{i,j},y_{i,j}]. \end{equation}
Set $x_i=\psi(x_{i,1}\oplus\cdots\oplus x_{i,l})$ and $y_i= \psi(y_{i,1} \oplus \cdots \oplus y_{i,l})$ so that $x_i,y_i\in \M\cap B'$ satisfy
 \eqref{eq:vnCommTraces10Comm}, as required.

Now consider the case when $(\M,\tau_\M)$ is an ultraproduct of tracial von Neumann algebras. This means that we have a sequence $(\M_n)_{n=1}^\infty$ of finite von Neumann algebras with a distinguished faithful trace $\tau_n \in T(\M_n)$ for each $n$, and $\M=\M^\omega:=\prod^\omega \M_n$ is the tracial ultrapower, with the distinguished faithful trace $\tau_\M$ arising from the sequence $(\tau_n)_{n=1}^\infty$.
In this case Kirchberg's $\eps$-test enables us to take $x_i,y_i\in\mathcal N$. Indeed, lift $z$ to a sequence $(z_n)_{n=1}^\infty$ in $\prod_{n=1}^\infty \M_n$, $\phi$ to a sequence of c.p.c.\ maps $\phi_n:A\rightarrow\M_n$ (by the Choi-Effros lifting theorem \cite{CE:Ann}) and take a dense sequence $(a_k)_{k=1}^\infty$ in $A$.  For each $n\in\N$, let $X_n$ denote the set of tuples $(\bar{x}_n,\bar{y}_n)=(x_{1,n},\dots,x_{10,n},y_{1,n},\dots,y_{10,n})$ of elements of $\M_n$ of norm at most $12$, and define functions $f^{(k)}_n:X_n\rightarrow[0,\infty)$ for $k\geq 0$ by
\begin{align}
f^{(0)}_n(\bar{x}_n,\bar{y}_n)&=\|\sum_{i=1}^{10}[x_{i,n},y_{i,n}]-z_n\|_{2,\tau_n}\\ 
f^{(k)}_n(\bar{x}_n,\bar{y}_n)&=\max_{i=1,\dots,10}\big(\|[x_{i,n},\phi_n(a_k)]\|_{2,\tau_n},\|[y_{i,n},\phi_n(a_k)]\|_{2,\tau_n}),\quad k\geq 1.\nonumber
\end{align}
The first part of the lemma ensures that for any $\eps>0$ and $m\in\N$, we can find a sequence $(\bar{x}_n,\bar{y}_n)_{n=1}^\infty$ in $\prod_{n=1}^\infty X_n$ so that $\lim_{n\rightarrow\omega}f^{(k)}_n(\bar{x}_n,\bar{y}_n)\leq \eps$ for $k\leq m$.  Thus the $\eps$-test (Lemma \ref{epstest}) provides a sequence $(\bar{x}_n,\bar{y}_n)_{n=1}^\infty$ in $\prod_{n=1}^\infty X_n$ with $\lim_{n\rightarrow\omega}f^{(k)}_n(\bar{x}_n,\bar{y}_n)=0$ for all $k=0,1,\dots$.  The elements $x_1,\dots,x_{10},y_1,\dots,y_{10}$ in $\M^\omega$ represented by this sequence lie in $\mathcal N$ and have (\ref{eq:vnCommTraces10Comm}). Therefore, $\rho(z)=0$ for every $\rho \in T(\mathcal N)$. Thus by the Hahn-Banach theorem, the weak$^*$-closed convex hull of $T_0$ is $T(\mathcal N)$.
\end{proof}

\begin{prop}
\label{thm:W*CommTraces}
Let $(\M_n)_{n=1}^\infty$ be a sequence of strictly separable, McDuff $\mathrm{W}^*$-bundles over compact spaces $K_n$.
Write $\M^\omega:=\prod^\omega \M_n$ and $K^\omega:=\sum_\omega K_n$.
Let $A$ be a separable, unital, nuclear $\mathrm C^*$-algebra and $\phi:A \to \M^\omega$ a $^*$-homomorphism.
Set $C:=\M^\omega \cap \phi(A)' \cap \{1_{\M^\omega}-\phi(1_A)\}^\perp$.
Define $T_0$ to be the set of all traces on $C$ of the form $\tau(\phi(a)\cdot)$ where $\tau \in T(\M^\omega)$ and $a \in A_+$ with $\tau(\phi(a))=1$.
If $z \in C$ satisfies $\rho(z)=0$ for all $\rho\in T_0$ then there exist elements $x_1,\dots,x_{10},y_1,\dots,y_{10} \in C$ of norm at most 12, such that
\begin{equation}\label{T3.19:3.37}
 z = \sum_{i=1}^{10} [x_i,y_i].
\end{equation}
In particular, $T(C)$ is the closed convex hull of $T_0$.
\end{prop}

\begin{proof}
This is a direct consequence of Lemma \ref{lem:W*RealizingTypes} and Lemma \ref{lem:vnCommTraces}.  The condition (\ref{T3.19:3.37}) can be written as $E^\omega((z- \sum_{i=1}^{10} [x_i,y_i])^*(z- \sum_{i=1}^{10} [x_i,y_i]))=0$, so is of the form required by of Lemma \ref{lem:W*RealizingTypes}. As noted in Remark \ref{CommRemark}, we can also encode the condition $x_i,y_i\in \M^\omega\cap \phi(A)'$ in the required form, and similarly for the condition $x_i,y_i\in \{1_{\M^\omega}-\phi(1_A)\}^\perp$. Lemma \ref{lem:vnCommTraces} shows that we can find the required elements approximately satisfying the required conditions in the images $\pi_\lambda(\M^\omega)\cap \{1_{\pi_\lambda(\M^\omega)}-\pi_\lambda(\phi(1_A))\}^\perp$ for each $\lambda\in K^\omega$, and hence Lemma \ref{lem:W*RealizingTypes} provides suitable $x_1,\dots,x_{10}$ and $y_1,\dots,y_{10}$ in $C$.
\end{proof}

\bigskip

\subsection{\sc Unitary equivalence of maps into ultraproducts}\hfill  \\

\noindent
Though we do not need it in the sequel, we end with another application of the local-to-global structure mechanism which fits the theme of the paper, showing that maps from separable nuclear $\mathrm{C}^*$-algebras into ultraproducts of McDuff bundles are classified by traces.
This is an extension of a well known consequence of the equivalence of hyperfiniteness and injectivity, that two normal $^*$-homomorphisms $\phi,\psi:\M\rightarrow \mathcal N$ from a finite injective von Neumann algebra into a finite von Neumann algebra are approximately unitarily equivalent if and only if $\tau_{\mathcal N}\circ\phi=\tau_{\mathcal N}\circ\psi$  for all normal traces $\tau_{\mathcal N}$ on $\mathcal N$ (see \cite{J:MA}, for the case when $\mathcal N$ is a factor, which also shows that this uniqueness condition characterizes injectivity of $\M$).
Thus, when $A$ is separable and nuclear, two $^*$-homomorphisms $\phi,\psi:A\rightarrow \mathcal N$ with $\tau_{\mathcal N}\circ\phi=\tau_{\mathcal N}\circ\psi$ will be approximately unitarily equivalent, as these homomorphisms will extend to normal $^*$-representations on the finite part of the bidual $A^{**}$ which also agree on the trace (see the discussion before Lemma 2.5 of \cite{BCW:Abel}).  

\begin{prop}
\label{lem:TraceImpliesUE}
Let $(\M_n)_{n=1}^\infty$ be a sequence of strictly separable, McDuff $\mathrm{W}^*$-bundles over compact spaces $K_n$.
Write $\M^\omega:=\prod^\omega \M_n$ and $K^\omega:=\sum_\omega K_n$.  Let $A$ be a separable, nuclear $\mathrm C^*$-algebra, and suppose that $\phi,\psi:A\rightarrow \M^\omega$ are $^*$-homomorphisms such that 
\begin{equation}
\tau\circ\phi=\tau\circ\psi,\quad \tau\text{ a $\|\cdot\|_{2,u}^\omega$-continuous trace on }\M^\omega.
\end{equation}
Then $\phi$ and $\psi$ are unitarily equivalent in $\M^\omega$.
\end{prop}

\begin{proof}
This is a consequence of Lemma \ref{lem:W*RealizingTypes}. Fixing a dense sequence $(x_i)_{i=1}^\infty$ in $A$, we encode a unitary $u$ satisfying $\psi(x)=u\phi(x)u^*$ for all $x\in A$ by the conditions $\|u^*u-1\|_{2,u}^\omega=0$, $\|uu^*-1\|_{2,u}^\omega=0$, and $\|\psi(x_i)-u\phi(x_i)u^*\|_{2,u} ^\omega= 0$ for all $i$.
To verify the hypothesis of Lemma \ref{lem:W*RealizingTypes}, it suffices to show that for any $n\in\N, \eps>0$, and $\lambda \in K^\omega$, there exists a unitary $u_\lambda\in\pi_\lambda(\M^\omega)$ such that
\begin{align}
\|\pi_\lambda(\psi(x_i))-u_\lambda^*\pi_\lambda(\phi(x_i))u_\lambda\|_{2,\tau_{\pi_\lambda(\M^\omega)}} < \eps, \quad i=1,\dots,n.
\end{align}
Such a unitary does exist by the uniqueness result discussed in the paragraph preceding the lemma, since $\pi_\lambda \circ \phi$ agrees with $\pi_\lambda \circ \psi$ on all normal traces on $\pi_{\lambda}(\mathcal M^\omega)$.
\end{proof}

\clearpage\section{Property (SI) and its consequences}
\label{sec:StrictComp}

\noindent
In this section we develop key structural properties for relative commutant sequence algebras, establishing the following theorem.

\begin{thm}
\label{thm:StrictCompTraceSurj}
Let $(B_n)_{n=1}^\infty$ be a sequence of simple, separable, unital, finite, $\mathcal Z$-stable $\mathrm C^*$-algebras with $QT(B_n)=T(B_n)$ for all $n$. Set $B_\omega := \prod_\omega B_n$, and let $J_{B_\omega}$ be the trace-kernel ideal as defined in \eqref{eq:JBdef}.
Let $A$ be a separable, unital, nuclear $\mathrm C^*$-algebra and let $\pi:A\rightarrow B_\omega$ be a c.p.c.\ order zero map such that $\pi(a)$ is full for each nonzero $a\in A$ and the induced map $\bar{\pi}:A \to B_\omega/J_{B_\omega}$ is a ${}^*$-homomorphism.
Set
\begin{equation}
\label{eq:StrictCompCdef}
 C:=B_\omega \cap \pi(A)' \cap \{1_{B_\omega}-\pi(1_A)\}^\perp, \quad \bar{C}:= C/(C\cap J_{B_\omega}).
\end{equation}
Then:
\begin{enumerate}
\item Every trace on $C$ comes from a trace on $\bar{C}$, and in particular, $T(C)$ is a Choquet simplex.
\item If $\partial_e T(B_n)$ is compact for each $n\in\N$, then $C$ has strict comparison of positive elements with respect to traces (as in Definition \ref{defn:StrictComp}).
\item If $\partial_e T(B_n)$ is compact for each $n\in\N$, then the set of traces on $C$ of the form $\tau(\pi(a)\cdot)$ for $\tau\in T(B_\omega)$ and $a\in A_{+}$ with $\tau(\pi(a))=1$, has weak$^*$-closed convex hull $T(C)$.
\end{enumerate}
\end{thm}

In the situation of Theorem \ref{thm:StrictCompTraceSurj}, the algebra $\bar{C}$ will be a $\mathrm{W}^*$-bundle relative commutant sequence algebra as studied in Section \ref{sec4.3}.  Thus our strategy for proving the theorem follows the approach of \cite[Proposition 3.3(i)]{MS:DMJ} and \cite[Theorem 4.8]{MS:arXiv} in that we aim to lift the structural properties of $\bar{C}$ developed in Section \ref{sec:Bundles} back to $C$. The key tool required to do this is establishing (in Lemma \ref{lem:relSI}) that the c.p.c.\ order zero map $\pi$ has property (SI), in the sense of Definition \ref{def:SI}; this is the subject of Section \ref{sect5.1}.  We then give the proof of Theorem \ref{thm:StrictCompTraceSurj} in Section \ref{sect5.2}. 

The form of the algebra $C$ is critical: sticking to elements upon which $\pi(1_A)$ acts as a unit is needed for property (SI) to hold (see Remark \ref{rmk:relSI}(ii)), yet an algebra $C$ of the form in \eqref{eq:StrictCompCdef} can be used to establish unitary equivalence of order zero maps through the $2\times 2$ matrix trick of Section \ref{sec:Matrix}.

Note too that the collections of traces appearing in (iii) have been used as a critical set of traces in previous work on central sequence algebras; in particular the analogous set of traces of the form $\tau(a\cdot)$ on $\ell^\infty(A)/c_0(A)\cap A'$ play a crucial role in \cite{W:Invent1,W:Invent2} (see \cite[Proposition 5.2]{W:Invent2} for example, as well as \cite[Section 5]{RT:Preprint}).
At least in the $\Z$-stable compact extreme tracial boundary case, part (iii) shows why these are the right traces to consider.

\bigskip

\subsection{\sc Property (SI)}\hfill  \\

\label{sect5.1}

Here we show that the order zero maps appearing in Theorem \ref{thm:StrictCompTraceSurj} have property (SI), in the following sense. 

\begin{defn}[{cf.\ \cite[Definition 4.1]{MS:CMP} and \cite[Lemma 3.2]{MS:DMJ}}]
\label{def:SI}
Let $(B_n)_{n=1}^\infty$ be a sequence of simple, separable, unital, finite, $\mathrm C^*$-algebras with $QT(B_n)=T(B_n)\neq\emptyset$ for all $n$. Set $B_\omega := \prod_\omega B_n$, and define $J_{B_\omega}$ as in \eqref{eq:JBdef}.
Let $A$ be a separable, unital $\mathrm C^*$-algebra, let $\pi:A\rightarrow B_\omega$ be a c.p.c.\ order zero map, and define 
\begin{equation}
C:=B_\omega \cap \pi(A)' \cap \{1_{B_\omega}-\pi(1_A)\}^\perp, \quad \bar{C}:= C/(C\cap J_{B_\omega}).
\end{equation}
The map $\pi$ has \emph{property (SI)} if the following condition holds. For all $e,f\in C_+$ such that $e\in J_{B_\omega}$, $\|f\|=1$ and $f$ has the property that, for every nonzero $a \in A_+$, there exists $\gamma_a > 0$ such that
\begin{equation}
\label{eq:relSItr}
\tau(\pi(a)f^n) > \gamma_a,\quad \tau\in T_\omega(B_\omega),\ n\in\mathbb N, \end{equation}
there exists $s \in C$ such that
\begin{equation}
\label{eq:relSIToShow}
 s^*s = e \quad\text{and}\quad fs=s.
\end{equation}
\end{defn}

\begin{remark}
\label{rmk:relSI-simple}
If, in addition to the conditions above, $A$ is simple, then, as we show below, for $f \in C_+$, the condition \eqref{eq:relSItr} is the same as asking that there exists $\gamma > 0$ such that $\tau(f^n) > \gamma$ for all $\tau\in T_\omega(B_\omega)$ and $n\in\N$ (this is the largeness condition that appears in prior property (SI) theorems in the literature).
In particular:
\begin{enumerate}
\item When $A$ is simple, $B_n=B$ for all $n$, and $\pi$ is a unital ${}^*$-homo\-morphism, then $\pi$ has property (SI) in the sense of Definition \ref{def:SI} if and only if $B$ has property (SI) relative to $\pi(A)$ in the sense of \cite[Lemma 3.2]{MS:DMJ}.
\item When $A$ is simple, $B_n=A$ for all $n$, and $\pi:A\to A_\omega$ is the canonical inclusion, then $\pi$ has property (SI) in the sense of Definition \ref{def:SI} if and only if $A$ has property (SI) in the sense of \cite[Definition 4.1]{MS:CMP}.
\end{enumerate}

To see this, first note that if \eqref{eq:relSItr} holds, then as $f\vartriangleleft\pi(1_A)$, $\tau(f^n) > \gamma_{1_A}$ for all $\tau\in T_\omega(B_\omega)$ and all $n\in\N$.  Conversely, when $A$ is simple and $\tau(f^n)>\gamma$ for all $\tau\in T_\omega(B_\omega)$ and all $n\in\N$, take a nonzero $a\in A_+$. Since $A$ is simple, there exists $y_1,\dots,y_m \in A$ such that $\sum y_jay_j^* = 1_A$.
Let $\hat\pi:A \to B_\omega \cap \{f\}'$ be a supporting c.p.c.\ order zero map for $\pi$ as given by Lemma \ref{lem:SupportingMap}.  Then for $\tau\in T(B_\omega)$, since $f \vartriangleleft \pi(1_A)$,
\begin{eqnarray}
\notag
\tau(f^n)&=&\tau(\pi(1_A)f^n)\\
\notag
&=&\sum_{j=1}^m\tau(\pi(y_jay_j^*)f^n)\\
\notag
&\stackrel{(\ref{eq:Supporting}),\ f\in \hat{\pi}(A)'}{=}&\sum_{j=1}^m\tau(\hat{\pi}(y_j^*)\hat{\pi}(y_j)\pi(a)f^n)\\
\notag
&=&\sum_{j=1}^m\tau(f^{n/2}\pi(a)^{1/2}\hat{\pi}(y_j^*)\hat{\pi}(y_j)\pi(a)^{1/2}f^{n/2})\\
&\leq&\sum_{j=1}^m\|\hat{\pi}(y_j)\|^2\tau(\pi(a)f^n)
\end{eqnarray}
so that \eqref{eq:relSItr} holds with $\gamma_a=\gamma/(\sum_{j=1}^m\|\hat{\pi}(y_j)\|^2)>0$.
\end{remark}

We now state our property (SI) result. While we state it for all c.p.c.\ order maps $\pi$ note that in order for property (SI) to be non-vacuous, it is necessary that $\pi(a)$ is full in $B_\omega$ for each nonzero $a\in A_+$.

\begin{lemma}
\label{lem:relSI}
Let $(B_n)_{n=1}^\infty$ be a sequence of simple, separable, unital, finite, $\mathcal Z$-stable $\mathrm C^*$-algebras with $QT(B_n)=T(B_n)$ for all $n$. Set $B_\omega := \prod_\omega B_n$, and define $J_{B_\omega}$ as in \eqref{eq:JBdef}.
Let $A$ be a separable, unital, nuclear $\mathrm C^*$-algebra.
Then every c.p.c.\ order zero map $\pi:A\rightarrow B_\omega$ has property (SI).
\end{lemma}

\begin{remark}
\label{rmk:relSI}
(i) Let $\theta:A\rightarrow\mathcal B(\mathcal H)$ be the faithful representation of $A$ obtained by taking the direct sum of one  representation from each unitary equivalence class of irreducible representations of $A$.  If $\theta(A)$ contains no compact operators then the proof of Lemma \ref{lem:relSI} will work if, instead of assuming that each $B_n$ is $\mathcal Z$-stable, we only ask that each $B_n$ is stably finite and has strict comparison of positive elements. This as because the only essential use of $\Z$-stability of $B_n$ in our proof of Lemma \ref{lem:relSI} (beyond strict comparison) is to provide space for an augmented map $\tilde{\pi}:A\otimes\Z\rightarrow B_\omega$, as certainly no irreducible representation of $A\otimes\Z$ can contain a compact operator. When $\theta(A)$ contains no compact operators this augmented map is not necessary. In this way when $\pi$ is a unital $^*$-homomorphism and $A$ is simple and infinite dimensional, we regain precisely the property (SI) statement of \cite[Lemma 3.2]{MS:DMJ} (by Remark \ref{rmk:relSI-simple} (i)).

(ii) The choice of definition of $C$ is crucial. If, in place of $C$, one used the hereditary subalgebra of $B_\omega \cap \pi(A)'$ generated by $\pi(1_A)$ (or even, all of $B_\omega \cap \pi(A)'$), the conclusion of Lemma \ref{lem:relSI} no longer holds.
To see this, suppose that $A=\mathbb C$ and $B=\Z$. Let $h\in\Z_+$ be a contraction of full spectrum with $d_\tau(h)=1$ so that $\pi_n:A\rightarrow B$ given by $\pi_n(1_A)=h^{1/n}$ induces a c.p.c.\ order zero map $\pi:A\rightarrow B_\omega$. In this way, $\bar{\pi}:A\rightarrow B_\omega/J_{B_\omega}$ is a unital $^*$-homomorphism and $\pi(1_A)$ is not a projection. Then there exists $g \in C_c((0,1))_+$ such that $g(\pi(1_A)) \neq 0$.
Choose a function $k \in C_0((0,1])_+$ which is orthogonal to $g$ and for which $k(1)=1$.
Set 
\begin{equation} e:=g(\pi(1_A)),\ f:=k(\pi(1_A)) \in B_\omega \cap \pi(A)' \cap \mathrm{her}(\pi(1_A)). \end{equation}
Since $\bar{\pi}(1_A)=1_{B_\omega/J_{B_{\omega}}}$, we see that $e \in J_{B_\omega}$ whereas $\tau_\omega(f^n)=k(1)^n=1$ for all $n\in\N$ and the unique trace $\tau_\omega$ on $B_\omega=\Z_\omega$.
However,
\begin{equation} (B_\omega \cap \pi(A)')f(B_\omega \cap \pi(A)') = f^{1/2}(B_\omega \cap \pi(A)')f^{1/2} \subseteq \{e\}^\perp, \end{equation}
so that $e$ is not even in the ideal of $B_\omega \cap \pi(A)'$ generated by $f$, and in particular, $s$ satisfying \eqref{eq:relSIToShow} cannot exist.
\end{remark}

We now commence the proof of Lemma \ref{lem:relSI}. Previous (SI) results have been set up in the case when $\pi$ is a unital $^*$-homomorphism and $A$ is simple. To reach beyond the simple case, we generalize the 1-step elementary maps used in the treatment of property (SI) in \cite{KR:Crelle}. The first step is to simultaneously excise multiple pure states in an almost-orthogonal way. We do this with the following result, which uses the quasicentral approximate unit argument from \cite[Theorem 3.3]{KR:Crelle}.

\begin{lemma}
\label{lem:StrongStarSurjectivity}
Let $A$ be a $\mathrm C^*$-algebra and suppose that $c^{(1)},\dots,c^{(n)}$ are pairwise orthogonal positive contractions in the centre of $A^{**}$.
Then there exist nets $(e^{(i)}_j)_j$ of positive contractions in $A$ (indexed by the same directed set) such that
\begin{enumerate}
\item $\lim_j\|[e^{(i)}_j,x]\|=0$ for all $x\in A$, $i=1,\dots,n$;
\item $\lim_j\|e^{(i_1)}_je^{(i_2)}_j\|=0$ for $i_1\neq i_2$; and
\item $e^{(i)}_j\rightarrow c^{(i)}$ in strong$^*$-topology.
\end{enumerate}
\end{lemma}

\begin{proof}
It suffices to prove that: given $\eps >0$, a finite subset $\F$ of $A$, and a strong$^*$-neighbourhood $U$ of $0$ in $A^{**}$, there exist positive contractions $e^{(1)},\dots,e^{(n)} \in A$ such that
\begin{enumerate}
\item $\|[e^{(i)},x]\|<\eps$ for all $x\in \F$, $i=1,\dots,n$;
\item $\|e^{(i_1)}e^{(i_2)}\|<\eps$ for $i_1\neq i_2$; and
\item $e^{(i)} - c^{(i)} \in U$ for $i=1,\dots,n$.
\end{enumerate}

Let $\mathcal V$ be the net of strong$^*$-neighbourhoods of $0$ in $A^{**}$ directed by reverse inclusion.  As in \cite[Lemma 1.1]{HKW:Adv}, define
\begin{equation}
 D := \{(x_V)_{V \in \mathcal V} \in \ell^\infty(\mathcal V,A^{**}) \mid \mathrm{strong}^*\text{-}\lim_{V} x_V\text{ exists}\}, 
\end{equation}
and $J:=\{(x_V)_{V\in\mathcal V}\in D\mid \mathrm{strong}^*\text{-}\lim_{V} x_V=0\}$ which is a closed 2-sided ideal in $D$ as multiplication is jointly strong$^*$-continuous on bounded sets.  Let $D_0$ be the subalgebra of those $(x_V)_V\in D$ such that each $x_V\in A$.

By the Kaplansky density theorem there exist positive contractions of the form $f^{(i)} = (e^{(i)}_V)_{V\in\mathcal V} \in D_0$ (so that each $e^{(i)}_V$ is a contraction) such that $e^{(i)}_V - c^{(i)} \in V$ for each $V\in \mathcal V$. 
Thus, $c^{(i)}$ is the strong$^*$-limit of $(e^{(i)}_V)$, i.e., $f^{(i)} - c^{(i)} \in J$, where we are identifying $c^{(i)}$ with its canonical image in $D$.
Since $c^{(i)}$ is central in $D$, for $x\in \mathcal F$,
\begin{equation} [f^{(i)},x] \equiv [c^{(i)},x] \equiv 0,\quad \mod J, \end{equation}
i.e., $[f^{(i)},x] \in J \cap D_0$.
Likewise, since $c^{(1)},\dots,c^{(n)}$ are orthogonal, $f^{(i_1)}f^{(i_2)}\in J \cap D_0$ for $i_1\neq i_2$.

Using a quasicentral approximate unit for the ideal $J \cap D_0$ of $D_0$, there exists a positive contraction $h=(h_V)_V\in J \cap D_0$ such that
\begin{equation}\label{lem:StrongStarSurjectivity.e1}
hy \approx_{\eps/2} y
\end{equation}
for all $y = [f^{(i)},x]$ where $x\in\mathcal F$ and $i=1,\dots,n$, and for all $y=f^{(i_1)}f^{(i_2)}$ where $i_1\neq i_2$, and
\begin{equation}\label{lem:StrongStarSurjectivity.e2}
\|[h,z]\|<\eps/4
\end{equation}
for $z \in \mathcal F \cup \{f^{(1)},\dots,f^{(n)}\}$.
Since $h \in J$, $(1_{A^\sim}-h)f^{(i)}(1_{A^\sim}-h) - c^{(i)} \in J$.
By definition of $J$, this means that, eventually, $(1_{A^\sim}-h_V)e^{(i)}_V(1_{A^\sim}-h_V) - c^{(i)} \in U$.

Fixing such $V$, we set $e^{(i)}:=(1_{A^\sim}-h_V)e^{(i)}_V(1_{A^\sim}-h_V)\in A_+^1$.
Then, for $x\in\mathcal F$,
\begin{eqnarray}
\notag
[e^{(i)},x] &=& (1_{A^\sim}-h_V)e^{(i)}_V(1_{A^\sim}-h_V) x - x(1-h_V)e^{(i)}_V(1_{A^\sim}-h_V)\\
\notag
&\stackrel{(\ref{lem:StrongStarSurjectivity.e2})}{\approx_{\eps/2}}& (1_{A^\sim}-h_V)[e^{(i)}_V, x](1_{A^\sim}-h_V)\\
&\stackrel{(\ref{lem:StrongStarSurjectivity.e1})}{\approx_{\eps/2}}& 0.
\end{eqnarray}
Likewise, for $i_1\neq i_2$,
\begin{eqnarray}
\notag
e^{(i_1)}e^{(i_2)} &=& (1_{A^\sim}-h_V)e^{(i_1)}_V(1_{A^\sim}-h_V)^2e^{(i_2)}_V(1_{A^\sim}-h_V) \\
\notag
&\stackrel{(\ref{lem:StrongStarSurjectivity.e2})}{\approx_{\eps/2}}& (1_{A^\sim}-h_V)^2e^{(i_1)}_Ve^{(i_2)}_V(1_{A^\sim}-h_V)^2 \\
&\stackrel{(\ref{lem:StrongStarSurjectivity.e1})}{\approx_{\eps/2}}& 0.\qedhere
\end{eqnarray}
\end{proof}

We now strengthen the Akemann-Anderson-Pedersen excision theorem (\cite[Proposition 2.2]{AAP:CJM}) to simultaneously excise multiple pure states by almost-orthogonal elements.\footnote{
Note that this does not give excisors for states given as a convex combination of pure states; indeed \cite[Proposition 2.3]{AAP:CJM} shows that a state can only be excised if it is a limit of pure states.
To see what happens if we try to use this lemma to excise a convex combination, consider $\phi = \frac12(\lambda_1 + \lambda_2)$ where $\lambda_1,\lambda_2$ are pure states on $A$ whose associated representations are pairwise inequivalent.
With $a_i$ as in the lemma (where $\mathcal F$ includes $1_A$), suppose that some combination $a=\mu_1a_1+\mu_2a_2$ with $0\leq \mu_1,\mu_2\leq 1$ approximately excises $\phi$, in that $axa\approx\phi(x)a^2$ and $\phi(a)\approx1$ for $x\in\mathcal F$ (and suitable tolerances which we do not keep track of here). A routine, but somewhat messy, computation then shows that 
$\lambda_1$ and $\lambda_2$ approximately agree on $\mathcal F$, i.e., $\lambda_1$ and $\lambda_2$ are weak$^*$-close (relative to $\mathcal F$).
}

\begin{lemma}
\label{lem:OrthogExcision}
Let $A$ be a $\mathrm C^*$-algebra and let $\lambda_1,\dots,\lambda_n$ be pure states whose associated irreducible representations are pairwise inequivalent.
Then given a finite subset $\F$ of $A$ and $\eps>0$ there exist positive contractions $a_1,\dots,a_n \in A$ such that
\begin{enumerate}[(i)]
\item $\lambda_i(a_i)=1$, $i=1,\dots, n$;\label{lem:OrthogExcision.1}
\item $a_ixa_i \approx_\eps \lambda_i(x)a_i^2$, $i=1,\dots,n$, $x\in\F$;\label{lem:OrthogExcision.2}
\item $a_ixa_j \approx_\eps 0$, $i\neq j$, $x\in\F$.\label{lem:OrthogExcision.3}
\end{enumerate}
\end{lemma}

\begin{proof}
Assume, without loss of generality, that $\F$ consists of contractions. Let $\pi_1,\dots,\pi_n$ be the irreducible representations corresponding to $\lambda_1,\dots,\lambda_n$.
These extend to normal representations $\bar{\pi}_1,\dots,\bar{\pi}_n$ of $A^{**}$ whose kernels take the form $p_iA^{**}$ for central projections $p_i\in A^{**}$ satisfying $(1_{A^{**}}-p_i)(1_{A^{**}}-p_j)=0$ for $i\neq j$ (see \cite[Proposition 1.4.3]{BO:Book}, for example). Write $c_i:=1_{A^{**}}-p_i$.

Fix $\eta>0$ so that $2(\eta+(2\eta)^{1/2})<\eps$. By Lemma \ref{lem:StrongStarSurjectivity}, we can find positive contractions $e_1,\dots,e_n$ in $A$ with 
\begin{align}
\label{eq:OrthogExcisioneiOrthog}
\|e_ie_j\|&<\eta, \quad i\neq j, \\
\label{eq:OrthogExcisioneiCentral}
\|[e_i,x]\|&<\eta, \quad x\in\F,
\end{align}
and $\lambda_i(e_i)>1-\eta$ (as $\lambda_i(c_i)=1$).

For each $i$, the Akemann-Anderson-Pedersen excision theorem (\cite[Proposition 2.2]{AAP:CJM}) gives a positive contraction $a_i \in A$ such that $\lambda_i(a_i)=1$ and
\begin{equation}\label{lem:OrthogExcision.saw1}
a_ixa_i \approx_{\eta} \lambda_i(x)a_i^2, \quad x \in \mathcal F \cup \{e_i\},
\end{equation}
giving (\ref{lem:OrthogExcision.1}) and (\ref{lem:OrthogExcision.2}). Since $\lambda_i(e_i) \approx_{\eta} 1$, (\ref{lem:OrthogExcision.saw1}) gives 
\begin{equation}
\|a_i(1_{A^\sim}-e_i)\|^2 \leq \|a_i(1_{A^\sim}-e_i)^{1/2}\|^2 = \|a_i(1_{A^\sim}-e_i)a_i\| < 2\eta
\end{equation} and so 
\begin{equation}
\label{eq:OrthogExcisioneiUnit}
 a_i \approx_{(2\eta)^{1/2}} a_ie_i.
\end{equation}
We compute, for $x\in\mathcal F$,
\begin{eqnarray}
\notag
a_ixa_j &\stackrel{\eqref{eq:OrthogExcisioneiUnit}}{\approx_{2(2\eta)^{1/2}}}& a_ie_ixe_ja_j \\
\notag
&\stackrel{\eqref{eq:OrthogExcisioneiCentral}}{\approx_\eta}& a_ixe_ie_ja_j \\
&\stackrel{\eqref{eq:OrthogExcisioneiOrthog}}{\approx_\eta}& 0.
\end{eqnarray}
The choice of $\eta$ ensures that (\ref{lem:OrthogExcision.3}) holds.
\end{proof}

To prove property (SI) in \cite[Lemma 3.1]{MS:Acta}, Matui and YS approximate the identity map on a simple, infinite dimensional, nuclear  $\mathrm C^*$-algebra by so-called one-step elementary maps (see \cite[Section 5]{KR:Crelle}) arising from a single pure state. The next lemma provides an extension of this result to nonsimple, nuclear $\mathrm C^*$-algebras at the price of allowing multiple inequivalent pure states to appear.  We require a technical hypothesis on the irreducible representations of $A$.

\begin{lemma}
\label{lem:StrongKR}
Let $A$ be a separable, unital, nuclear $\mathrm C^*$-algebra such that $\theta(A)$ contains no compact operators, where $\theta$ is the faithful representation obtained as the direct sum of one irreducible representation from each unitary equivalence class of irreducible representations of $A$.  
Let $\F$ be a finite subset of $A$ and $\eps > 0$.
Then there exist $L,N\in\N$, pairwise inequivalent pure states $\lambda_1,\dots,\lambda_L$ on $A$ and elements $c_{i},d_{i,l} \in A$ for $i=1,\dots,N,\ l=1,\dots,L$ such that
\begin{equation}\label{lem:StrongKR.E1}
x \approx_\eps \sum_{l=1}^L\sum_{i,j=1}^N \lambda_l(d_{i,l}^*xd_{j,l})c_{i}^*c_{j}
\end{equation}
for $x\in\F$.
\end{lemma}

\begin{proof}
By hypothesis, 
$\theta$ is a faithful representation of $A$ containing no compacts.
We identify $A$ with $\theta(A)$ and let the underlying Hilbert space for $\theta$ be denoted $\mathcal H$. Then $\mathcal H$ and $\theta$ split as the direct sum $\bigoplus_{\gamma\in \Gamma}\mathcal H_\gamma$ and $\bigoplus_{\gamma\in \Gamma}\theta_\gamma$ respectively, where $\Gamma$ indexes equivalence classes of irreducible representations of $A$, and we fix one representative $\theta_\gamma$ from each class. 

By the completely positive approximation property, there exists $n\in\mathbb N$, and completely positive maps $\phi:A\rightarrow M_n$, $\rho:M_n\rightarrow A$ such that $\rho(\phi(x))\approx_{\eps/2}x$ for $x\in\F$. As in the proof of \cite[Lemma 5.9]{KR:Crelle}, we can decompose $\rho$ as the composition $M_n\rightarrow M_n\otimes M_n\rightarrow A$ so that $\rho(y)=C^*(y\otimes 1_n)C$ for some column matrix $C$ in $M_{n^2,1}(A)$.\footnote{This is not related to the algebra $C$ in Definition \ref{def:SI}; we are using this notation for comparison with the proof of \cite[Lemma 5.9]{KR:Crelle}.}  Set $N:=n^2$, let $C=(c_1,\dots,c_N)^{\mathrm{T}}$, and let $\tilde{\phi}:A\rightarrow M_n\otimes M_n\cong M_N$ be given by $\tilde{\phi}(x):=\phi(x)\otimes 1_n$.  In this way,
\begin{equation}\label{lem:StrongKR.E2}
\sum_{i,j=1}^N\tilde{\phi}(x)_{i,j}c_{i}^*c_{j}\approx_{\eps/2}x,\quad x\in\F.
\end{equation}

Set $\eta:=\frac\eps{2N^2}(\max_{i,j}\|c_{i}^*c_{j}\|)^{-1}$.
Using the hypothesis that $\theta(A)$ has zero intersection with the compact operators, \cite[Theorem 2.5]{BS:DMJ} tells us that there exists an operator $V:\mathbb C^N\rightarrow\mathcal H$ with $\|\tilde{\phi}(x)-V^*xV\|<\eta$ for $x\in\mathcal F$.
Let $e_1,\dots,e_N$ be the usual basis for $\mathbb C^N$ so that $|\tilde{\phi}(x)_{i,j}-\langle xVe_i,Ve_j\rangle|<\eta$ for all $i,j=1,\dots,N$.
We can approximate each $Ve_i$ by a finite direct sum $\oplus_{l=1}^{L}\xi_{i,l}$ in $\oplus_{l=1}^{L}\mathcal H_{\gamma_l}\subset\mathcal H$ for some $\gamma_1,\dots,\gamma_L\in\Gamma$, so that 
\begin{equation}\label{lem:StrongKR.E3}
\tilde{\phi}(x)_{i,j}\approx_{\eta}\sum_{l=1}^{L}\langle \theta_{\gamma_l}(x)\xi_{i,l},\xi_{j,l}\rangle,\quad x\in\F,\ i,j=1,\dots,N.
\end{equation}
Then, for each $l$, choose a pure state $\lambda_l$ inducing $\theta_{\gamma_l}$, and let $\xi_{0,l}\in\mathcal H_{\gamma_l}$ be a GNS-vector corresponding to this pure state. By Kadison's transitivity theorem (\cite{K:PNAS}), we can find $d_{i,l}\in A$ such that $\theta_{\gamma_l}(d_{i,l})\xi_{0,l}=\xi_{i,l}$.
In this way $\langle\theta_{\gamma_l}(x)\xi_{i,l},\xi_{j,l}\rangle=\lambda_l(d_{j,l}^*xd_{i,l})$.  Combining this with (\ref{lem:StrongKR.E2}) and (\ref{lem:StrongKR.E3}) gives (\ref{lem:StrongKR.E1}).
\end{proof}

We need another technical ingredient to handle the nonsimple version of property (SI).  Recall that $a \vartriangleleft b$ means that $b$ acts as a unit on $a$.

\begin{lemma}
\label{lem:relSItrick}
Let $(B_n)_{n=1}^\infty$ be a sequence of simple, separable, unital, stably finite $\mathrm C^*$-algebras with strict comparison of positive elements, set $B_\omega := \prod_\omega B_n$ and define $J_{B_\omega}$ as in \eqref{eq:JBdef}.
Let $A$ be a separable, unital $\mathrm C^*$-algebra and let $\pi:A\rightarrow B_\omega$ be a c.p.c.\ order zero map.
Let $a \in A_+$ have norm at least $1$. 
Then there exists a countable set $S \subset A_+\setminus \{0\}$ such that the following holds.
If $e,t,h \in (B_\omega \cap \pi(A)' \cap \{1_{B_\omega}-\pi(1_A)\}^\perp)_+$ are contractions such that
\begin{equation} e \in J_{B_\omega}\quad \text{and} \quad h \vartriangleleft t,
\end{equation}
and if for all $b \in S$, there exists $\gamma_b > 0$ such that
\begin{equation}
\label{eq:relSItrickTr}
\tau(\pi(b)h) > \gamma_b,\quad \tau\in T_\omega(B_\omega), \end{equation}
then there exists a contraction $r \in B_\omega$ such that
\begin{equation}
\label{eq:relSItrickrDef}
 \pi(a)r=tr=r \quad \text{and} \quad r^*r=e. \end{equation}
\end{lemma}

\begin{proof}
Since $\|a\|\geq 1$, we may use functional calculus to produce a countable set $S \subset C^*(a)_+ \setminus \{0\}$ such that, for every $\eps > 0$, there exists $a_0 \in S$ and $a_1 \in C^*(a)_+$ such that $a_0 \vartriangleleft a_1$ and $a_1 \approx_\eps a$.
This is our set $S$; let us now suppose that \eqref{eq:relSItrickTr} holds for all $b\in S$.

For $\eps>0$, we will produce a contraction $r\in B_\omega$ satisfying $r^*r\approx_\eps e$, $r\approx_\eps \pi(a)r$ and $tr=r$.  An easy application of Kirchberg's $\eps$-test (Lemma \ref{epstest}) then yields $r$ satisfying (\ref{eq:relSItrickrDef}) exactly.

Let $a_0 \in S, a_1 \in C^*(a)_+$ be such that $a_0 \vartriangleleft a_1$ and $a_1 \approx_\eps a$. By hypothesis, $d_\tau(e)=0$ and $\tau(h^{1/2}\pi(a_0)h^{1/2}) \geq \gamma_{a_0}$ for all $\tau\in T_\omega(B_\omega)$, and hence also for $\tau \in \overline{T_\omega(B_\omega)}$ (the weak$^{*}$-closure), whence 
\begin{equation}
d_\tau(e)<\gamma_{a_0}\leq \tau(h^{1/2}\pi(a_0)h^{1/2})\leq d_\tau(h^{1/2}\pi(a_0)h^{1/2}),\quad \tau\in \overline{T_\omega(B_\omega)}.
\end{equation}
By Lemma \ref{lem:StrictCompLimTraces}, $B_\omega$ has strict comparison by limit traces, whence $e\preceq h^{1/2}\pi(a_0)h^{1/2}$ in $B_\omega$.  By \cite[Proposition 2.4]{R:JFA2}, there exists $r\in B_\omega$ and $\dl>0$ such that $r^*r=(e-\eps)_+$ (so that $r$ is contractive) and $g_{\dl}(h^{1/2}\pi(a_0)h^{1/2})r=r$. As $h\vartriangleleft t$, it follows that $r=tr$.  

By Lemma \ref{lem:SupportingMap}, we may fix a supporting c.p.c.\ order zero map $\hat{\pi}:A\rightarrow B_\omega\cap\{h,t\}'$ for $\pi$.  We have $g_{\dl}(h^{1/2}\pi(a_0)h^{1/2})\vartriangleleft \hat{\pi}(a_1)$ as $a_0\vartriangleleft a_1$ and hence also $\pi(a_0)\vartriangleleft \hat{\pi}(a_1)$. In this way $\hat{\pi}(a_1)r=r$, and so
\begin{equation} r = \hat{\pi}(a_1)r=\hat{\pi}(a_1)tr \approx_\eps \hat{\pi}(a)tr \stackrel{t\vartriangleleft \pi(1_A)}{=} \hat{\pi}(a)\pi(1_A)tr\stackrel{(\ref{eq:Supporting})}{=}\pi(a)tr=\pi(a)r. \end{equation}
We have $r^*r\approx_\eps e$, $r\approx_\eps \pi(a)r$ and $tr=r$, as required.\end{proof}

A calculation needed in the proof of property (SI) is extracted in the following lemma, for subsequent reuse in Section \ref{sec:KirAlgs}.

\begin{lemma}\label{lem:CommLem}
Let $A$ and $B$ be unital $\mathrm C^*$-algebras and $\pi:A\rightarrow B$ be a c.p.c.\ order zero map. Suppose that $s\in B$ satisfies 
\begin{equation}
\label{eq:CommLemHyp}
s^*\pi(a)s=\pi(a)s^*s, \quad a\in A
\end{equation}
and $s^*s\vartriangleleft \pi(1_A)$. Then $s\in B\cap \pi(A)'\cap \{1_{B}-\pi(1_A)\}^\perp$.
\end{lemma}

\begin{proof}
Note that $s^*s\geq s^*\pi(1_A)s=\pi(1_A)s^*s=s^*s$, so that by the $\mathrm C^*$-identity,
\begin{equation}
\|\pi(1_A)^{1/2}s-s\|^2 = \|s^*s-s^*\pi(1_A)s\| = 0,
\end{equation}
which implies that $s=\pi(1_A)^{1/2}s=\pi(1_A)s$.
Then, for $a=a^*\in A$, we have
\begin{eqnarray}
\notag
s^*\pi(a)^2s&\stackrel{\eqref{eq:Ord0Ident}}=&s^*\pi(a^2)\pi(1_A)s\notag\\
\notag&\stackrel{\pi(1_A)s=s}=&s^*\pi(a^2)s\\
\notag
&\stackrel{\eqref{eq:CommLemHyp}}=&\pi(a^2)s^*s\\
\notag
&\stackrel{s^*s\vartriangleleft \pi(1_A)}=& \pi(a^2)\pi(1_A)s^*s\\
&\stackrel{\eqref{eq:Ord0Ident}}=& \pi(a)^2s^*s.
\end{eqnarray}
Then, for $a=a^*\in A$, in the expansion of $[s,\pi(a)]^*[s,\pi(a)]$, all the terms cancel and so by the $\mathrm C^*$-identity, $[s,\pi(a)]=0$. Thus also $s\pi(1_A)=s$ and so $s\in B\cap \pi(A)'\cap \{1_{B}-\pi(1_A)\}^\perp$.
\end{proof}

\begin{proof}[Proof of Lemma \ref{lem:relSI}]
Let $A$ and $B_n$ be as in Lemma \ref{lem:relSI} and let $\pi:A\rightarrow B_\omega$ be a c.p.c.\ order map.
Recall that
\begin{equation}
C:=B_\omega \cap \pi(A)' \cap \{1_{B_\omega}-\pi(1_A)\}^\perp, \quad \bar{C}:= C/(C\cap J_{B_\omega}).
\end{equation}
Let $e,f \in C_+$ satisfy the conditions in the definition of property (SI) so that $e\in J_{B_\omega}$ and $\|f\|=1$ has the property that for every nonzero $a \in A_+$, there exists $\gamma_a > 0$ such that 
\begin{equation}
\label{eq:relSItr-inproof}
\tau(\pi(a)f^n) > \gamma_a,\quad \tau\in T_\omega(B_\omega),\ n\in\mathbb N. \end{equation}
We will show that there exists $s \in B_\omega$ that satisfies
\begin{equation}\label{ne4.30}
s^*\pi(a)s = \pi(a)e, \quad \text{for all }a \in A\text{ and}\quad fs=s.
\end{equation}
It will then follow that $\pi(1_A)s=\pi(1_A)fs=s$.
Then, taking $a:=1_A$ gives $s^*s=\pi(1_A)e=e\vartriangleleft\pi(1_A)$. 
By Lemma \ref{lem:CommLem}, $s\in B_\omega\cap \pi(A)'\cap \{1_{B_\omega}-\pi(1_A)\}^\perp=C$.

Fix a finite set $\mathcal F$ of contractions in $A$ and $\eps > 0$; we will produce a contraction $s \in B_\omega$ that satisfies
\begin{equation}
\label{eq:relSIToShow2}
 s^*\pi(x)s \approx_{3\eps} \pi(x)e \quad\text{for all }x \in \mathcal F \text{ and}\quad fs=s.
\end{equation}
Once this is done, a routine application of Kirchberg's $\eps$-test (Lemma \ref{epstest}), will give a contraction $s\in B_\omega$ satisfying (\ref{ne4.30}).

Since each $B_n$ is $\mathcal Z$-stable, using Lemma \ref{lem:Zfacts}(\ref{lem:Zfacts1}) we can find a unital $^*$-homo\-morphism $\alpha:\Z\to B_\omega\cap \pi(A)'\cap\{e,f\}'$.  Thus we can define a c.p.c.\ order zero map $\tilde{\pi}:A\otimes\Z\to B_\omega$ satisfying $\pi(a)=\tilde\pi(a \otimes 1_{\mathcal Z})$ for $a\in A$ by $\tilde{\pi}(a\otimes z)=\pi(a)\alpha(z)$. This is c.p.c.\ order zero by \cite[Corollary 4.3]{WZ:MJM} as it is the tensor product of $\pi$ and $\alpha$. Note that $e,f\in B_\omega\cap \tilde{\pi}(A\otimes\Z)'\cap\{1_{B_\omega}-\tilde{\pi}(1_{A\otimes\Z})\}^\perp$. 

We claim that for every nonzero $b\in (A\otimes \Z)_+$, there exists $\tilde{\gamma}_b>0$ such that
\begin{equation}\label{eq:relSItr-tilde}
\tau(\tilde{\pi}(b)f^n)>\tilde{\gamma}_b,\quad \tau\in T_\omega(B_\omega),\ n \in \mathbb{N}.
\end{equation}
When $b=a\otimes z$ for some nonzero $a\in A_+$ and $z\in\Z_+$, (\ref{eq:relSItr-tilde}) is a consequence of (\ref{eq:relSItr-inproof}) as we can find $y_1,\dots,y_m\in\Z$ with $\sum_{i=1}^m y_izy_i^*=1_\Z$. Write $K=\|\sum_{i=1}^m y_iy_i^*\|>0$, and compute
\begin{align}
K\tau(\pi(a)\alpha(z)f^n)&\geq\tau\Big(\pi(a)f^n\alpha\Big(z^{1/2}\Big(\sum_{i=1}^m y_iy_i^*\Big)z^{1/2}\Big)\Big)\notag\\
&=\sum_{i=1}^m\tau(\pi(a)\alpha(y_i^*zy_i)f^n)\notag\\
&=\sum_{i=1}^m\tau(\tilde{\pi}(a\otimes y_i^*zy_i)f^n)\notag\\
&=\tau(\pi(a)f^n)\geq \gamma_a,\quad \tau\in T_\omega(B_\omega),\ n\in\mathbb N,
\end{align}
so that (\ref{eq:relSItr-tilde}) holds with $\tilde{\gamma}_b=K^{-1}\gamma_a$ (where $\gamma_a$ is as given in (\ref{eq:relSItr-inproof})).  For general nonzero $b\in (A\otimes \Z)_+$, use Kirchberg's slice lemma (see \cite[Lemma 4.1.9]{R:Book}), to find a nonzero element $c\in A\otimes\Z$ such that $c^*c=a\otimes z$ for some $a\in A_+$ and $z\in \mathcal{Z}_+$ and $cc^*\in\overline{b^{1/2}(A\otimes \Z)b^{1/2}}$.  Then, as $f\in \tilde{\pi}(A\otimes\Z)'$,
\begin{equation}
\tau(\tilde{\pi}(cc^*)f^n)=\tau(\tilde{\pi}(c^*c)f^n)\geq \tilde{\gamma}_{c^*c},\quad \tau\in T_\omega(B_\omega),\ n\in\mathbb N
\end{equation}
by the calculation above.  

Take a nonzero $x\in (A\otimes\Z)_+$ such that $b^{1/2}xb^{1/2}\approx_{\frac{\tilde{\gamma}_{c^*c}}{2}}cc^*$ so that
\begin{equation}
\tau(\tilde{\pi}(b^{1/2}xb^{1/2})f^n)\geq \frac{\tilde{\gamma}_{c^*c}}{2},\quad \tau\in T_\omega(B_\omega),\ n\in\mathbb N.
\end{equation}
Therefore
\begin{equation}
\tau(\tilde{\pi}(b)f^n)\geq\|x\|^{-1}\frac{\tilde{\gamma}_{c^*c}}{2},\quad \tau\in T_\omega(B_\omega),\ n\in\mathbb N,
\end{equation}
establishing (\ref{eq:relSItr-tilde}).

Write $\mathcal G=\{x\otimes 1_\Z:x\in\F\}\subset A\otimes\Z$. Note that no irreducible representation of $A \otimes \mathcal Z$ contains any compact operators.
 So, by Lemma \ref{lem:StrongKR}, there exist $L,N\in\N$, pairwise inequivalent pure states $\lambda_1,\dots,\lambda_L$ on $A\otimes \mathcal Z$ and elements $c_{i},d_{i,l} \in A\otimes \mathcal Z$ for $i=1,\dots,N$, $l=1,\dots,L$ such that
\begin{equation}
\label{eq:relSI1Step}
x \approx_\eps \sum_{l=1}^L\sum_{i,j=1}^N\lambda_l(d_{i,l}^*xd_{j,l})c_{i}^*c_{j},\quad x\in\mathcal G.
\end{equation}

By our orthogonal excision lemma, Lemma \ref{lem:OrthogExcision} applied to the finite set $\{d_{i,l}^*xd_{j,l'}:x\in\mathcal G,\ i,j=1,\dots,N,\ l,l'=1,\dots,L\}$ there exist positive contractions $a_1,\dots,a_L\in (A\otimes \mathcal Z)_+$ such that for $l=1,\dots,L$, $\lambda_l(a_l)= 1$ and
\begin{equation}
\label{eq:relSIExcision} a_ld_{i,l}^*xd_{j,l}a_l \approx_{\eps/(N^2L\max_k\|c_k\|^2)} \lambda_l(d_{i,l}^*xd_{j,l})a_l^2,\quad x\in\mathcal G,\ i,j=1,\dots,N,\end{equation}
while for $l\neq l'$,
\begin{equation}
\label{eq:relSIOrthog}
a_ld_{i,l}^*xd_{j,l'}a_{l'}\approx_{\eps/(N^{2}L^2\max_k\|c_k\|^{2})} 0,\quad x\in\mathcal G,\ i,j=1,\dots,N.
\end{equation}

By Lemma \ref{lem:SupportingMap}, let $\hat{\pi}:A\otimes \mathcal Z\rightarrow B_\omega\cap \{f\}'$ be a supporting c.p.c.\ order zero map for $\tilde\pi$.
For each $l=1,\dots,L$, let $S_l \subset (A\otimes \mathcal Z)_+ \setminus \{0\}$ be the countable set provided by Lemma \ref{lem:relSItrick} with $\tilde\pi$ in place of $\pi$ and $a_l$ in place of $a$ (recalling from \cite[Corollary 4.6]{R:IJM} that separable, simple, unital, finite and $\mathcal Z$-stable $\mathrm C^*$-algebras whose quasitraces are traces are stably finite and have strict comparison; see Remark \ref{EQTT}).
Using (\ref{eq:relSItr-tilde}), we can apply Lemma \ref{lem:LargeTraceSubordinate} twice (with $x:=0$ and with $S_0:=\tilde\pi(S_1 \cup \cdots \cup S_L)$) to obtain $t,h \in B_\omega \cap \hat\pi(A\otimes \mathcal Z)'\cap \tilde\pi(A\otimes \mathcal Z)'$ such that $h \vartriangleleft t \vartriangleleft f$ and, for every $b\in S_1\cup \cdots \cup S_L$, 
\begin{equation}
 \tau(\tilde\pi(b)h^n) \geq \tilde{\gamma}_b, \quad \tau \in T_\omega(B_\omega),\ n\in\mathbb N. \end{equation}
In particular this holds for $n=1$, so for each $l=1,\dots,L$, by the definition of $S_l$ in the application of Lemma \ref{lem:relSItrick} (with $\tilde\pi$ in place of $\pi$), there exists a contraction $r_l \in B_\omega$ such that $\tilde\pi(a_l)r_l=tr_l=r_l$ and $r_l^*r_l=e$.
Since $t \vartriangleleft f\vartriangleleft \tilde\pi(1_A)$, it follows that $\tilde\pi(1_A)r_l=r_l$ for each $l$, and therefore,
\begin{eqnarray}
\label{eq:relSIrDef}
\notag
r_l^*\tilde\pi(a_l^2)r_l &=& r_l^*\tilde\pi(a_l^2)\tilde\pi(1_A)r_l \\
\notag
&\stackrel{(\ref{eq:Ord0Ident})}{=}& r_l^*\tilde\pi(a_l)^2r_l \\
&=& r_l^*r_l =e=\tilde\pi(1_A)^{1/2}e\tilde\pi(1_A)^{1/2}.
\end{eqnarray}
Set
\begin{equation}
\label{eq:relSIsDef}
 s := \sum_{l=1}^L\sum_{i=1}^N \hat{\pi}(d_{i,l}a_l)r_l \hat{\pi}(c_i) \in B_\omega.
\end{equation}
Since $r_l= tr_l$, $t \vartriangleleft f$ and $t$ commutes with the image of $\hat\pi$, it follows that $fs=s$.
For $x\in \mathcal F$, \begin{eqnarray}
\notag
\lefteqn{ s^*\pi(x)s }\\
\notag
&=& s^*\tilde\pi(x\otimes 1_\Z) s \\
\notag
&\stackrel{\eqref{eq:relSIsDef}}=& \sum_{l,l'=1}^L\sum_{i,j=1}^N \hat{\pi}(c_i^*)r_l^* \hat{\pi}(a_ld_{i,l}^*)\tilde\pi(x\otimes 1_\Z)\hat{\pi}(d_{j,l'}a_{l'})r_{l'}\hat{\pi}(c_j) \\
\notag
&\stackrel{r_l \vartriangleleft \pi(1_A), \eqref{eq:Supporting}}=&  \sum_{l,l'=1}^L\sum_{i,j=1}^N\hat{\pi}(c_i^*)r_l^*\tilde\pi(a_ld_{i,l}^*(x\otimes 1_\Z)d_{j,l'}a_{l'})r_{l'}\hat{\pi}(c_j) \\
\notag
&\stackrel{\eqref{eq:relSIOrthog}}{\approx_\eps}& \sum_{l=1}^L \sum_{i,j=1}^N \hat{\pi}(c_i^*)r_l^*\tilde\pi(a_ld_{i,l}^*(x\otimes 1_\Z)d_{j,l}a_{l})r_{l}\hat{\pi}(c_j) \\
\notag
&\stackrel{\eqref{eq:relSIExcision}}{\approx_\eps}& \sum_{l=1}^L \sum_{i,j=1}^N \lambda_l(d_{i,l}^*(x\otimes 1_\Z)d_{j,l})\hat{\pi}(c_i^*)r_l^*\tilde\pi(a_l^2)r_{l}\hat{\pi}(c_j) \\
\notag
&\stackrel{\eqref{eq:relSIrDef}}{=}& \sum_{l=1}^L \sum_{i,j=1}^N \lambda_l(d_{i,l}^*(x\otimes 1_\Z)d_{j,l})\hat{\pi}(c_i^*)\tilde\pi(1_A)^{1/2}e\tilde\pi(1_A)^{1/2}\hat{\pi}(c_j) \\
\notag
&\stackrel{\eqref{eq:Supporting}}=& \sum_{l=1}^L\sum_{i,j=1}^N \lambda_l(d_{i,l}^*(x\otimes 1_\Z)d_{j,l}) \tilde\pi^{1/2}(c_i^*)e\tilde\pi^{1/2}(c_j) \\
\notag
&\stackrel{e\in \tilde{\pi}(A\otimes\Z)',\ (\ref{eq:Ord0Ident})}=& \sum_{l=1}^L \sum_{i,j=1}^N \lambda_l(d_{i,l}^*(x\otimes 1_\Z)d_{j,l}) \tilde\pi(c_i^*c_j)e \\
&\stackrel{\eqref{eq:relSI1Step}}{\approx_\eps}& \tilde\pi(x\otimes 1_\Z)e = \pi(x)e,
\end{eqnarray}
as required.
\end{proof}

\bigskip

\subsection{\sc Proof of Theorem \ref{thm:StrictCompTraceSurj}}\hfill \\

\label{sect5.2}

\noindent
With the property (SI) result in place, we now establish Theorem \ref{thm:StrictCompTraceSurj}.

\begin{proof}[Proof of Theorem \ref{thm:StrictCompTraceSurj}(i)]
This follows from property (SI) essentially as in the proof of \cite[Proposition 3.3]{MS:DMJ}.
For completeness, here is the proof.
Let $\tau \in T(C)$.
Our aim is to show that $\tau(x)=0$ for all positive contractions $x \in C \cap J_{B_\omega}$. Temporarily fix $x\in (C\cap J_{B_\omega})^{1}_+$.

For a nonzero $a \in A_+$, set 
\begin{equation}
\gamma_a := \min_{\eta \in T(B_\omega)} \eta(\pi(a)),
\end{equation}
and note that for each $n\in\N$,
\begin{equation}
\gamma_{a} = \min_{\eta\in T(B_\omega)}\eta(\pi(1_A)^n\pi(a))> 0,
\end{equation}
since $\bar{\pi}$ is a $^*$-homomorphism and $\pi(a)$ is full in $B_\omega$.
Let $S$ be a countable dense subset of $A_+ \setminus \{0\}$.
By Lemma \ref{lem:LargeTraceSubordinate} (applied to $f:=\pi(1_A)$ and with $S_0:=S$), 
there exists a positive contraction $f'\in C$ with $x\vartriangleleft f'$, $\eta((f')^n\pi(a))\geq \gamma_a$ for all $n\in\N, \eta\in T(B_\omega)$ and $a \in S$.\footnote{The hypotheses on each $B_n$ ensure that this holds on $T(B_\omega)$ and not just $T_\omega(B_\omega)$ by Proposition \ref{NoSillyTraces}.}
Since $a \mapsto \gamma_a$ is continuous and $S$ is dense in $A_+$,
\begin{equation} \eta((f')^n\pi(a)) \geq \gamma_a, \quad a \in A_+ \setminus \{0\}. \end{equation}
Now $f'-x\in C_{+}^1$ and this difference satisfies
\begin{equation} \eta(\pi(a)(f'-x)^n) = \eta(\pi(a)(f')^n) \geq \gamma_a,\quad \eta\in T(B_\omega),\ a\in A_+,\ a\neq 0. \end{equation}
Hence, as $\pi$ has property (SI) (Lemma \ref{lem:relSI}), there exists a contraction $s\in C$ such that $s^*s=x$ and $(f'-x)s=s$, so that $(f'-x)^{1/2}s=s$.  In particular $x+ss^*\leq x+(f'-x)^{1/2}ss^*(f'-x)^{1/2}\leq x+f'-x\leq f'$. Hence $x+ss^*$ is a positive contraction in $C \cap J_{B_\omega}$ with $\tau(x+ss^*)=2\tau(x)$.
This shows that, if
\begin{equation} \alpha := \sup \tau(x) \end{equation}
where the supremum is taken over all positive contractions $x\in C \cap J_{B_\omega}$ then $\alpha \geq 2\alpha$.
Therefore, $\alpha=0$, as required.

This establishes surjectivity of the map $T(\bar{C}) \to T(C)$ induced by the quotient map $C \to \bar{C}$.
It is a general fact that a surjective $^*$-homomorphism induces an affine continuous open injective map between the tracial state spaces.\footnote{Surjectivity is required for (a) the image of a state to be a state, (b) injectivity, and (c) openness.}
Thus, it follows that $T(\bar{C})$ and $T(C)$ are affinely homeomorphic.

Since $\bar{C}$ is unital (Lemma \ref{lem:RelCommSurjectivity}), $T(\bar{C})$ is a Choquet simplex, and therefore so is $T(C)$.
\end{proof}

The proof of the second part of Theorem \ref{thm:StrictCompTraceSurj} uses ideas from \cite[Theorem 4.8]{MS:arXiv}. 

\begin{proof}[Proof of Theorem \ref{thm:StrictCompTraceSurj}(ii)]
Let $b,c \in C_+$ be positive contractions such that
\begin{equation} d_\tau(b) < d_\tau(c),\quad \tau\in T(C). \end{equation}

Let $\eps > 0$.
Then, with $g_\eps$ as defined in \eqref{eq:gepsDef},
\begin{equation} d_\tau((b-\eps)_+) \leq \tau(g_\eps(b)) \leq d_\tau(b) < d_\tau(c),\quad \tau\in T(C). \end{equation}
Since $T(C)$ is compact (by (i)), $\tau \mapsto \tau(g_\eps(b))$ is continuous, and $\tau \mapsto d_\tau(c)$ is the increasing (as $\dl \to 0$) pointwise supremum of continuous functions $\tau \mapsto \tau(g_\dl(c))$, there are $\delta > 0$ and $k\in\mathbb N$ such that $\gamma := \min_{\tau \in T(C)} \tau(g_{2\delta}(c)) > 0$ and
\begin{equation} d_\tau((b-\eps)_+) \leq \tau(g_\eps(b)) 
< \frac{k-1}{k}\tau(g_\dl(c)),
\quad \tau\in T(C). \end{equation}
Set 
\begin{equation} b_0 := (b-\eps)_+ \quad\text{and}\quad c_0:= g_\delta(c), \end{equation}
so that it suffices to show that $(b_0-\eps)_+\preceq c_0$ in $C$. We know that
\begin{align}
\label{eq:StrictCompReducedSetup1}
 d_\tau(b_0) &< \frac{k-1}k \tau(c_0),\quad \tau \in T(C), \text{ and} \\
\label{eq:StrictCompReducedSetup2}
\tau(c_0^n) &\geq \gamma, \quad \tau \in T(C),\ n\in \N.
\end{align}

Let $\M^\omega$ be the ultraproduct of the sequence of tracial continuous $\mathrm{W}^*$-bundles $(\overline{B}^{\mathrm{st}}_n)_{n=1}^\infty$ so that Lemma \ref{lem:CentralSurjectivity} gives $B_\omega/J_{B_\omega}\cong \M^\omega$ and (taking $S:=\emptyset$ in Lemma \ref{lem:CentralSurjectivity}) $\bar{C}=\bar{\pi}(1_{A})(\M^\omega\cap \bar{\pi}(A)')$. In particular, as each $B_n$ is separable and $\Z$-stable, each $\overline{B_n}^{\mathrm{st}}$ is a strictly separable McDuff $\mathrm{W}^*$-bundle over $\partial_{e} T(B_n)$ (by Proposition \ref{TensorBundle}, Proposition \ref{DefMcDuff}, and the fact that $\overline{\Z}^{\mathrm{st}}\cong\R$) so that $\bar{C}$ has strict comparison of positive elements by Lemma \ref{lem:W*StrictComp}. 

Let $\bar{c}$ denote the image of $c$ in $\bar{C}$.
By Remark \ref{McDuffProductEmbedding}, there exists a unital embedding 
\begin{equation} 
\bar{\phi}':M_k \to \M^\omega\cap \bar{\pi}(A)'\cap\{\bar{c}\}'.
 \end{equation}
Cutting this homomorphism by the projection $\bar{\pi}(1_{A})$, gives a unital embedding
\begin{equation}
\bar{\phi}:M_k\to \bar{\pi}(1_{A})\M^\omega\bar{\pi}(1_{A})\cap \bar{\pi}(A)'\cap\{\bar{c}\}'=(C\cap\{c\}')/(J_{B_\omega}\cap C\cap \{c\}'),
\end{equation} 
where the last identity is obtained by another application of Lemma \ref{lem:CentralSurjectivity} (with $S:=\{c\}$). Then, using projectivity of $C_0((0,1],M_k)$ (\cite[Theorem 4.9]{Loring:JFA}), $\bar{\phi}$ lifts to a c.p.c.\ order zero map
\begin{equation} \phi:M_k \to C \cap \{c\}'. \end{equation}

Let $p \in M_k$ be a rank one projection.
Since $\phi(1_{M_{k}}) \equiv \pi(1_{A}) \mod J_{B_\omega}$, Theorem \ref{thm:StrictCompTraceSurj}(i)  gives
\begin{equation}
\tau(\phi(1_{M_{k}})c_0^n) =  \tau(\pi(1_{A})c_0^n) = \tau(c_0^n), \quad n\in\N,\ \tau \in T(C).
\end{equation}
As $\phi(p)$ commutes with $c_0$, and $\phi(p)^n\equiv\phi(p)\mod J_{B_\omega}$, we similarly have
\begin{align}
\label{eq:StrictCompPhiTraceFormula2}
\tau((\phi(p)c_0)^n)= \tau(\phi(p)c_0^n) &= \frac{\tau(c_0^n)}k, \quad n\in\mathbb N,\ \tau\in T(C),\text{ and} \\
\label{eq:StrictCompPhiTraceFormula2.1}
\tau(\phi(1_{M_{k}}-p)c_0) &= \frac{k-1}k \tau(c_0),\quad \tau \in T(C).
\end{align}

In particular, letting $\bar{b}_0$, $\bar{c}_0$ denote the images of $b_0,c_0$ in $\bar{C}=C/(C\cap J_{B_\omega})$, (\ref{eq:StrictCompReducedSetup1}) and (\ref{eq:StrictCompPhiTraceFormula2.1}) give 
\begin{equation}
\label{eq:REQEXP}
d_\tau(\bar{b}_0)<\tau(\bar{\phi}(1_{M_{k}}-p)\bar{c}_0)\leq d_\tau(\bar{\phi}(1_{M_{k}}-p)\bar{c}_0),\quad \tau\in T(\bar{C}).
\end{equation}
Since $\bar{C}$ has strict comparison of positive elements by bounded traces it follows that $\bar{b}_0\preceq \bar{\phi}(1_{M_{k}}-p)\bar{c}_0$ in $\bar{C}$. 
By \cite[Proposition 2.4]{R:JFA2}, there exist $\mu>0$ and $\bar{v} \in \bar{C}$ such that 
\begin{equation} \bar{v}^*\bar{v}= g_{\frac{\eps}{2}}(\bar{b}_0),\quad g_\mu(\bar{\phi}(1_{M_{k}}-p)\bar{c}_0)\bar{v}=\bar{v};
\end{equation}
in particular $\bar{v}$ is a contraction, and, since $\bar C = C/(C \cap J_{B_\omega})$, we can lift $\bar{v}$ to a contraction $v\in C$.
Define a contraction $r:=v(b_0-\eps)_+^{1/2}\in C$ so that 
\begin{align}
r^*g_\mu(\phi(1_{M_{k}}-p)c_0)r&=(b_0-\eps)_+^{1/2}v^*g_\mu(\phi(1_{M_{k}}-p)c_0)v(b_0-\eps)_+^{1/2}\notag\\
&\equiv (b_0-\eps)_+\mod J_{B_\omega}.
\end{align}
Therefore 
\begin{equation}
\label{eq:StrictCompa''Def}
 b_1 := (b_0-\eps)_+ - r^*g_\mu(\phi(1_{M_{k}}-p)c_0)r \in (C \cap J_{B_\omega})_+.
\end{equation}

We will apply property (SI) to $e:=b_1$ and $f:=\phi(p)c_0$, so must check the `largeness' requirement for the latter operator.   Let $a \in A_+$ be nonzero. Then for $\tau \in T(B_\omega)$, $\tau(\pi(a)\cdot)$ is a scalar multiple, with scalar $\tau(\pi(a))$ (by  Lemma \ref{lem:RelCommSurjectivity}(ii)) of a tracial state on $C$, so that
\begin{eqnarray}
\notag
\tau(\pi(a)(\phi(p)c_0)^n) &\stackrel{(\ref{eq:StrictCompPhiTraceFormula2})}{=}& \frac{\tau(\pi(a)c_0^n)}k \\
&\stackrel{\eqref{eq:StrictCompReducedSetup2}}\geq& \frac{\gamma\tau(\pi(a))}k.
\end{eqnarray}
By hypothesis, $\pi(a)$ is full in $B_\omega$, so that there exists $\gamma_a > 0$ such that $\frac{\gamma\tau(\pi(a))}k > \gamma_a$ for all $\tau \in T(B_\omega)$.
It follows that
\begin{equation}
\label{eq:StrictCompPhiTraceFormula3}
\tau(\pi(a)(\phi(p)c_0)^n) > \gamma_a, \quad n\in\N,\ \tau\in T(B_\omega),\ a\in A_+\setminus\{0\}.
\end{equation}

Therefore as $\pi$ has property (SI) (Lemma \ref{lem:relSI}), taking $e:=b_1\in (C\cap J_{B_\omega})_+$ and $f:=\phi(p)c_0\in C_+$, we obtain
\begin{equation}
\label{eq:StrictCompa''Comp}
b_1 \preceq \phi(p)c_0\text{ in }C.
\end{equation}
Thus, we have (using \cite[Lemma 2.10]{APT:Contemp} for the standard Cuntz semigroup fact in the first line),
\begin{eqnarray}
[(b_0-\eps)_+] &\stackrel{\eqref{eq:StrictCompa''Def}}\leq& [r^*g_\mu(\phi(1_{M_{k}}-p)c_0)r] + [b_1] \notag\\
\notag
&\stackrel{\eqref{eq:StrictCompa''Comp}}\leq& [\phi(1_{M_{k}}-p)c_0] + [\phi(p)c_0] \\
\notag
&=& [\phi(1_{M_{k}})c_0]\notag \\
& \leq & [c_0],
\end{eqnarray}
where on the penultimate line, we used the facts that $\phi(1_{M_{k}})$ is the orthogonal sum of $\phi(1_{M_{k}}-p)$ and $\phi(p)$, and that $\phi(M_k)\in \{c_0\}'$. We have thus established $(b_{0}-\eps)_{+} \preceq c_{0}$, as desired.
\end{proof}

\begin{proof}[Proof of Theorem \ref{thm:StrictCompTraceSurj}(iii)]
Again, let $\M^\omega$ be the ultraproduct of the strictly separable, McDuff $\mathrm{W}^*$-bundles $(\overline{B}^{\mathrm{st}}_n)_{n=1}^\infty$ so that $B_\omega/J_{B_\omega}\cong \M^\omega$ and $\bar{C}=\M^\omega\cap\bar{\pi}(A)'\cap\{1_{\M^\omega}-\bar{\pi}(1_{A})\}^\perp$ by Lemma \ref{lem:CentralSurjectivity} (with $S=\emptyset$).  As $\bar{\pi}$ is a $^*$-homomorphism, Proposition  \ref{thm:W*CommTraces} shows that the collection of traces of the form $\tau(\bar{\pi}(a)\cdot)$ for $\tau\in T(\M^\omega)$ and $a\in A_{+}$ with $\tau(\bar{\pi}(a))=1$ has closed convex hull $T(\bar{C})$.  By Theorem \ref{thm:StrictCompTraceSurj}(i), traces on $C$ come from those on $\bar{C}$, and so the collection of traces $\tau(\pi(a)\cdot)$ for $\tau\in T(B_\omega)$ and $a\in A_{+}$ with $\tau(\pi(a))=1$ have closed convex hull $T(C)$.
\end{proof}

\begin{remark}
\label{rmk:QTT}
Let $B_n$, $A$, and $\pi$ be as in Theorem \ref{thm:StrictCompTraceSurj}, and define $C$ as in \eqref{eq:StrictCompCdef}.
Suppose that $\partial_e T(B_n)$ is compact, so that (ii) of Theorem \ref{thm:StrictCompTraceSurj} applies.
Then a result of Ng and Robert, \cite[Theorem 3.6(ii)]{NR:arXiv}, shows that $QT(C)=T(C)$.

Let us explain how to verify the hypotheses of \cite[Theorem 3.6(ii)]{NR:arXiv}, namely that the primitive ideal space of $C$ is compact and that $\Cu(C)$ has strict comparison of full positive elements by traces, in the sense of \cite[Definition 3.2]{NR:arXiv} (whereas Theorem \ref{thm:StrictCompTraceSurj}(ii) only shows the analogous property for $W(C)$).
The quotient $\bar C$ is unital, and if $e \in C$ is a positive contraction lifting the unit of $\bar{C}$, then it is clear from Theorem \ref{thm:StrictCompTraceSurj}(i) and (ii) that $(e-1/2)_+$ is full.
Hence by \cite[Proposition 3.1]{TT:CMB}, the primitive ideal space of $C$ is compact.

Next, let $a,b \in (C \otimes \mathcal K)_+$ be full positive elements and let $\gamma > 0$ be such that $d_\tau(a) \leq (1-\gamma)d_\tau(b)$ for all $\tau \in T(C)$; we must show that $[a] \leq [b]$ in $\Cu(C)$.
Since $[a]$ is the supremum of elements from $W(C)$, we may assume that $[a] \in W(C)$.
Let $(b_n)_{n=1}^\infty$ be a sequence of elements from matrix algebras over $C$, such that $([b_n])_{n=1}^\infty$ is an increasing sequence in $\Cu(B)$ whose supremum is $[b]$ (for example, use $b_n:=(1_C \otimes 1_{M_n})b(1_C \otimes 1_{M_n})$).
Let $\eps > 0$; then for each $\tau\in T(C)$, $\tau(g_\eps(a)) \leq d_\tau(a) < d_\tau(b) = \sup_n d_\tau(b_n)$ (because $d_\tau(a) < \infty$ and $d_\tau(b) > 0$).
Since the map $\tau \to \tau(g_\eps(a))$ is continuous and $T(C)$ is compact (by Theorem \ref{thm:StrictCompTraceSurj}(i)), there exists $n\in\N$ such that
\begin{equation} \tau(g_\eps(a)) < d_\tau(b_n), \quad \tau \in T(C). \end{equation}
By strict comparison (Theorem \ref{thm:StrictCompTraceSurj}(ii)), $[(a-\eps)_+] \leq [b_n] \leq [b]$.
Since $\eps$ is arbitrary, $[a] \leq [b]$, as required.
\end{remark}

\clearpage\section{Unitary equivalence of totally full positive elements}
\label{sec:TotallyFullClass}

\noindent
In this section, we classify totally full positive elements in relative commutant sequence algebras with respect to c.p.c.\ order zero maps, up to unitary equivalence.
We state the result in the following theorem, and prove it in Section \ref{sec6.1}.
In Section \ref{sec6.2}, we use Theorem \ref{cor:TotallyFullClassFinite} to access the $2\times 2$ matrix trick, to establish a strengthening of Theorem \ref{onecolourclass} (as Theorem \ref{KeyLemmaP}), which is the key classification result used to establish two coloured classification and our covering dimension estimates. Note that we do not need simplicity of the $B_n$ in Theorem \ref{cor:TotallyFullClassFinite} or in the lemmas of Section \ref{sec6.1}, just that $C$ has strict comparison.\footnote{However, when we come to use Theorem \ref{cor:TotallyFullClassFinite}, we will use simplicity of $B_n$ to obtain strict comparison of $C$ via Section \ref{sec:StrictComp}.}

Recall from Definition \ref{def:TotallyFull} that a nonzero $h\in C_+$ is totally full if $f(h)$ is full in $C$ for every non-zero $f \in C_0((0,\|h\|])_+$.

\begin{thm}
\label{cor:TotallyFullClassFinite}
Let $(B_n)_{n=1}^\infty$ be a sequence of separable, unital, stably finite $\mathcal Z$-stable $\mathrm C^*$-algebras with $QT(B_n)=T(B_n)$ for all $n$, and set $B_\omega := \prod_\omega B_n$.
Let $A$ be a separable, unital $\mathrm C^*$-algebra and let $\pi:A\rightarrow B_\omega$ be a c.p.c.\ order zero map such that
\begin{equation}\label{e6.1}
C:=B_\omega \cap \pi(A)' \cap \{1_{B_\omega}-\pi(1_A)\}^\perp
\end{equation}
is full in $B_\omega$ and has strict comparison of positive elements with respect to bounded traces (as in Definition \ref{defn:StrictComp}).

Let $a,b \in C_+$ be totally full positive elements.
Then $a$ and $b$ are unitarily equivalent (by unitaries in the unitization of $C$) if and only if $\tau(a^k) = \tau(b^k)$ for every $\tau \in T(C)$ and $k\in \N$.
\end{thm}

\bigskip

\subsection{\sc Proof of Theorem \ref{cor:TotallyFullClassFinite}}\hfill \\
\label{sec6.1}

\noindent
Our proof of Theorem \ref{cor:TotallyFullClassFinite} is motivated by stable rank one considerations: if we knew that the stable rank of the relative commutant sequence algebra $C$ given in \eqref{e6.1} is one, then by \cite[Theorem 4]{CE:IMRN}, one gets a classification of \textit{all} positive elements by their behaviour on the Cuntz semigroup (as Kirchberg's $\eps$-test shows that unitary equivalence and approximate unitary equivalence are the same in the unitization of $C$ (Lemma \ref{NewEpsLemma}(i)).
However, we were unable to determine whether algebras of the form (\ref{e6.1}) have stable rank one. Indeed, even the following question is open.

\begin{question}\label{Q5.2}
Let $A$ be a simple, separable, unital, nuclear $\mathrm C^*$-algebra which is $\mathcal Z$-stable (or even UHF-stable) and has a unique trace. Does $A_\omega \cap A'$ have stable rank one?
\end{question}

To the authors' knowledge, the answer is only known to be positive for a small handful of $\mathrm C^*$-algebras, namely AF algebras and the Jiang-Su algebra, while there are no $\mathrm C^*$-algebras for which the answer is known to be negative.
For an AF algebra, if $B \subseteq A$ is finite dimensional then $A_\omega \cap B'$ is a direct sum of unital corners of $A_\omega$, so it has stable rank one. The $\eps$-test can then be used to prove that $A_\omega \cap A'$ has stable rank one if $A$ is AF (in fact, unique trace is not needed for this argument).
For $\mathcal Z$, \cite[Theorem 2(2)]{FHRT} shows that $\mathcal Z$ is elementarily equivalent to $\mathcal Z_\omega \cap \mathcal Z'$ (in the language of continuous model theory --- see \cite{F:ICM}); since stable rank one is preserved by elementary equivalence (a consequence of Lemma \ref{lem:PolarDecomp}), it follows that $\mathcal Z_\omega \cap \mathcal Z'$ has stable rank one.

If $A$ is as in Question \ref{Q5.2}, then our results imply that every element of $A_\omega \cap A' \cap J_{A_\omega}$ can be approximated by invertibles.
Certainly, for such an element $x$, an application of Kirchberg's $\eps$-test similar to Lemma \ref{lem:ActAsUnit} provides another element $e \in A_\omega \cap A' \cap J_{A_\omega}$ such that $x \vartriangleleft e$.
Thus $1_{A_\omega}-e$ is a two-sided zero divisor for $x$, and it is full in $A_\omega \cap A'$ by Theorem \ref{thm:StrictCompTraceSurj}(i) and (ii).
Hence by Lemma \ref{lem:Robertsr1} (with $B_n:=A$, $S:=A$, and $d:=1_A$), $x$ is approximated by invertibles in $A_\omega \cap A'$.  However, this does not show that $A_\omega \cap A' \cap J_{A_\omega}$ has stable rank one, because it says nothing about elements of the unitization.

Further, for $A$ as in Question \ref{Q5.2}, $A_\omega \cap A'$ has stable rank one if and only if every element of the form $1_{A_\omega}+x$, where $x \in A_\omega \cap A' \cap J_{A_\omega}$, can be approximated by invertibles (in the unitization of $A_\omega \cap A' \cap J_{A_\omega}$, or equivalently by \cite[Theorem 4.4]{R:PLMS}, in $A_\omega \cap A'$).
Certainly, if every element of the form $1_{A_\omega}+x$, where $x \in I:=A_\omega \cap A' \cap J_{A_\omega}$, can be approximated by invertibles, then this shows that the ideal $I$ has stable rank one.
The quotient of $A_\omega \cap A'$ by $I$ is the II$_1$ factor $\mathcal R^\omega \cap \mathcal R'$, and hence has stable rank one.\footnote{Finite von Neumann algebras have stable rank one, since operators in a finite von Neumann algebra have unitary polar decompositions.} Moreover, the unitary group of $(A_\omega \cap A')/I$ is connected (\cite[Example V.1.2.3 (i)]{B:Encyc}).
By \cite[Theorem 4.11]{R:PLMS}, stable rank one for $I$ then implies stable rank one for $A_\omega \cap A'$.

In the absence of an answer to Question \ref{Q5.2} (let alone whether the algebra $C$ in Theorem \ref{cor:TotallyFullClassFinite} has stable rank one), we instead adapt an argument of Robert and Santiago (\cite{RS:JFA}), allowing us to prove Theorem \ref{cor:TotallyFullClassFinite} using only that \textit{certain} elements in \textit{certain} hereditary subalgebras of $C$ in \eqref{eq:TotallyFullClassCdef} are approximated by invertibles. We establish Theorem \ref{cor:TotallyFullClassFinite} using Lemma \ref{thm:TotallyFullClass} below, which is designed for re-use later: our proof that Kirchberg algebras have nuclear dimension one involves an analogue of Theorem \ref{cor:TotallyFullClassFinite} for purely infinite $B_n$, which will also make use of this lemma.
\begin{lemma}
\label{thm:TotallyFullClass}
Let $(B_n)_{n=1}^\infty$ be a sequence of separable, unital, $\mathcal Z$-stable $\mathrm C^*$-algebras and set $B_\omega := \prod_\omega B_n$.
Let $A$ be a separable, unital $\mathrm C^*$-algebra and let $\pi:A\rightarrow B_\omega$ be a c.p.c.\ order zero map such that
\begin{equation}
\label{eq:TotallyFullClassCdef}
C:=B_\omega \cap \pi(A)' \cap \{1_{B_\omega}-\pi(1_A)\}^\perp
\end{equation}
is full in $B_\omega$.

Assume that every full hereditary subalgebra $D$ of $C$ satisfies the following: if $x \in D$ is such that there exist totally full elements $e_l,e_r \in D_+$ such that $e_l x = xe_r = 0$, then there exists a full element $s \in D$ such that $sx = xs = 0$.

Let $a,b \in C_+$ be totally full positive contractions.
Then $a$ and $b$ are unitarily equivalent (by unitaries in the unitization of $C$) if and only if for every $f \in C_0((0,1])_+$, $f(a)$ is Cuntz equivalent to $f(b)$ in $C$.
\end{lemma}

Before proceeding, we show how to recapture Theorem \ref{cor:TotallyFullClassFinite} from Lemma \ref{thm:TotallyFullClass}.

\begin{proof}[Proof that Theorem \ref{cor:TotallyFullClassFinite} follows from Lemma \ref{thm:TotallyFullClass}]
Let $a,b \in C_+$.
If $a$ and $b$ are unitarily equivalent then of course $\tau(a^k)=\tau(b^k)$ for every $k\in\N$ and $\tau \in T(C)$; we must prove the converse.
For this, we may assume that $a,b$ are positive contractions.
We check (i) that when $C$ is full and has strict comparison, then it satisfies the technical condition of Lemma \ref{thm:TotallyFullClass};\footnote{In this stably finite case the technical condition holds when $e_l$ and $e_r$ are full; it is convenient in the purely infinite setting of Section \ref{sec:KirAlgs} to ask for these elements to be totally full in this technical condition.} 
and (ii) that strict comparison of $C$ and the hypothesis $\tau(a^k)=\tau(b^k)$ for all $k\in\N$ and $\tau\in T(C)$ implies that $a$ and $b$ satisfy the hypothesis concerning Cuntz equivalence. Then the conclusion follows from Lemma \ref{thm:TotallyFullClass}.

(i): Let $D$ be a full hereditary subalgebra of $C$, which inherits strict comparison by traces from $C$ (see Remark \ref{rmk:StrictCompFullHered}). Let $x\in D$ and let $e_l, e_r \in D_+$ be full and satisfy $e_l x=xe_r=0$.
We claim that there exists $h\in D_+$ such that $(h-\dl)_+$ is full, for some $\dl>0$, and 
\begin{equation}\label{e6.3}
d_\tau(h)<\min\{d_\tau(e_l),d_\tau(e_r)\},\quad \tau \in T(C).
\end{equation}

To see this, first, using Lemma \ref{lem:ActAsUnit} with $S_1:=\pi(A)$, $S_2:=\{1_{B_\omega}-\pi(1_A)\}$, and $T:=\{e_r,e_l\}$, we obtain an element $f' \in C_+$ such that $e_r,e_l \vartriangleleft f'$; consequently, $e_r,e_l \vartriangleleft 2(f'-1/2)_+$ as well.
Let $h_0 \in D_+$ be any full element (eg.\ take $h_0=e_r$).
There exists $\eps>0$ such that $(f'-1/2)_+$ is in the ideal generated by $(h_0-\eps)_+$, so that $(h_0-\eps)_+$ is full.
Let $m \in \mathbb N$ be such that $[(h_0-\eps/2)_+] \leq m[e_r]$ and $[(h_0-\eps/2)_+] \leq m[e_l]$ in the Cuntz semigroup of $D$. Then by Lemma \ref{lem:Zfacts}(\ref{lem:Zfacts6}), find $0\leq h\leq (h_0-\eps/2)_{+}$ in $B_\omega\cap \pi(A)'$ (so that $h\in D_+$), with $(m+1)[h]\leq [(h_0-\eps/2)_{+}] \leq (m+2)[h]$. 
In particular, since $d_\tau(e_l), d_\tau(e_r)>0$ (as these elements are full), \eqref{e6.3} holds.
Moreover, since $(h_0-\eps)_+$ is full, there exists $\dl>0$ such that $(h-\dl)_+$ is full in $D$.

By strict comparison, $h$ is Cuntz below $e_l, e_r$, so it follows by \cite[Proposition 2.4]{R:JFA2} that there exists $y,z \in D$ such that
\begin{equation} (h-\dl)_+ = y^*y=z^*z \end{equation}
and
\begin{equation} yy^* \in \mathrm{her}(e_l), \quad zz^* \in \mathrm{her}(e_r). \end{equation}
Then set $s:= zy^*$ so that $sx=xs=0$.
Also, since $y^*s^*sy = (h-\dl)_+^3$, it follows that $s$ is full in $D$.

(ii): Let $f \in C_0((0,1])_+$ be nonzero, and let us show that $f(a)$ and $f(b)$ are Cuntz equivalent.
The assumption that $\tau(a^k)=\tau(b^k)$ for all $k\in\N$ and $\tau\in T(C)$ shows that $\tau(g(a))=\tau(g(b))$ for every continuous function $g\in C_0((0,1])_+$ and every $\tau \in T(C)$. This implies that $\|a\|=\|b\|$, as if $\|a\|<\|b\|$ say, then there would be some positive $g\in C_0((0,1])_+$ with $g(a)=0$ and $g(b)\neq 0$, whence $g(b)$ is full in $C$ and so has $\tau(g(b))>0$ for all $\tau\in T(C)$, while $\tau(g(a))=0$. Assume without loss of generality that $\|a\|=\|b\|=1$.  Note too that $d_\tau(g(a))=d_\tau(g(b))$ for all $\tau\in T(C)$ and $g\in C_0((0,1])_+$.

For $\eps > 0$, set $U := f^{-1}((0,\eps))$ and let $g \in C_0(U)_{+}$ be a nonzero function.
Since $a$ is totally full in $C$, $g(a)$ is full in $C$ and so $d_\tau(g(a))>0$ for every $\tau \in T(C)$.
However, $g(a)$ is orthogonal to $(f(a)-\eps)_+$, so that for every $\tau \in T(C)$,
\begin{align}
 d_\tau((f(a)-\eps)_+) &< d_\tau((f(a)-\eps)_+) + d_\tau(g(a))\nonumber \\&= d_\tau((f(a)-\eps)_+ + g(a)) \leq d_\tau(f(a)) = d_\tau(f(b)).
\end{align}
Now by applying strict comparison, it follows that $(f(a)-\eps)_+ \preceq f(b)$ in $D$, and since $\eps$ is arbitrary, $f(a) \preceq f(b)$.
By symmetry, $f(a)$ and $f(b)$ are Cuntz equivalent.
\end{proof}

We now turn to the proof of Lemma \ref{thm:TotallyFullClass} following the Robert-Santiago strategy. A key tool is provided by Lemma \ref{lem:GapInvertibles2}, which shows that appropriate elements are approximated by invertibles: namely those that have full zero divisors.  We start with our replacement for \cite[Lemma 2]{RS:JFA}; the notation in the following lemma has been selected to enable comparisons with \cite[Lemma 2]{RS:JFA}.
Note that, since we are working in a relative commutant sequence algebra to which Kirchberg's $\eps$-test applies (see Section \ref{sec:Reindexing}), our statement avoids an approximation as in the conclusion of \cite[Lemma 2]{RS:JFA}.

Recall that $a \vartriangleleft b$ means that $b$ acts as a unit on $a$.  In the rest of this section we use $\sim$ for Murray-von Neumann equivalence of positive elements, i.e., for $a,b\in A_+$, write $a\sim b$ to mean that there exists $x\in A$ with $x^*x=a$ and $xx^*=b$.
It is well-known that this is an equivalence relation; the proof is not difficult using polar decomposition.

\begin{lemma}[{cf.\ \cite[Lemma 2]{RS:JFA}}]
\label{lem:RSLem2}
Let $A,B_n,\pi$ be as in Lemma \ref{thm:TotallyFullClass}, and define $C$ as in \eqref{eq:TotallyFullClassCdef}. 
Let $e,f,f',\alpha, \beta \in C_+$ be such that
\begin{equation}
\label{eq:RSLem2Hyp1}
\alpha \vartriangleleft e,\quad \alpha \sim \beta \vartriangleleft f, \quad\text{and} \quad f \sim f' \vartriangleleft e.
\end{equation}
Suppose also that there exist $d_e,d_f \in C_+$ that are totally full, such that
\begin{align}
\notag
d_e &\vartriangleleft e, \quad d_e\alpha = 0, \quad \text{and} \\
d_f &\vartriangleleft f, \quad d_f\beta = 0.\label{eq:RSLem2Hyp2}
\end{align}
Then there exists $e' \in C_+$ such that
\begin{align}
\alpha \vartriangleleft e' \vartriangleleft e, \quad \text{and} \quad \alpha+e' \sim \beta+f.
\end{align}
\end{lemma}

\begin{proof}
Let $y \in C$ be such that
\begin{equation}
\label{eq:RSLem2yDef}
f = yy^*, \quad y^*y = f'.
\end{equation}
Set
\begin{equation}
\label{eq:RSLem2wDef}
\alpha_1 := y^*\beta y \in C \cap \{1_{B_\omega}-e\}^\perp.
\end{equation}

Note that $\alpha_1 \sim \beta^{1/2}yy^*\beta^{1/2} = \beta \sim \alpha$, so there exists $x \in C$ such that
\begin{equation}\label{eq:RSLem2New1}
\alpha = xx^*, \quad x^*x = \alpha_1. 
\end{equation}
Clearly, $x \vartriangleleft e$. Then $\alpha_1y^*d_fy=y^*\beta yy^*d_fy=y^*\beta d_fy=0$, so that 
\begin{equation}
xy^*d_fy=0.
\end{equation}
Also, 
\begin{equation} d_ex=0 \end{equation}
and $y^*d_fy, d_e \vartriangleleft e$ and these operators are totally full in $C$ (in the case of $y^*d_fy$, note that it is Murray-von Neumann equivalent to $d_f$).
The subalgebra
\begin{equation} \check{C} := C \cap \{1_{B_\omega}-e\}^\perp \end{equation}
is a full hereditary subalgebra of $C$, so by the technical hypothesis of Lemma \ref{thm:TotallyFullClass}, there exists $s \in \check{C}$ that is full, such that $sx=xs=0$.
Hence by Lemma \ref{lem:GapInvertibles2}, $x$ is approximated by invertibles in the unitization of $\check{C}$.

By Lemma \ref{lem:PolarDecomp} and Lemma \ref{NewEpsLemma}(ii), $x$ has a unitary polar decomposition, $x = u|x|$, where $u$ is a unitary in the unitization of $\check{C}$. Using (\ref{eq:RSLem2New1}), we have
\begin{equation}
\label{eq:RSLem2uDef}
 u\alpha_1 = \alpha u.
\end{equation}

Set
\begin{equation}
\label{eq:RSLem2e'Def}
 e' := uf'u^*. \end{equation}
Since $u$ is in the unitization of $\check{C}$, $u$ commutes with $e$, and therefore, $e' \vartriangleleft e$.
We can see that
\begin{eqnarray}
\notag
e'\alpha &\stackrel{\eqref{eq:RSLem2uDef},\eqref{eq:RSLem2e'Def}}=& (uf'u^*)(u\alpha_1 u^*) \\
\notag
&\stackrel{\eqref{eq:RSLem2yDef},\eqref{eq:RSLem2wDef}}=& uy^*y y^*\beta yu^* \\
\notag
&\stackrel{\eqref{eq:RSLem2Hyp1},\eqref{eq:RSLem2yDef}}=& uy^* \beta yu^* \\
&\stackrel{\eqref{eq:RSLem2wDef},\eqref{eq:RSLem2uDef}}=& \alpha.
\end{eqnarray}
Finally, we likewise have
\begin{align}
\notag
\alpha + e' &= u(\alpha_1+f')u^* \\
\notag
&= uy^*(\beta + 1_{\check{C}^\sim})yu^* \\
\notag
&\sim (\beta+1_{\check{C}^\sim})^{1/2}yy^*(\beta+1_{\check{C}^\sim})^{1/2} \\
\notag &= (\beta+1_{\check{C}^\sim})^{1/2}f(\beta+1_{\check{C}^\sim})^{1/2} \\
& =  \beta+f,
\end{align}
since $\beta$ and $f$ commute.
\end{proof}

\begin{proof}[Proof of Lemma \ref{thm:TotallyFullClass}]
This is essentially the proof of \cite[Theorem 1, (I) $\Rightarrow$ (II)]{RS:JFA}, although we need to be careful to verify the hypotheses of Lemma \ref{lem:RSLem2} that aren't present in \cite[Lemma 2]{RS:JFA}.
Also, since we aim for exact unitary equivalence, we avoid the hard analysis component of the proof of \cite[Theorem 1]{RS:JFA}.
For completeness, we give a full proof here. 
Note that one direction is trivial: if $a$ and $b$ are unitarily equivalent then for any $f \in C_0((0,1])_+$, $f(a)$ and $f(b)$ are unitarily equivalent, whence Cuntz equivalent.

Assume now that $f(a)$ and $f(b)$ are Cuntz equivalent (in $C$) for every $f \in C_0((0,1])_+$.  Note that $\|a\|=\|b\|$, as if $\|a\|<\|b\|$ say, then there would exist some $f\in C_0((0,1])_+$ with $f(a)=0$ and $f(b)\neq 0$, giving a contradiction.\footnote{No nonzero positive contraction is Cuntz equivalent to the zero element.} We will first show that $a$ and $b$ are Murray-von Neumann equivalent; using Lemma \ref{NewEpsLemma}(iii) it suffices to do this approximately. Therefore, fix a tolerance $\eps > 0$.

Let $m \in \mathbb N$ be even such that $\frac3m < \eps$.
For $i=0,1/2,1,\dots,m-1/2$, let $\xi_i \in C_0((0,1])_+$ be such that $\xi_i$ is identically $0$ on $[0,\frac im]$ and identically $1$ on $[\frac im+ \frac 1{2m},1]$. Note: the choice of indices here is selected to make comparison with the proof of \cite[Theorem 1]{RS:JFA} easier.

Define $a_i := \xi_{m-i+1/2}(a)$ and $b_i := \xi_{m-i+1/2}(b)$ for $i=1,\frac32,\dots,m+\frac12$.
As in \cite[proof of Theorem 1]{RS:JFA}, these satisfy
\begin{align}
a_i \vartriangleleft a_{i+1/2},\ b_i &\vartriangleleft b_{i+1/2}, \quad \text{for }i=1,\frac32,\dots,m.
\end{align}

By assumption, $a_{i+1/2}$ is Cuntz equivalent to $b_{i+1/2}$, and therefore by \cite[Proposition 2.4]{R:JFA2}, there exist $c_i,d_i \in C_+$ such that
\begin{align}
a_i \sim d_i \vartriangleleft b_{i+1/2},\ b_i \sim c_i &\vartriangleleft a_{i+1/2}, \quad \text{for }i=1,\dots,m.
\end{align}

Inductively, we will produce positive contractions $a_1',a_3',\dots,a_{m-1}'$ and $b_2',b_4',\dots,b_m'$, such that
\begin{gather*}
a_1' \vartriangleleft a_{3/2} \vartriangleleft a_2 \vartriangleleft a_3' \vartriangleleft \cdots \vartriangleleft a_{m-1}' \vartriangleleft a_{m-1/2}, \\
b_1 \vartriangleleft b_2' \vartriangleleft b_{5/2}\vartriangleleft b_3\vartriangleleft b_4' \vartriangleleft \cdots \vartriangleleft b_m' \vartriangleleft b_{m+1/2},
\end{gather*}
and, for each $k=1,\dots,m$,
\begin{equation}
\sum_{i=1}^k a_i'' \sim \sum_{i=1}^k b_i'', \end{equation}
where
\begin{equation}
a_i'' = \begin{cases} a_i',\quad &i\text{ odd}; \\ a_i,\quad &i\text{ even,} \end{cases}\text{ and }
b_i'' = \begin{cases} b_i,\quad &i\text{ odd}; \\ b_i',\quad &i\text{ even.} \end{cases}
\end{equation}

For the base case, $k=1$, we define $a_1'=c_1$.
For the inductive step, having defined $a_i',b_i'$ for odd, respectively even, $i \leq k$, let us describe how to construct $a_{k+1}'$ or $b_{k+1}'$, (depending on whether $k+1$ is odd or even).
The even and odd cases are essentially the same, so assume that $k+1$ is odd.

Set $\alpha := \sum_{i=1}^k a_i'', \beta := \sum_{i=1}^k b_i'', e:=a_{k+3/2}$ and $f:=b_{k+1}$.
These satisfy the hypotheses (\ref{eq:RSLem2Hyp1}) of Lemma \ref{lem:RSLem2}, with $f' =c_{k+1}$.
For the hypotheses (\ref{eq:RSLem2Hyp2}), note that $\alpha \vartriangleleft a_{k+1}$, and so there exists a nonzero $g \in C_0((0,1])_+$ such that
\begin{equation} g(a)\alpha = 0, \quad g(a) \vartriangleleft a_{k+3/2}=e, \end{equation}
and we therefore may set $d_e:=g(a)$, which is totally full since $a$ is totally full.
Likewise, since $\beta \vartriangleleft b_{k+1/2}$, there exists $d_f$ totally full such that $d_f\beta = 0$ and $d_f \vartriangleleft b_{k+1}=f$.
Having satisfied all the hypotheses of Lemma \ref{lem:RSLem2}, this lemma provides $e'\in C_+$ such that
\begin{equation} \alpha \vartriangleleft e' \vartriangleleft a_{k+3/2} \quad \text{and} \quad \alpha+e' \sim \beta+f. \end{equation}
Therefore, setting $a_{k+1}':=e'$, we get
\begin{equation} \sum_{i=1}^{k+1} a_i'' =\alpha+e' \sim \beta+f = \sum_{i=1}^{k+1} b_i'', \end{equation}
as required.

This completes the induction, and therefore shows that the $a_i'$ and the $b_i'$ can be chosen for $i$ up to $m$.
Now, set
\begin{equation} \alpha := \sum_{i=1}^m a_i'', \quad \beta := \sum_{i=1}^m b_i''. \end{equation}
We note that, by the choice of $\xi_0,\dots,\xi_{m-1/2}$, 
\begin{align}
\notag
a-\frac3m &\leq \frac1m (0 + a_1 + \cdots + a_{m-1}) \\
&\leq \frac1m (a_1' + a_2 + \cdots + a_{m-1}' + a_m) = \frac1m \alpha,
\end{align}
and likewise,
\begin{align}
\frac1m \alpha \leq a, \quad \text{and} \quad
b-\frac3m \leq \frac1m \beta \leq b.
\end{align}
Thus, 
\begin{equation} a \approx_\eps \frac\alpha m \sim \frac \beta m \approx_\eps b. \end{equation}
By Lemma \ref{NewEpsLemma}(iii) $a$ and $b$ are Murray-von Neumann equivalent in $C$.

We end by showing that $a$ and $b$ are unitarily equivalent.  Let $\eps > 0$.
We may apply the above argument to show that $(a-\eps)_+$ and $(b-\eps)_+$ are approximately, and therefore exactly, Murray-von Neumann equivalent, so let $x \in C_+$ be such that
\begin{equation} (a-\eps)_+ = xx^* \quad \text{and} \quad x^*x = (b-\eps)_+. \end{equation}
Let $g \in C_0((0,\eps))_+$ be nonzero.
Then $g(a)x=xg(b)=0$, and $g(a),g(b)$ are totally full.
Therefore by the technical hypothesis (applied to $D:=C$) and Lemma \ref{lem:GapInvertibles2}, $x$ is approximated by invertibles in the unitization of $C$.
Hence, by Lemma \ref{lem:PolarDecomp} and Lemma \ref{NewEpsLemma}(ii), $x$ has a unitary polar decomposition $x=u|x|$, where $u$ is in the unitization of $C$.
Then $u(b-\eps)_+u^* = (a-\eps)_+$.  In particular $ubu^*\approx_{2\eps}a$. Since $\eps>0$ was arbitrary, Lemma \ref{NewEpsLemma}(i) shows that $b$ and $a$ are unitarily equivalent, by a unitary in the unitization of $C$.
\end{proof}

\bigskip

\subsection{\sc Theorem \ref{onecolourclass}}\hfill  \\

\label{sec6.2}

\noindent
This section contains the key classification theorem, of which Theorem \ref{onecolourclass} is a special case.  This result is the main ingredient in both the classification up to two-coloured equivalence (Theorem \ref{thm:2ColourUniqueness}) and the nuclear dimension (and decomposition rank) computation (Theorem \ref{thm:FiniteDim}). We would like to thank Thierry Giordano for asking whether nuclearity of the $B_n$ was necessary in a presentation of an earlier version of this lemma, leading us to the statement presented here.

\begin{thm}
\label{KeyLemmaP}
Let $(B_n)_{n=1}^\infty$ be a sequence of simple, separable, unital, finite, $\mathcal Z$-stable $\mathrm C^*$-algebras with $QT(B_n)=T(B_n)$ for all $n$ and such that $\partial_eT(B_n)$ is compact for each $n\in\N$. Set $B_\omega := \prod_\omega B_n$ and $(B\otimes \mathcal Z)_\omega := \prod_\omega (B_n \otimes \mathcal Z)$ so that we have a canonical embedding of $B_\omega\otimes\Z$ into $(B\otimes\Z)_\omega$.\footnote{This is an abuse of notation (based on the case when all the $B_n$ are constant and equal to $B$); there is no  C$^*$-algebra $B$ in the lemma.} 
Let $A$ be a separable, unital, nuclear $\mathrm C^*$-algebra. Let $\phi_1:A\rightarrow B_\omega$ be a totally full $^*$-homomorphism and let $\phi_2:A\rightarrow B_\omega$ be a c.p.c.\ order zero map such that 
\begin{equation}\label{KeyLemma.AgreeTraces}
\tau\circ\phi_1=\tau\circ\phi_2^m,\quad \tau\in T(B_\omega),\ m\in\N,
\end{equation}
where order zero functional calculus is used to interpret $\phi_2^m$.
Let $k\in\mathcal Z_+$ be a positive contraction with spectrum $[0,1]$ and set $\psi_i:=\phi_i(\cdot) \otimes k:A \to (B\otimes \mathcal Z)_\omega$ for $i=1,2$.\footnote{$\phi_i(\cdot)\otimes k$ denotes the map $a \mapsto \phi_i(a) \otimes k$ from $A$ to $(\prod_\omega B_n) \otimes \mathcal Z \subset \prod_\omega (B_n\otimes \mathcal Z)$.}
Then $\psi_1$ is unitarily equivalent to $\psi_2$ in $(B\otimes \mathcal Z)_\omega$.
\end{thm}

\begin{proof}
Define a $^*$-homomorphism $\hat{\psi}_1:=\phi_1(\cdot)\otimes 1_{\mathcal Z}:A\rightarrow (B\otimes\Z)_\omega$ (which is a supporting c.p.c.\ order zero map for $\psi_1$). Equation (\ref{KeyLemma.AgreeTraces}) ensures that $d_\tau(\phi_2(1_{A}))=\tau(\phi_1(1_{A}))$ for all $\tau\in T(B_\omega)$. Accordingly $\tau\mapsto d_\tau(\phi_2(1_{A}))$ is weak$^*$-continuous, and so Lemma \ref{lem:SupportingMap} provides a supporting order zero map $\hat{\phi}_2:A\rightarrow B_\omega$ 
for which the induced map $\bar{\hat{\phi}}_2:A\rightarrow B_\omega/J_{B_\omega}$ is a $^*$-homomorphism.  Note that (\ref{KeyLemma.AgreeTraces}) shows that 
\begin{equation}
\tau(f(\phi_2)(a))=f(1)\tau(\phi_1(a)),\quad \tau\in T_\omega(B_\omega),\ a\in A_+,\ f\in C_0((0,1])_+,
\end{equation}
when
(\ref{eq:SupportingTrace}) gives
\begin{equation}\label{KeyLemma.NewEquationRevised}
\tau(\hat{\phi}_2(a))=\lim_{n\rightarrow\infty}\tau(\phi^{1/n}_2(a))=\tau(\phi_1(a)),\quad \tau\in T_\omega(B_\omega),\ a\in A_+.
\end{equation}

Write $\hat{\psi}_2=\hat{\phi}_2(\cdot)\otimes 1_\Z:A\rightarrow (B\otimes\Z)_\omega$.  This is a supporting order zero map for $\psi_2$ whose induced map $\bar{\hat{\psi}}_2:A\rightarrow (B\otimes \Z)_\omega/J_{(B\otimes\Z)_\omega}$ is a $^*$-homomorphism.

Define $\pi:A\rightarrow M_2(B_\omega) \subset M_{2}((B \otimes \mathcal{Z})_{\omega})$ by
\begin{equation}
\label{KeyLemma.2}
\pi(a):=\begin{pmatrix}\hat{\psi}_1(a)&0\\0&\hat{\psi}_2(a)\end{pmatrix},\quad a\in A,
\end{equation}
and set
\begin{equation}\label{KeyLemma.1}
C:=M_2((B\otimes \mathcal Z)_\omega)\cap \pi(A)'\cap \{1_{M_2(B\otimes\mathcal Z)_\omega}-\pi (1_A)\}^\perp.
\end{equation}
We wish to use the $2\times 2$ matrix trick (Lemma \ref{lem:CombinedUnitaryEquiv}), to obtain unitary equivalence of $\psi_1$ and $\psi_2$, and so we first use Theorem \ref{cor:TotallyFullClassFinite} to show that 
\begin{equation}
h_1:=\begin{pmatrix}\psi_1(1_A)&0\\0&0\end{pmatrix}\text{ and }h_2:=\begin{pmatrix}0&0\\0&\psi_2(1_A)\end{pmatrix}
\end{equation}
are unitarily equivalent in the unitization $C^\sim$ of $C$.

Since $QT(B_n)=T(B_n)$ and $B_n \cong B_n \otimes \mathcal Z$, it follows that $QT(M_2 \otimes B_n \otimes \mathcal Z)=T(M_2\otimes B_n \otimes \mathcal Z)$.

For $a\in A_+$ nonzero, $\psi_1(a)$ is full in $(B\otimes\Z)_\omega$ as $\phi_1:A \to B_{\omega}$ is totally full and $k \in \mathcal{Z}$ is full.  As
\begin{equation}
0\leq \begin{pmatrix}\psi_1(a)&0\\0&0\end{pmatrix}\leq\begin{pmatrix}\hat{\psi}_1(a)&0\\0&0\end{pmatrix}\leq \pi(a),
\end{equation}
$\pi(a)$ is full in $M_2((B\otimes\Z)_\omega)$.
Write $J$ for the trace kernel ideal in $M_2((B\otimes\Z)_\omega)\cong \prod_\omega(M_2\otimes B_n\otimes\Z)$.    As both $\hat{\psi}_1,\hat{\psi}_2:A\rightarrow (B\otimes\Z)_\omega$ induce $^*$-homomorphisms $\bar{\hat{\psi}}_i:A\rightarrow (B\otimes\Z)_\omega / J_{(B\otimes\Z)_\omega}$, the induced map $\bar{\pi}:A\rightarrow M_2((B\otimes\Z)_\omega)/J$ is a $^*$-homomorphism. Noting that each $M_2\otimes B_n\otimes\Z$ has compact extremal tracial boundary, we have verified the hypotheses of Theorem \ref{thm:StrictCompTraceSurj} parts (ii) and (iii). Therefore $C$ has strict comparison of positive elements with respect to bounded traces and $T(C)$ is the closed convex hull of the set $T_0$ of all traces on $C$ of the form $\tau(\pi(a)\cdot)$ where $\tau \in T(M_2((B\otimes\Z)_\omega))$ and $a \in A_+$ satisfies
\begin{equation}
\label{eq:KeyLemmaTrNorm}
\tau(\pi(a))=1.
\end{equation}

We now show that Theorem \ref{cor:TotallyFullClassFinite} can be applied.  It is immediate that $C$ is full in $M_2((B\otimes\Z)_\omega)$ as $h_1\in C$ is full in $M_2((B\otimes\Z)_\omega)$. We must verify that $\rho(h_1^m)=\rho(h_2^m)$ for all $\rho\in T(C)$ and $m\in\N$ and that $h_1$ and $h_2$ are totally full in $C$ (and not just in $M_2((B\otimes\Z)_\omega)$).

Let $\rho =\tau(\pi(a)\cdot) \in T_0$ where $\tau \in T(M_2 ( (B \otimes \Z)_\omega))$ and $a \in A_+$ satisfies (\ref{eq:KeyLemmaTrNorm}). For $i=1,2$ we have
\begin{equation}
\label{eq:psiExpand}
\hat{\psi}_i(a)\psi_i(1_{A})^m=\phi_i^m(a)\otimes k^m,\quad m\in\N,
\end{equation}
where $\phi_i^m$ is given by order zero functional calculus.
Define $\tilde\tau \in T(B_\omega)$ by $\tilde\tau(b) = \tau(1_2 \otimes b \otimes 1_{\Z})$.
Thus, using the fact that $\Z$ and $M_2$ have unique trace,
\begin{eqnarray}
\rho(h_i^m) &=& \tau(\pi(a)h_i^m) \nonumber\\
&\stackrel{\eqref{eq:psiExpand}}=& \tilde\tau(\phi_i^m(a))\tau_{\mathcal Z}(k^m)/2\nonumber\\
&\stackrel{(\ref{KeyLemma.AgreeTraces})}{=}&\tilde{\tau}(\phi_1(a))\tau_\Z(k^m)/2\nonumber\\
&=&\tau(\pi(a))\tau_{\Z}(k^m)/2 \nonumber \\
&\stackrel{\eqref{eq:KeyLemmaTrNorm}}=& \tau_{\Z}(k^m)/2,\quad m\in \N, \label{eq:KeyLemmaTraceCalc}
\end{eqnarray}
using \eqref{KeyLemma.AgreeTraces} and the fact that 
\begin{equation}
\tau(\pi(a))=\frac12\tilde{\tau}(\phi_1(a)+\hat{\phi}_2(a))\stackrel{(\ref{KeyLemma.NewEquationRevised})}=
\frac12\tilde{\tau}(\phi_1(a)+\phi_2(a))
\end{equation} for the fourth equality. Therefore, for all $f\in C_0((0,1])_+$ we have
\begin{equation}\label{e6.38}
\rho(f(h_1))=\rho(f(h_2)))=\tau_\Z(f(k))/2;
\end{equation}
since the closed convex hull of $T_0$ is dense in $T(C)$, this holds for all $\rho \in T(C)$.

For $f\in C_0((0,1])_+$ nonzero, $\tau_\Z(f(k))\neq 0$, so that by strict comparison of $C$, $f(h_1)$ and $f(h_2)$ are full in $C$, i.e., $h_1$ and $h_2$ are totally full in $C$.
Thus Theorem \ref{cor:TotallyFullClassFinite} applies and converts (\ref{e6.38}) to the unitary equivalence of $h_1$ and $h_2$ in the unitization of $C$. 

Since each $B_n\otimes \mathcal Z$ is unital, simple, finite and $\Z$-stable, $(B\otimes\mathcal Z)_\omega$ has stable rank one by Lemma \ref{lem:Zfacts}(\ref{lem:Zfacts4}) and so $\psi_1$ and $\psi_2$ are unitarily equivalent by the $2\times 2$ matrix trick of Lemma \ref{lem:CombinedUnitaryEquiv}.
\end{proof}

\clearpage\section{$2$-coloured equivalence}\label{sec:2colour}

\noindent
We now have all the ingredients to establish the $2$-coloured uniqueness theorem (Theorem \ref{thm:2colouredIntroVersion}). We start by recalling the definition of coloured equivalence from (\ref{e1.1}).
\begin{defn}
\label{def:ColourEquiv}
Let $A,B$ be unital $\mathrm C^*$-algebras, and let $\phi_1,\phi_2:A \to B$ be unital ${}^*$-homomorphisms, and $n\in\N$. Say that $\phi_1$ and $\phi_2$ are \emph{$n$-coloured equivalent} if there exist $w^{(0)},\dots,w^{(n-1)}\in B$ such that
\begin{align}
\notag
\phi_1(a) &= \sum_{i=0}^{n-1} w^{(i)}\phi_2(a)w^{(i)}{}^*, \quad a\in A, \\
\phi_2(a) &= \sum_{i=0}^{n-1} w^{(i)}{}^*\phi_1(a)w^{(i)}, \quad a\in A,
\end{align}
$w^{(i)}{}^*w^{(i)}$ commutes with the image of $\phi_2$, and $w^{(i)}w^{(i)}{}^*$ commutes with the image of $\phi_1$, for $i=0,\dots,n-1$.
(Consequently, the summands $w^{(i)}\phi_2(\cdot)w^{(i)}{}^*$ and $w^{(i)}{}^*\phi_1(\cdot)w^{(i)}$ are order zero.)

We will say that $\phi_1$ and $\phi_2$ are \emph{approximately $n$-coloured equivalent} if the compositions $\iota\circ\phi_1,\iota\circ\phi_2:A\rightarrow B_\omega$ are $n$-coloured equivalent.
\end{defn}

There are other possible notions of coloured versions of equivalence; \cite{CTW:InPrep} will discuss the relationship between $n$-coloured equivalence and the stronger concept of $n$-intertwined $^*$-homomorphisms.\footnote{Unital $^*$-homomorphisms $\phi_1,\phi_2:A\rightarrow B$ are said to be \emph{$n$-intertwined} if there exist $v^{(0)},\dots,v^{(n-1)}\in B$ with $\phi_1(a)v^{(i)}=v^{(i)}\phi_2(a)$ for all $a\in A$ and $i=0,\dots,n-1$ satisfying $\sum_{i=0}^{n-1}v^{(i)}v^{(i)}{}^*=\sum_{i=0}^{n-1}v^{(i)}{}^*v^{(i)}=1_B$.} 

Note that, despite its name, there is no reason to expect that (approximate) $n$-coloured equivalence is an equivalence relation. However it is immediate that (approximately) $n$-coloured equivalent $^*$-homomorphisms agree on traces as necessarily the $w^{(i)}$ in the definition will satisfy $\sum_{i=0}^{n-1}w^{(i)}w^{(i)}{}^*=\sum_{i=0}^{n-1}w^{(i)}{}^*w^{(i)}=1_{B}$. The following theorem, of which Theorem \ref{thm:2colouredIntroVersion} is a special case, provides a converse; this is established by decomposing each $\phi_i$ as in Theorem \ref{KeyLemmaP}.   In particular this shows that, for the maps as in Theorem \ref{thm:2ColourUniqueness}, $2$-coloured equivalence is an equivalence relation.

\begin{thm}
\label{thm:2ColourUniqueness}
Let $(B_n)_{n=1}^\infty$ be a sequence of simple, separable, unital, finite, $\mathcal Z$-stable $\mathrm C^*$-algebras with $QT(B_n)=T(B_n)$ and such that $\partial_e T(B_n)$ is compact and nonempty for each $n\in\N$. Set $B_\omega := \prod_\omega B_n$ and define $J_{B_\omega}$ as in \eqref{eq:JBdef}.
Let $A$ be a separable, unital, nuclear $\mathrm C^*$-algebra, and let $\phi_1,\phi_2:A\rightarrow B_\omega$ be $^*$-homomorphisms with $\phi_1$ totally full.
The following are equivalent.
\begin{enumerate}[(i)]
\item\label{2ColU.1} $\tau\circ\phi_1=\tau\circ\phi_2$ for all $\tau\in T(B_\omega)$.
\item\label{2ColU.2} There exists $k,l \in \N$ and $v^{(0)},\dots,v^{(k)},w^{(0)},\dots,w^{(l)}\in B_\omega$ such that
\begin{gather}
\notag
\sum_{i=0}^k v^{(i)}\phi_1(a)v^{(i)}{}^* = \sum_{j=0}^l w^{(j)}\phi_2(a)w^{(j)}{}^*, \quad a\in A, \text{ and} \\
v^{(0)}{}^*v^{(0)}+\dots+v^{(k)}{}^*v^{(k)}=w^{(0)}{}^*w^{(0)}+\dots+w^{(l)}{}^*w^{(l)} = 1_{B_\omega}.
\end{gather}
\item\label{2ColU.3} There exist $w^{(0)},w^{(1)}\in B_\omega$ such that
\begin{gather}
\notag
\phi_1(a) = 
w^{(0)}\phi_2(a)w^{(0)}{}^*+w^{(1)}\phi_2(a)w^{(1)}{}^* \quad a\in A, \text{ and} \\
w^{(0)}{}^*w^{(0)}+w^{(1)}{}^*w^{(1)} = 1_{B_\omega},
\end{gather}
and in addition, $w^{(i)}{}^*w^{(i)}$ commutes with the image of $\phi_2$ for $i=0,1$ (consequently, $w^{(i)}\phi_2(\cdot)w^{(i)}{}^*$ is order zero).
\item\label{2ColU.4} There exist $\tilde w^{(0)},\tilde w^{(1)}\in B_\omega$ such that
\begin{gather}
\phi_1(a) = 
\tilde{w}^{(0)}\phi_2(a)\tilde{w}^{(0)}{}^*+\tilde{w}^{(1)}\phi_2(a)\tilde{w}^{(1)}{}^* \quad a\in A, \\
\phi_2(a) = 
\tilde{w}^{(0)}{}^*\phi_1(a)\tilde{w}^{(0)}+\tilde{w}^{(1)}{}^*\phi_1(a)\tilde{w}^{(1)} \quad a\in A, \\
\phi_1(1_A)=\tilde{w}^{(0)}\tilde{w}^{(0)}{}^*+\tilde{w}^{(1)}\tilde{w}^{(1)}{}^*, \quad \phi_2(1_A)=\tilde{w}^{(0)}{}^*\tilde{w}^{(0)}+\tilde{w}^{(1)}{}^*\tilde{w}^{(1)},
\end{gather}
and in addition $\tilde{w}^{(i)}{}^*\tilde{w}^{(i)}$ commutes with the image of $\phi_2$ and $\tilde{w}^{(i)}\tilde{w}^{(i)}{}^*$ commutes with the image of $\phi_1$, for $i=0,1$.
\end{enumerate}
In the case that $\phi_1$ and $\phi_2$ are unital ${}^*$-homomorphisms, condition (\ref{2ColU.4}) is just that $\phi_1$ and $\phi_2$ are $2$-coloured equivalent, and the above conditions are also equivalent to that there exists some $m\in\N$ such that $\phi_1$ and $\phi_2$ are $m$-coloured equivalent.
Also in this case, $\tilde w^{(i)}$ can be chosen to be normal in (\ref{2ColU.4}).
\end{thm}

\begin{remark}
Note that the condition in (\ref{2ColU.3}) that $w^{(i)}{}^*w^{(i)}$ commutes with $\phi_2(A)$ is strictly stronger than $w^{(i)}\phi_2(\cdot)w^{(i)}{}^*$ being order zero.  Indeed if $\phi:\mathbb C\rightarrow M_2$ is the embedding into the top left hand corner and $p$ a nondiagonal projection, then certainly $p\phi(\cdot)p$ is order zero, but $p$ does not commute with $\phi(\mathbb C)$.
In general if $\phi:A\rightarrow B$ is an order zero map with $A$ unital, then for $w\in B$, the map $w^*\phi(\cdot)w$ is order zero if and only if $w$ satisfies the cut-down commutation relation:
\begin{equation}\label{CutDownRemark.1}
\phi^{1/2}(1_A)ww^*\phi^{1/2}(a)=\phi^{1/2}(a)ww^*\phi^{1/2}(1_A),\quad a\in A.
\end{equation}
To see this, fix $a\in A_+$ and check that $d:=\phi^{1/2}(1_A)ww^*\phi^{1/2}(1_A)$ and $e:=\phi^{1/2}(1_A) ww^*\phi^{1/2}(a)$ commute (by expanding $(de-ed)^*(de-ed)$ and using the order zero identity \eqref{eq:Ord0Ident} for $w^*\phi(\cdot)w$). Thus $de=ed$, which is positive. As $ee^*$ is in the hereditary subalgebra $\overline{dBd}$, it follows that there are positive functions $f_n$ such that $f_n(d)de\rightarrow e$ as $n\rightarrow\infty$. As $f_n(d)$ commutes with $de$, each $f_n(d)de$ is positive, and hence $e=e^*$, giving (\ref{CutDownRemark.1}).  Conversely if the cut-down commutation condition (\ref{CutDownRemark.1}) holds, then for $ab=0$ in $A$, we have
\begin{align}
\notag
w^*\phi(a)ww^*\phi(b)w&=w^*\phi^{1/2}(a)\phi^{1/2}(1_A)ww^*\phi^{1/2}(b)\phi^{1/2}(1_A)w\\
&=w^*\phi(ab)ww^*\phi(1_A)w=0.
\end{align}
\end{remark}

\begin{proof}[Proof of Theorem \ref{thm:2ColourUniqueness}]
It is easy to see that (\ref{2ColU.4})  $\Longrightarrow$ (\ref{2ColU.1}) and that  (\ref{2ColU.3}) $\Longrightarrow$ (\ref{2ColU.2}) $\Longrightarrow$ (\ref{2ColU.1}).   When $\phi_1$ and $\phi_2$ are unital, it is immediate that (\ref{2ColU.4}) is just $2$-coloured equivalence, and further $m$-coloured equivalence for some $m$, implies (\ref{2ColU.1}). Thus it remains to prove that (\ref{2ColU.3}) and (\ref{2ColU.4}) follow from (\ref{2ColU.1}), and obtain the additional normality statement in the case of unital maps. So assume (i), that is, that $\tau\circ\phi_1=\tau\circ\phi_2$ for all $\tau\in T(B_\omega)$.
By Lemma \ref{lem:Zfacts}(\ref{lem:Zfacts2}), without loss of generality we may also assume that $B_n = C_n\otimes\Z$ (for a copy $C_n$ of $B_n$) such that,
for $i=1,2$, $\phi_i=\check{\phi}_i\otimes 1_\Z$, for some $^*$-homomorphisms $\check{\phi}_i:A \to \prod_\omega C_n$ with $\check{\phi}_1$ totally full.

Let $h \in \mathcal Z_+$ have spectrum $[0,1]$. The conditions on $B_n$ and $A$ ensure that Theorem \ref{KeyLemmaP} can be applied, and so using this theorem twice, first with $k:=h$ and then with $k:=1_\Z-h$, we obtain unitaries $u_h,u_{1-h} \in B_\omega = \prod_\omega (C_n \otimes \Z)$ with
\begin{align}
\notag
\check{\phi}_1(a)\otimes h &=u_h(\check{\phi}_2(a) \otimes h)u_h^*, \\
\check{\phi}_1(a)\otimes (1_\Z-h)&=u_{1-h}(\check{\phi}_2(a)\otimes (1_\Z-h))u_{1-h}^*,\quad a\in A.\label{eq:2colour.2}
\end{align}

To obtain (\ref{2ColU.3}), define 
\begin{equation} w^{(0)}:=u_h(1_{\prod_\omega C_n} \otimes h^{1/2}) \quad \text{and} \quad w^{(1)} := u_{1-h} (1_{\prod_\omega C_n}\otimes (1_\Z-h)^{1/2}) \end{equation}
so that
\begin{align}
\phi_1(a)&=\check{\phi}_1(a)\otimes h + \check{\phi}_1(a)\otimes (1_\Z-h) \nonumber\\
&= u_h(\check{\phi}_2(a)\otimes h)u_h^*+u_{1-h}(\check{\phi}_2(a)\otimes (1_\Z-h))u_{1-h}^*\nonumber\\
&=w^{(0)}\phi_2(a)w^{(0)}{}^*+w^{(1)}\phi_2(a)w^{(1)}{}^*,\quad a\in A,\label{e7.9}
\end{align}
and
\begin{align}
\notag
& \hspace*{-2em} w^{(0)}{}^*w^{(0)}+ w^{(1)}{}^*w^{(1)} \\
&=(1_{\prod_\omega C_n} \otimes h)^{1/2}u_h^*u_h(1_{\prod_\omega C_n}\otimes h)^{1/2}\notag  \\
& \qquad+
(1_{\prod_\omega C_n} \otimes (1_\Z-h))^{1/2}u_{1-h}^*u_{1-h}(1_{\prod_\omega C_n}\otimes (1_\Z-h))^{1/2} \nonumber\\
&=1_{\prod_\omega C_n} \otimes h + 1_{\prod_\omega C_n} \otimes (1_\Z-h) = 1_{B_\omega}.\label{e7.10}
\end{align}
Further $w^{(0)}{}^*w^{(0)}=1_{\prod_\omega C_n} \otimes h$ and $w^{(1)}{}^*w^{(1)}=1_{\prod_\omega C_n} \otimes (1_\Z-h)$, and these operators commute with $\phi_2(A)$. This establishes (\ref{2ColU.3}).

To obtain (\ref{2ColU.4}), define
\begin{equation}\label{e.6.13}
\tilde{w}^{(0)}:=u_h(\check{\phi}_2(1_A)\otimes h^{1/2})\text{ and }\tilde{w}^{(1)}:=u_{1-h}(\check{\phi}_2(1_A)\otimes (1_\Z-h)^{1/2})
\end{equation}
so that $\tilde{w}^{(0)}{}^*\tilde{w}^{(0)}=\check{\phi}_2(1_A)\otimes h$ and $\tilde{w}^{(1)}{}^*\tilde{w}^{(1)}=\check{\phi}_2(1_A)\otimes (1_\Z-h)$ and thus, these operators commute with $\phi_2(A)$ and satisfy 
\begin{equation}
\tilde{w}^{(0)}{}^*\tilde{w}^{(0)}+\tilde{w}^{(1)}{}^*\tilde{w}^{(1)}=\phi_2(1_A).
\end{equation}
As $\check{\phi}_2$ is a $^*$-homomorphism,
\begin{align}
\phi_1(a)&=\check{\phi}_1(a)\otimes h + \check{\phi}_1(a)\otimes (1_\Z-h) \nonumber\\
&= u_h(\check{\phi}_2(a)\otimes h)u_h^*+u_{1-h}(\check{\phi}_2(a)\otimes (1_\Z-h))u_{1-h}^*\nonumber\\
&=\tilde{w}^{(0)}\phi_2(a)\tilde{w}^{(0)}{}^*+\tilde{w}^{(1)}\phi_2(a)\tilde{w}^{(1)}{}^*,\quad a\in A.
\end{align}
Calculating in a very similar fashion to (\ref{e7.9}) and (\ref{e7.10}), using the fact that $\check{\phi}_1$ and $\check{\phi}_2$ are $^*$-homomorphisms gives 
\begin{align}
\phi_1(a)&=\tilde{w}^{(0)}\phi_2(a)\tilde{w}^{(0)}{}^*+\tilde{w}^{(1)}\phi_2(a)\tilde{w}^{(1)}{}^*,\quad a\in A.
\end{align}
By (\ref{eq:2colour.2}), we also have
\begin{equation}
\tilde{w}^{(0)}=(\check{\phi}_1(1_{A})\otimes h^{1/2})u_h\text{ and }\tilde{w}^{(1)}=(\check{\phi}_1(1_{A})\otimes (1_\Z-h)^{1/2})u_{1-h}
\end{equation}
so that
\begin{equation}
\tilde{w}^{(0)}{}^*=u_h^*(\check{\phi}(1_A)\otimes h^{1/2})\text{ and }\tilde{w}^{(1)}{}^*=u_{1-h}^{*}(\check{\phi}_1(1_A)\otimes (1_\Z-h)^{1/2}).
\end{equation}
Thus, just as above (now using that $\check{\phi}_1$ is a $^*$-homomorphism),
\begin{equation}
\phi_2(a)=\tilde{w}^{(0)}{}^*\phi_1(a)\tilde{w}^{(0)}+\tilde{w}^{(1)}{}^*\phi_1(a)\tilde{w}^{(1)},\quad a\in A
\end{equation}
with $\tilde{w}^{(0)}\tilde{w}^{(0)}{}^*=\check\phi_1(1_A) \otimes h \in \phi_1(A)'$ and $\tilde{w}^{(1)}\tilde{w}^{(1)}{}^*=\check\phi_1(1_A) \otimes (1_{\mathcal Z}-h) \in \phi_1(A)'$, and so $\tilde{w}^{(0)}\tilde{w}^{(0)}{}^*+\tilde{w}^{(1)}\tilde{w}^{(1)}{}^*=\phi_1(1_A)$. 
This establishes (\ref{2ColU.4}).

In case $\phi_1,\phi_2$ are unital, so are $\check\phi_1$ and $\check\phi_2$, and so
\begin{equation}
\tilde{w}^{(0)}\tilde{w}^{(0)*} = \check\phi_1(1_A)\otimes h = \check\phi_2(1_A)\otimes h = \tilde{w}^{(0)*}\tilde{w}^{(0)},
\end{equation}
and likewise, $\tilde{w}^{(1)}\tilde{w}^{(1)*}=\tilde{w}^{(1)*}\tilde{w}^{(1)}$.
\end{proof}

\begin{remark}
In Theorem \ref{thm:2ColourUniqueness}, the assumption that $\phi_1$ is totally full
is essential, as otherwise for some nonzero $a\in A_+$, one or both of $\phi_1(a)$ and $\phi_2(a)$ could lie in the trace kernel ideal $J_{B_\omega}$, where they would be positive elements with infinitesimal traces, but possibly of different sizes not seen by condition (\ref{2ColU.1}).
In contrast, the other conditions do relate how fast the trace of representative sequences of $\phi_1(a)$ and $\phi_2(a)$ vanishes.  At the worst extreme, if $\phi_2$ is the zero map, then condition (iii) would imply that so too is $\phi_1$, whereas any nonzero $^*$-homomorphism $\phi_1:A\rightarrow J_{B_\omega}\lhd B_\omega$ has $\tau\circ\phi_2=\tau\circ\phi_1$ for all $\tau\in T(B_\omega)$.
\end{remark}

Specializing the $2$-coloured uniqueness theorem to $^*$-homomorphisms $A\rightarrow B$ yields the following corollary.
\begin{cor}
Let $A$ be a separable, unital and nuclear $\mathrm C^*$-algebra, and let $B$ be a separable, simple, unital $\Z$-stable $\mathrm{C}^*$-algebra such that $QT(B)=T(B)$ and $\partial_eT(B)$ is compact and nonempty.  Let $\phi_1,\phi_2:A\rightarrow B$ be unital $^*$-homomorphisms such that $\phi_1$ is injective. Then the following are equivalent:
\begin{enumerate}
\item $\tau\circ\phi_1=\tau\circ\phi_2$ for all $\tau\in T(B)$;
\item $\phi_1$ and $\phi_2$ are approximately $n$-coloured equivalent for some $n\in\N$;
\item $\phi_1$ and $\phi_2$ are approximately $2$-coloured equivalent.
\end{enumerate}
\end{cor}
\begin{proof}
Write $\tilde{\phi}_i=\iota\circ\phi_i:A\rightarrow B_\omega$, where $\iota:B\hookrightarrow B_\omega$ is the canonical inclusion.  As $\phi_1$ is injective, it follows that $\tilde{\phi}_i$ is totally full.  By Proposition \ref{NoSillyTraces}, the limit traces $T_\omega(B_\omega)$ are dense in $T(B_\omega)$.  Accordingly, condition (i) is equivalent to $\tau\circ\tilde{\phi}_1=\tau\circ\tilde{\phi}_2$ for all $\tau\in T(B_\omega)$, and so the corollary follows from Theorem \ref{thm:2ColourUniqueness}.
\end{proof}

We also obtain results in the case when $\phi_2$ is only a c.p.c.\ order zero map such that $\phi_2^m$ agrees on traces with $\phi_1$ for all $m\in\N$; however in this case we do not obtain a symmetric decomposition.
\begin{thm}\label{thm7.8}
Let $(B_n)_{n=1}^\infty$ be a sequence of simple, separable, unital, finite $\mathcal Z$-stable $\mathrm C^*$-algebras such that $QT(B_n)=T(B_n)$ and $\partial_e T(B_n)$ is compact and nonempty for each $n\in\N$, set $B_\omega := \prod_\omega B_n$ and define $J_{B_\omega}$ as in \eqref{eq:JBdef}.
Let $A$ be a separable, unital, nuclear $\mathrm C^*$-algebra, let $\phi_1:A\rightarrow B_\omega$ be a totally full $^*$-homomorphism and $\phi_2:A\rightarrow B_\omega$ a c.p.c. order zero map with $\tau\circ\phi_1=\tau\circ\phi_2^m$ for all $\tau\in T(B_\omega)$ and $m\in\mathbb N$.  Then, there exist contractions $v^{(0)},v^{(1)},w^{(0)},w^{(1)}\in B_\omega$ such that 
\begin{align}
\notag
\phi_1(a)&=w^{(0)}\phi_2(a)w^{(0)}{}^*+w^{(1)}\phi_2(a)w^{(1)}{}^* \quad a\in A, \text{ and} \\
\phi_2(a)&=v^{(0)}\phi_1(a)v^{(0)}{}^*+v^{(1)}\phi_1(a)v^{(1)}{}^* \quad a\in A,
\end{align}
with $w^{(0)}{}^*w^{(0)},w^{(1)}{}^*w^{(1)}\in \phi_2(A)'$, $v^{(0)}{}^*v^{(0)},v^{(1)}{}^*v^{(1)}\in \phi_1(A)'$ and 
\begin{equation}
w^{(0)}{}^*w^{(0)}+w^{(1)}{}^*w^{(1)}=v^{(0)}{}^*v^{(0)}+v^{(1)}{}^*v^{(1)}=1_{B_\omega}.
\end{equation}
\end{thm}
\begin{proof}
The existence of the $w^{(i)}$ is obtained as in the proof of Theorem \ref{thm:2ColourUniqueness} (\ref{2ColU.1}) $\Longrightarrow$ (\ref{2ColU.3}), as the assumptions on $\phi_1$ and $\phi_2$ enable us to apply Theorem \ref{KeyLemmaP} exactly as in that proof.  Once one has applied Theorem \ref{KeyLemmaP}, we can then exchange the roles of $\phi_1$ and $\phi_2$ in the rest of the proof of Theorem \ref{thm:2ColourUniqueness} (\ref{2ColU.1}) $\Longrightarrow$ (\ref{2ColU.3}), producing the $v^{(i)}$ with the stated properties.
\end{proof}

\clearpage\section{Nuclear dimension and decomposition rank}
\label{sec:FiniteDim}

\noindent In this section we use the two-coloured uniqueness strategy to obtain two-coloured covering (i.e., one-dimension) results for simple, separable, unital, nuclear, $\Z$-stable $\mathrm C^*$-algebras with compact extremal tracial boundaries, proving Theorem \ref{thm:dn1}. We first recall the definitions of nuclear dimension and decomposition rank.

\begin{defn}[{\cite[Definition 3.1]{KW:IJM}, \cite[Definition 2.1]{WZ:Adv}}]
Let $A$ be a $\mathrm C^*$-algebra and let $n\in \N$.
The \emph{nuclear dimension} of $A$ is at most $n$ (written $\dimnuc(A)\leq n$) if for any finite subset $\mathcal F$ of $A$ and $\eps>0$, there exist finite dimensional $\mathrm C^*$-algebras $F^{(0)},\dots,F^{(n)}$ and maps
\begin{equation}
\xymatrix{
A \ar[r]^-{\psi} & F^{(0)} \oplus \cdots \oplus F^{(n)} \ar[r]^-{\phi} & A
}
\end{equation}
such that $\psi$ is c.p.c., $\phi|_{F^{(i)}}$ is c.p.c.\ order zero for $i=0,\dots,n$, and such that $\phi(\psi(x)) \approx_\eps x$ for $x\in \mathcal F$.
The \emph{decomposition rank} of $A$ is at most $n$ (written $\mathrm{dr}(A)\leq n$) if additionally $\phi$ can be taken contractive.
\end{defn}

The importance of quasidiagonality to the structure of stably finite simple nuclear $\mathrm{C}^*$-algebras dates back to \cite{Popa:PJM}. Now we can see that quasidiagonality, or more precisely the stronger condition that all traces are quasidiagonal in the sense of \cite{B:MAMS}, provides the essential difference between finite nuclear dimension and finite decomposition rank. We explore quasidiagonal traces further in Section  \ref{sec:TWQD}, in particular showing that all traces on a $\mathrm C^*$-algebra with finite decomposition rank are quasidiagonal.

\begin{defn}[{\cite[Definition 3.3.1]{B:MAMS}}]\label{QDTraces}
Let $A$ be a separable, unital $\mathrm C^*$-algebra. A trace $\tau$ on $A$ is \emph{quasidiagonal} if there exists a sequence $(F_n)_{n=1}^\infty$ of finite dimensional $\mathrm C^*$-algebras, a limit trace $\sigma \in T_\omega(\prod_\omega F_n)$, and a u.c.p.\ map $A \to \prod_{n=1}^\infty F_n$ which induces a ${}^*$-homomorphism $\phi:A \to \prod_\omega F_n$ such that $\tau = \sigma \circ \phi$. Write $T_{QD}(A)$ for the set of all quasidiagonal traces on $A$.
\end{defn}

The definition above formally differs from that in \cite{B:MAMS} in that the original uses full matrix algebras in place of finite dimensional algebras, and limits as $n\rightarrow\infty$ as opposed to ultrapowers. It is routine to embed a finite dimensional algebra $F$ unitally into a full matrix algebra in a fashion which approximately preserves a fixed trace on $F$ (if the trace on $F$ is a rational convex combination of extreme traces, then it can be exactly preserved). Combining this with an application of Kirchberg's $\eps$-test (Lemma \ref{epstest}) shows that the two definitions agree.

Without quasidiagonality, \cite[Proposition 3.2]{SWW:arXiv} provides, for each trace $\tau$ on a unital, nuclear $\mathrm C^*$-algebra, a map $\phi$ as in Definition \ref{QDTraces}, except that $\phi$ is only c.p.c.\ order zero rather than a ${}^*$-homomorphism.

To obtain covering dimension estimates, we patch together these trace-witnessing maps into ultraproducts of finite dimensional algebras, to obtain a map factoring through order zero maps on finite dimensional $\mathrm C^*$-algebras, which agrees on all traces with the constant-sequence embedding of $A$ into $A_\omega$.
A $\mathrm W^*$-bundle partition of unity argument allows us to do this patching.

Note that $T_{QD}(A)$ is a weak$^*$-closed convex subset of $T(A)$ (see \cite[Proposition 3.5.1]{B:MAMS}), so in Lemma \ref{lem:GoodTraceMaps} and Theorem \ref{thm:FiniteDim} there is no difference between asking for all traces to be quasidiagonal and asking for extremal traces to be quasidiagonal (which is what we use). It remains an open question whether $T_{QD}(A)$ is necessarily a face in $T(A)$ (asked after Proposition 3.5.1 in \cite{B:MAMS}).
A more important question is the following:

\begin{question}[{cf.\ \cite[Question 6.7(2)]{B:MAMS}}]
Let $A$ be a unital, nuclear, quasidiagonal $\mathrm C^*$-algebra.
Is $T(A)$ equal to $T_{QD}(A)$, i.e., is every trace quasidiagonal?
\end{question}

\begin{lemma}
\label{lem:GoodTraceMaps}
Let $A$ be a simple, separable, unital, nuclear, $\mathcal{Z}$-stable $\mathrm C^*$-algebra such that $T(A)$ is a Bauer simplex.
Then there exists a sequence $(\phi_n)_{n=1}^\infty$ of c.p.c.\ maps $\phi_n:A\rightarrow A$, which factorize through finite dimensional algebras $F_n$ as 
\begin{equation}
\label{eq:GoodTraceMapsFactor}
\xymatrix{A\ar[dr]_{\theta_n}\ar[rr]^{\phi_n}&&A\\&F_n \ar[ur]_{\eta_n}}
\end{equation}
with $\theta_n$ c.p.c.\ and $\eta_n$ c.p.c.\ order zero, such that the induced maps $(\theta_n)_{n=1}^\infty:A \to \prod_\omega F_n$ and $\Phi=(\phi_n)_{n=1}^\infty:A\rightarrow A_\omega$ are order zero and
\begin{equation}
\label{eq:GoodTraceMaps}
\tau\circ\Phi(x)=\tau(x) \quad x\in A,\ \tau\in T(A_\omega).
\end{equation}

If additionally all traces on $A$ are quasidiagonal, then each $\theta_n$ can be taken unital.
\end{lemma}

\begin{proof}
Fix a finite subset $\F$ of $A$ and $\eps>0$.  By Kirchberg's $\eps$-test (Lemma \ref{epstest}), it suffices to find a sequence $(\phi_n)_{n=1}^\infty$ as in the statement, except with \eqref{eq:GoodTraceMaps} weakened to
\begin{equation}
\label{eq:GoodTraceMapsApprox}
|\sigma \circ \Phi(x)-\sigma(x)|\leq\eps,\quad \quad x\in\F,\sigma \in T_\omega(A_\omega).
\end{equation}
Indeed, take $X_n$ to be the set of triples $(F,\theta,\eta)$, where $F$ is a finite dimensional $\mathrm{C}^*$-subalgebra,
$\theta:A\rightarrow F$ is c.p.c., $\eta:F\rightarrow A$ is c.p.c.\ order zero, and, if all traces on $A$ are quasidiagonal, then in addition, $\theta$ is unital.
Taking a dense sequence $(x^{(k)})_{k=1}^\infty$ in $A$, consider functions $f_n^{(k)}:X_n\rightarrow[0,\infty)$ given by
\begin{equation}
f_n^{(k)}(F_n,\theta_n,\eta_n)=\sup_{\sigma\in T(A)}|\sigma\circ\eta_n\circ\theta_n(x^{(k)})-\sigma(x^{(k)})|.
\end{equation}
Given a sequence $(F_n,\theta_n,\eta_n)_{n=1}^\infty$ from $\prod_{n=1}^\infty X_n$, we obtain an induced c.p.c.\ map $\Theta:A\rightarrow\prod_\omega F_n$.
By the version of Lemma \ref{orderzerotest} discussed in Remark \ref{rmk:orderzerotestplus}, there is a countable collection of functions $g^{(k)}_n:X_n\rightarrow [0,\infty)$ such that the induced map $\Theta$ is order zero if and only if $\lim_{n\rightarrow\omega}g^{(k)}_n(F_n,\theta_n,\eta_n)=0$ for all $k$.\footnote{We implicitly fix a countable dense $\mathbb Q[i]$-$^*$-subalgebra $A_0$ of $A$ to do this, but note that since each $\theta_n$ is c.p.c., it is determined by its restriction to $A_0$.}  Our condition (\ref{eq:GoodTraceMapsApprox}), shows that for any $\eps>0$ and $m\in\N$, we can find a sequence $(F_n,\theta_n,\eta_n)$ from $\prod_{n=1}^\infty X_n$ satisfying $\lim_{n\rightarrow\omega}f_n^{(k)}(F_n,\theta_n,\eta_n)\leq\eps$ for $k=1,\dots,m$ and $\lim_{n\rightarrow\omega}g_n^{(k)}(F_n,\theta_n,\eta_n)=0$ for all $k\in\mathbb N$.  Then the $\eps$-test (Lemma \ref{epstest}) provides a sequence $(F_n,\theta_n,\eta_n)$ from $\prod_{n=1}^\infty X_n$ satisfying $\lim_{n\rightarrow\omega}f_n^{(k)}(F_n,\theta_n,\eta_n)\leq\eps$ and $\lim_{n\rightarrow\omega}g_n^{(k)}(F_n,\theta_n,\eta_n)=0$ for all $k\in\mathbb N$.  Thus the map $\Theta:A\rightarrow\prod_\omega F_n$ induced by this sequence is order zero, and hence so too is the map $\Phi$ (as each $\eta_n$ is c.p.c.\ and order zero so the induced map $\prod_\omega F_n\rightarrow A_\omega$ is order zero). Further, by density of the sequence $(x^{(k)})_{k=1}^\infty$ in $A$, (\ref{eq:GoodTraceMaps}) holds for all limit traces $\tau\in T_\omega(A_\omega)$, whence by Proposition \ref{NoSillyTraces}, for all $\tau\in T(A_\omega)$.

We now proceed to produce $\phi_n$ satisfying \eqref{eq:GoodTraceMapsApprox}.
Let $\Q$ be the universal UHF-algebra.
Fix, momentarily, an extreme trace $\lambda \in K:=\partial_e T(A)$.
By \cite[Proposition 3.2]{SWW:arXiv}, there exists a c.p.c.\ order zero map $\Theta_{\lambda}:A \to \Q_\omega$ such that $\lambda(x)= \tau_{\Q_\omega} \circ \Theta_\lambda(x)$ for $x\in A$.
Using the Choi-Effros lifting theorem (\cite{CE:Ann}), we can lift $\Theta_{\lambda}$ to a sequence $(\tilde{\theta}_{\lambda,n})_{n=1}^\infty$ of c.p.c.\ maps $A\rightarrow \Q$.
Fix a dense sequence $(a_m)_{m=1}^\infty$ in the unit ball of $A$.
For each $n$, find a unital finite dimensional subalgebra $F_{\lambda,n}\subset \Q$ which approximately contains $\{\tilde{\theta}_{\lambda,n}(a_i):i=1,\dots,n\}$ up to the tolerance $1/n$, and set $\theta_{\lambda,n}:=E_{\lambda,n}\circ\tilde{\theta}_{\lambda,n}$, where $E_{\lambda,n}:\Q\rightarrow F_{\lambda,n}$ is a conditional expectation.
Thus $(\theta_{\lambda,n})_{n=1}^\infty$ also represents the map $\Theta_\lambda$.

When $T(A)=T_{QD}(A)$, then we use quasidiagonality instead of \cite[Proposition 3.2]{SWW:arXiv} in the previous paragraph to construct the maps $\Theta_{\lambda}$. In this case, each $\theta_{\lambda,n}$ is unital.

Let $(\psi_n)_{n=1}^\infty$  be a sequence of c.p.c.\ order zero maps $\Q\rightarrow\Z$ provided by \cite[Lemma 6.1]{SWW:arXiv} which induce a c.p.c.\ order zero map $\Psi:\Q_\omega\rightarrow\Z_\omega$ satisfying $\tau_{\Z_\omega}\circ\Psi(1_{\Q_\omega})=1$, so that by uniqueness of the trace on $\Q_\omega$, $\tau_{\Z_\omega}\circ\Psi=\tau_{\Q_\omega}$.
Hence for each $\lambda\in K$,
\begin{equation}\label{e8.5}
\tau_{\Z_\omega}\circ\Psi\circ\Theta_\lambda=\lambda.
\end{equation}

For each $\lambda\in K$, let $V_\lambda$ be a neighbourhood of $\lambda$ such that, for $\sigma \in V_\lambda$,
\begin{equation}
|\sigma(x)-\lambda(x)|<\frac{\eps}{3},\quad x\in\F.
\end{equation}
By compactness, there exists $\lambda_1,\dots,\lambda_l\in K$ such that $\bigcup_{i=1}^l V_{\lambda_i}=K$.
Note that if $\sigma\in\overline{V_{\lambda_{i_0}}}$ and $\overline{V_{\lambda_{i_0}}}\cap V_{{\lambda_i}}\neq\emptyset$, then 
\begin{equation}\label{GoodTraceMaps.1}
|\sigma(x)-\lambda_{i}(x)|\leq\eps,\quad x\in\F.
\end{equation}

Since $A$ is $\Z$-stable, and $K=\partial _e T(A)$ is compact, Ozawa's characterization of trivial $\mathrm W^*$-bundles shows that $\overline{A}^{\mathrm{st}}$ is isomorphic to the trivial $\mathrm{W}^*$-bundle $C_\sigma(K,\R)$ with fibres the hyperfinite II$_1$ factor $\R$ (Theorem \ref{OzawaBundles}).  Then $A_\omega/J_{A_\omega}\cong C_\sigma(K,\R)^\omega$ (see the first part of Lemma \ref{lem:CentralSurjectivity}). Let $f_1,\dots,f_l \in C(K)_+$ be a partition of unity, such that $f_i$ is supported on $V_{\lambda_i}$ and let $p_1,\dots,p_l$ be pairwise orthogonal projections in $C_\sigma(K,\R)$ with $\tau_\R(p_i(\lambda))=f_i(\lambda)$ for all $\lambda\in K$ (see Lemma \ref{W*SC:Lem1}, but note that we do not need the approximate centrality given by that lemma). Regarding these projections as embedded as constant sequences into $C_\sigma(K,\R)^\omega$, we may lift them to pairwise orthogonal contractions $d_1,\dots,d_l$ in $A_\omega$.

We can lift again to pairwise orthogonal sequences $(d_{1,n})_{n=1}^\infty,\dots,(d_{l,n})_{n=1}^\infty$ of positive contractions in $\ell^\infty(A)$.
These satisfy
\begin{equation}\label{GoodTraceMaps.Eq1}
\lim_{n\rightarrow\omega}\max_{\tau\in K}|\tau(d_{i,n})-f_i(\tau)|=0,
\end{equation}
as if this was not the case, then there would exist $\eps>0$ and $I \in \omega$ such that, for each $n\in I$ there exists $\tau_n \in K$ such that $|\tau_n(d_{i,n})-f_i(\tau_n)|\geq \eps$; then with $\tau$ equal to the limit trace on $A_\omega$ given by $(\tau_n)_{n=1}^\infty$, we obtain $|\tau(d_i)-\tau(p_i)|\geq \eps$, a contradiction.

Define c.p.c.\ maps $\tilde\phi_n:A\rightarrow A\otimes \Z$ by
\begin{equation}\label{GoodTraceMaps.Eq2}
\tilde\phi_n(x)=\sum_{i=1}^l d_{i,n}\otimes (\psi_n\circ\theta_{\lambda_i,n})(x),
\end{equation}
which induces a c.p.c.\ order zero map $\tilde\Phi:A\rightarrow (A\otimes\Z)_\omega$.
Defining $F_n:=\bigoplus_{i=1}^l F_{\lambda_i,n}$ and $\theta_n:A\rightarrow F_n,\ \tilde\eta_n:F_n\rightarrow A\otimes\Z$ by 
\begin{align}
\theta_n(x)&:=\bigoplus_{i=1}^l\theta_{\lambda_i,n}(x), \quad x \in A, \text{ and} \\
\tilde\eta_n(y_1\oplus\cdots\oplus y_l)&:=\sum_{i=1}^l d_{i,n}\otimes\psi_n(y_i), \quad y_i\in F_{\lambda_i,n},
\end{align}
we see that 
\begin{equation}
\tilde\phi_n=\tilde\eta_n\circ\theta_n,\quad n\in\N,
\end{equation}  
and $(\theta_n)_{n=1}^\infty = \sum_{i=1}^l \Theta_{\lambda_i}:A \to \prod_\omega F_n$ is order zero.
In the case that $T(A)=T_{QD}(A)$, then all the maps $\theta_n$ are unital as each $\theta_{\lambda_i,n}$ is unital.

Since $K=\partial_e T(A)$ is compact, each $f_i$ extends to a continuous affine function on $T(A)$ by 
\begin{equation}
f_i(\sigma)=\int_K f_i\,d\mu_{\sigma},\quad\sigma\in T(A),
\end{equation}
where $\mu_{\sigma}$ is the measure induced on $K$ by $\sigma$ (\cite[Theorem II.4.1]{A:book}). In this way $(f_i)_{i=1}^l$ forms a continuous affine partition of unity on $T(A)$.
For any sequence of traces $(\sigma_n)_{n=1}^\infty$ in $T(A)$, we have
\begin{eqnarray}\label{GoodTraceMaps.2}
\notag
\lim_{n\rightarrow\omega} \sigma_n(d_{i,n})&=& \lim_{n\to\omega} \int_K \tau(d_{i,n})\,d\mu_{\sigma_n}(\tau) \\
\notag
&\stackrel{(\ref{GoodTraceMaps.Eq1})}{=}&\lim_{n\to\omega} \int_K f_i(\tau)\,d\mu_{\sigma_n}(\tau) \\
&=& \lim_{n\rightarrow\omega}f_i(\sigma_n),
\end{eqnarray}
using the fact that the convergence in (\ref{GoodTraceMaps.Eq1}) is uniform.

Fix $i_0\in\{1,\dots,l\}$ and suppose that $\sigma^{(i_0)}\in T_\omega(A)$ is a limit trace defined by a sequence $(\sigma_n)_{n=1}^\infty$ of traces in the convex hull of $V_{\lambda_{i_0}}$; that is, $\sigma^{(i_0)}((x_n)_{n=1}^\infty) = \lim_{n\to\omega} \sigma_n(x_n)$ for all $(x_n)_{n=1}^\infty \in A_\omega$.
Thus each $\mu_{\sigma_n}$ (the measure on $\partial_e T(A)$ which gives $\sigma_n$ by integration) will be supported on $\overline{V_{\lambda_{i_0}}}$, and so if $f_i(\sigma_n)>0$ for some $i$ and $n$, then $\overline{V_{\lambda_{i_0}}}\cap V_{\lambda_i}\neq\emptyset$, whence for all $x\in\F$,
\begin{equation}\label{GoodTraceMaps.3}
|\sigma_n(x)-\lambda_i(x)|\leq\eps,
\end{equation}
from (\ref{GoodTraceMaps.1}).  In particular,
\begin{equation}
\label{GoodTraceMaps.3.5}
\left|\sum_{i=1}^l f_i(\sigma_n)\lambda_i(x)-\sigma_n(x)\right|\leq\eps,\quad x\in\F,
\end{equation}
holds for all $n\in\N$.  Write $\sigma^{(i_0)}\otimes\tau_\Z$ for the limit trace on $(A\otimes\Z)_\omega$ induced by the sequence $(\sigma_n\otimes\tau_\Z)_n$.  As (\ref{e8.5}), (\ref{GoodTraceMaps.Eq2}) and (\ref{GoodTraceMaps.2}) give
\begin{align}
(\sigma^{(i_0)} \otimes \tau_{\Z})(\tilde\Phi(x))&=\sum_{i=1}^l \left(\lim_{n\rightarrow\omega}{\sigma}_n(d_{i,n})\right)\left(\lim_{n\rightarrow\omega}\tau_\Z(\psi_n(\theta_{\lambda_i,n}(x)))\right)\nonumber\\
&=\lim_{n\rightarrow\omega}\sum_{i=1}^l f_i(\sigma_n)\lambda_i(x),\quad x\in A,
\end{align}
we have by \eqref{GoodTraceMaps.3.5},
\begin{equation}\label{GoodTraceMaps.4}
|(\sigma^{(i_0)}\otimes\tau_\Z)(\tilde\Phi(x))-\sigma^{(i_0)}(x)|\leq\eps,\quad x\in \F.
\end{equation}

Now suppose $\sigma \in T_\omega(A_\omega)$.
Then $\sigma$ can be written as a convex combination of $\sigma^{(1)},\dots,\sigma^{(l)}$ where $\sigma^{(i)}$ is a limit trace represented by a sequence coming from the closed convex hull of $V_{\lambda_i}$. Thus (\ref{GoodTraceMaps.4}) passes to these convex combinations and holds for all $\tau \in T_\omega(A_\omega)$.

Let $\beta:A \otimes \Z\rightarrow A$ be an isomorphism which satisfies $\lambda\circ\beta=\lambda\otimes\tau_\Z$ for any trace $\lambda\in T(A)$. Then set $\eta_n:=\beta\circ\tilde\eta_n:F_n\rightarrow A$ and $\phi_n:=\beta\circ\tilde\phi_n$, so that the maps $\theta_n$, $\eta_n$ and $\phi_n$ factorize as in (\ref{eq:GoodTraceMapsFactor}) and thus the resulting map $\Phi$ satisfies (\ref{eq:GoodTraceMapsApprox}).\end{proof}

We can now compute the nuclear dimension and decomposition rank.

\begin{thm}[Theorem \ref{thm:dn1}]
\label{thm:FiniteDim}
Let $A$ be a simple, separable, unital, nuclear, $\Z$-stable $\mathrm C^*$-algebra such that $T(A)$ is a Bauer simplex.
Then $\dimnuc(A)\leq 1$.
If additionally $T(A)=T_{QD}(A)$, then $\dr(A)\leq 1$.
\end{thm}

\begin{proof}
We shall estimate the nuclear dimension of the first factor embedding $\iota:A\rightarrow A\otimes\Z$, $\iota(x)= x\otimes 1_\Z$ in the sense of \cite[Definition 2.2]{TW:APDE}; since $A$ is $\Z$-stable and $\Z$ is strongly self-absorbing, we have $\dimnuc(A)=\dimnuc(\iota)$ and $\dr(A)=\dr(\iota)$ (see \cite[Proposition 2.6]{TW:APDE}).

Let $\Phi:A\rightarrow A_\omega$ be the c.p.c.\ order zero map of Lemma \ref{lem:GoodTraceMaps} satisfying \eqref{eq:GoodTraceMaps}, which is represented by a sequence of c.p.c.\ maps $(\phi_n)_{n=1}^\infty$ which factorize as $\eta_n\circ\theta_n$ through finite dimensional algebras $F_n$ as in (\ref{eq:GoodTraceMapsFactor}).
By \eqref{eq:JBeqDef2}, $1_{A_\omega}-\Phi(1_A) \in J_{A_\omega}$, so that the map $\bar\Phi:A \to A_\omega/J_{A_\omega}$ is a $^*$-homomorphism.
Consequently, as $T_\omega(A_\omega)$ is dense in $T(A_\omega)$ (by Proposition \ref{NoSillyTraces}), $\tau = \tau \circ \Phi = \tau \circ \Phi^n$ for any $n \in \N$ and any $\tau \in T(A_\omega)$.

Let $h$ be a positive contraction in $\Z$ of full spectrum.
Thus, applying Theorem \ref{KeyLemmaP}, we obtain unitaries $w^{(0)},w^{(1)}\in (A\otimes\Z)_\omega$ such that
\begin{align}
x \otimes h&=w^{(0)}(\Phi(x) \otimes h)w^{(0)}{}^*\\
x \otimes (1_\Z-h)&=w^{(1)}(\Phi(x)\otimes (1_\Z-h))w^{(1)}{}^*,\quad x\in A.
\end{align}
Choose representing sequences $(w^{(0)}_n)_{n=1}^\infty$ and $(w^{(1)}_n)_{n=1}^\infty$ of unitaries in $A\otimes \Z$ for $w^{(0)}$ and $w^{(1)}$ respectively.  
We have c.p.c.\ maps $\theta_n\oplus\theta_n:A\rightarrow F_n\oplus F_n$, and $\tilde{\eta}_n:F_n\oplus F_n\rightarrow A \otimes \Z$, where 
\begin{equation}
\tilde{\eta}_n(y_0,y_1)=w^{(0)}_n(\eta_n(y_0)\otimes h)w_n^{(0)}{}^*+w_n^{(1)}(\eta_n(y_1)\otimes (1_\Z-h))w_n^{(1)}{}^*.
\end{equation}
Hence, $\iota(x)$ is the limit, as $n \to \omega$, of $(\tilde{\eta}_n\circ(\theta_n\oplus\theta_n)(x))_{n=1}^\infty$ and, since $\tilde{\eta}_n$ is the sum of two c.p.c.\ order zero maps, $\dimnuc(\iota)\leq 1$.

Finally, suppose that all traces on $A$ are quasidiagonal.  Then the maps $\theta_n$ from Lemma \ref{lem:GoodTraceMaps} can be taken unital.
Ergo $1_A\otimes1_\Z$ is the limit of $\tilde{\eta}_n(1_{F_n}\oplus 1_{F_n})$ and so the c.p.\ map $\tilde{\eta}_n$ is approximately contractive.
By rescaling, it can be made contractive.
That is, $\dr(\iota)\leq 1$.
\end{proof}

We immediately obtain the following abstract classification theorem when all traces are quasidiagonal, extending that of \cite{MS:DMJ} in the unique trace case.
In the following, for a unital $\mathrm C^*$-algebra $A$, we use $[1_A]_0$ to denote the class of the unit in $K_0(A)$.

\begin{cor}
\label{cor:BauerClass}
Let $\mathcal C_0$ denote the class of all simple, separable, unital, nuclear, $\Z$-stable $\mathrm C^*$-algebra $A$ for which the following hold:
\begin{enumerate}
\item $T(A)$ a Bauer simplex;
\item all traces on $A$ are quasidiagonal;
\item projections in $A$ separate traces;
\item $A$ satisfies the UCT.
\end{enumerate}
If $A,B \in \mathcal C_0$ then $A \cong B$ if and only if $K_0(A) \cong K_0(B)$ as partially ordered groups with order units $[1_A]_0$ and $[1_B]_0$ respectively, and $K_1(A) \cong K_1(B)$ as groups. \end{cor}

\begin{proof}
By Theorem \ref{thm:FiniteDim}, every algebra in $\mathcal C_0$ has finite decomposition rank.
The result thus follows by \cite[Corollary 5.2]{W:Invent1}, where the key points are that algebras in $\mathcal C_0$ are rationally real rank zero and thus rationally TAF (\cite[Theorem 4.1]{W:JLMS}), allowing the application of the classification technique from \cite{Win:localizingEC}, refined by Lin and Lin-Niu in \cite{L:appendix,LN:Adv}.
\end{proof}

\clearpage\section{Quasidiagonal traces}
\label{sec:TWQD}

\noindent
In \cite{B:MAMS}, NB explored approximation properties for traces, showing in particular that all traces on ASH algebras are quasidiagonal (and in fact have the stronger approximation property of being uniform locally finite dimensional (\cite[Corollary 4.4.4]{B:MAMS})). Here our objective is to show that very many nuclear, quasidiagonal $\mathrm C^*$-algebras have the property that all traces are quasidiagonal, namely those that can be tracially approximated by algebras of finite decomposition rank (Corollary \ref{cor:TWQDTracialdr}). In particular, to get finite decomposition rank in Theorem \ref{thm:dn1}, it is necessary to ask that all traces are quasidiagonal (see Corollary \ref{cor:drTWQD}).

The following technical lemma is our main tool.

\begin{lemma}
\label{lem:TWQDTool}
Let $A$ be a separable, unital $\mathrm C^*$-algebra.
Suppose that for $\tau \in T(A)$, there exists a sequence $(C_n)_{n=1}^\infty$ of unital $\mathrm C^*$-algebras with unital $\mathrm C^*$-subalgebras $B_n \subseteq C_n$, quasidiagonal traces $\sigma_n \in T(B_n)$, and a u.c.p.\ map $(\phi_n)_{n=1}^\infty:A \to \prod_n C_n$ which induces an injective ${}^*$-homomorphism $\phi:A \to \prod_\omega C_n$ such that:
\begin{enumerate} 
\item $\phi(A) \subseteq \prod_\omega B_n$;
\item $\tau = \lim_{n\to\omega} \sigma_n \circ \phi_n$.
\end{enumerate}
Then $\tau$ is quasidiagonal.
\end{lemma}

\begin{remark}
The purpose of the algebras $C_n$ is that condition (i) is much weaker than asking that $\phi_n(A) \subseteq B_n$ for each $n$.
\end{remark}

\begin{proof}
Let $\tau \in T(A)$, let $\mathcal F \subset A$ be a finite set of contractions, and let $\eps > 0$. 
Let $n_0$ be such that $\phi_{n_0}$ is $(\mathcal F,\eps)$-approximately multiplicative and $(\mathcal F,\eps)$-approximately isometric, and such that for each $x \in \mathcal F$, there exists $y_x \in B_{n_0}$ satisfying $\phi_{n_0}(x) \approx_\eps y_x$ and $\sigma_{n_0}(y_x) \approx_\eps \tau(x)$.
Set
\begin{equation} \mathcal G:=\{y_x \mid x \in \mathcal F\} \subset B_{n_0}. \end{equation}
As $\sigma_{n_0}$ is quasidiagonal, we can find a finite dimensional algebra $F$, a trace $\tau_F \in T(F)$, and a u.c.p.\ map $\psi:B_{n_0} \to F$ satisfying $\tau_F \circ \psi(a) \approx_{\eps}\sigma_{n_0}(a)$ for $a \in \mathcal G$, such that $\psi$ is $(\mathcal G,\eps)$-approximately multiplicative and $(\mathcal G,\eps)$-approximately isometric, i.e.,
\begin{equation}
\psi(ab) \approx_\eps \psi(a)\psi(b), \quad \text{ and} \quad
\|\psi(a)\| \approx_\eps \|a\|, \quad a,b \in \mathcal G.
\end{equation}
By Arveson's extension theorem (\cite{A:Acta}), $\psi$ extends to a u.c.p.\ map $\tilde\psi:C_{n_0} \to F$.
Then the composition $\tilde\psi\circ\phi_{n_0}$ will be u.c.p., $(\mathcal F,5\eps)$-approximately multiplicative, $(\mathcal F,4\eps)$-approximately isometric, and satisfies $\tau_F \circ \tilde\psi\circ \phi_{n_0}\approx_{\mathcal F,3\eps}\tau$.
This shows that $\tau$ is quasidiagonal.
\end{proof}

\begin{prop}
\label{prop:TracialTWQD}
Let $\mathcal S$ be a class of unital $\mathrm C^*$-algebras for which all traces are quasidiagonal.
If $A$ is a simple, separable, unital $\mathrm C^*$-algebra that is TA$\mathcal S$, in the sense of \cite[Definition 2.2]{EN:JFA}, then $T(A)=T_{QD}(A)$.\end{prop}

\begin{proof}
It follows directly from the definition of TA$\mathcal S$ that there exists a sequence of algebras $(B_n)_{n=1}^\infty$ from $\mathcal S$ with $B_n \subset A$, such that the u.c.p.\ map $\phi:A \to \prod_\omega 1_{B_n}A1_{B_n}$ given by
\begin{equation} \phi(a) := (1_{B_n}a1_{B_n})_n \end{equation}
is an injective ${}^*$-homomorphism satisfying $\phi(A) \subseteq \prod_\omega B_n$, and such that for each trace $\tau \in T(A)$,
\begin{equation} \tau = (\lim_{n\to\omega} \tau|_{B_n}) \circ \phi. \end{equation}
Since $T(B_n)=T_{QD}(B_n)$ for all $n\in\N$, Lemma \ref{lem:TWQDTool} implies that the same holds for $A$.
\end{proof} 

To show that traces on certain $\mathrm C^*$-algebras are quasidiagonal, we need to know the structure of traces on ultraproducts of finite dimensional algebras.  
\begin{lemma}
\label{lem:FDLimitTraces}
Let $(F_n)_{n=1}^\infty$ be a sequence of finite dimensional $\mathrm C^*$-algebras.
If $\tau$ is a tracial state on $\prod_\omega F_n$ and $S \subset \prod_\omega F_n$ is a separable and self-adjoint subset, then there is $\tau' \in T_\omega(\prod_\omega F_n)$ such that $\tau|_S = \tau'|_S$.
\end{lemma}

\begin{proof}
This can be proven using exactly the same Hahn-Banach argument as in \cite[Theorem 8]{O:JMSUT}, but using \cite[Lemma 3.5]{Fack:AIF} in place of \cite[Theorem 6]{O:JMSUT}; alternatively one can appeal to \cite[Theorem 1.2]{NR:arXiv}.
\end{proof}

\begin{prop}
\label{prop:drTWQD}
Let $A$ be a separable, unital $\mathrm C^*$-algebra with finite decomposition rank.
Then $T(A)=T_{QD}(A)$.
\end{prop}

\begin{proof}
When $A$ has finite decomposition rank \cite[Proposition 5.1]{KW:IJM} provides maps
\begin{equation} A \labelledrightarrow{\psi} \prod_\omega F_n \labelledrightarrow{\phi} A_\omega, \end{equation}
such that each $F_n$ is finite dimensional, $\psi$ is a unital ${}^*$-homomorphism that lifts to a u.c.p.\ map $A \to \prod_n F_n$, $\phi$ is a sum of $m$ c.p.c.\ order zero maps and is c.p.c.\ itself, and
\begin{equation} a = \phi\circ\psi(a) \end{equation}
for all $a \in A$.

For $\tau \in T(A)$, extend $\tau$ (in the canonical way) to $A_\omega$, and then set $\sigma_0:=\tau \circ \phi:\prod_\omega F_n \to \mathbb C$.
Since $\phi$ is a sum of c.p.c.\ order zero maps, $\sigma_0$ is a sum of traces (by \cite[Corollary 4.4]{WZ:MJM}).
Moreover, $\tau = \sigma_0 \circ \psi$, and hence $\sigma_0(1_{\prod_\omega F_n})=1$.
This shows that $\sigma_0 \in T(\prod_\omega F_n)$.
By Lemma \ref{lem:FDLimitTraces}, there exists a limit trace $\sigma \in T_\omega(\prod_\omega F_n)$ that agrees on $\psi(A)$ with $\sigma_0$, whence we have
\begin{equation} \tau = \sigma \circ \psi. \end{equation}
This is the definition of quasidiagonality of $\tau$ from Definition \ref{QDTraces}.
\end{proof}

\begin{cor}
\label{cor:drTWQD}
Let $A$ be a simple, separable, unital, infinite dimensional $\mathrm C^*$-algebra such that $T(A)$ is a  Bauer simplex.
The following are equivalent:
\begin{enumerate}
\item $A$ is nuclear, $\mathcal Z$-stable, and $T(A)=T_{QD}(A)$;
\item $A$ has finite nuclear dimension and $T(A)=T_{QD}(A)$;
\item $A$ has finite decomposition rank;
\item $\dr(A) \leq 1$.
\end{enumerate}
\end{cor}

\begin{proof}
(i) $\Longrightarrow$ (iv) is by Theorem \ref{thm:FiniteDim}.
(iv) $\Longrightarrow$ (iii) is trivial.
For (iii) $\Longrightarrow$ (ii), note that the nuclear dimension of a $\mathrm C^*$-algebra is always at least its decomposition rank, and $T(A)=T_{QD}(A)$ comes from Proposition \ref{prop:drTWQD}.
Finally (ii) $\Longrightarrow$ (i) combines the main result of \cite{W:Invent2} and the fact that finite nuclear dimension implies nuclearity (\cite[Remark 2.2(i)]{WZ:Adv}).
\end{proof}

Combining Propositions \ref{prop:TracialTWQD} and \ref{prop:drTWQD}, we arrive at the following consequence.

\begin{cor}
\label{cor:TWQDTracialdr}
If $A$ is a separable, unital $\mathrm C^*$-algebra, that is tracially approximated by unital $\mathrm C^*$-algebras of finite decomposition rank, then $T(A)=T_{QD}(A)$.
\end{cor}

Being tracially approximated by unital, finite decomposition rank $\mathrm C^*$-algebras is a very weak finiteness condition.
In particular, it includes algebras of locally finite decomposition rank (\cite[Definition 1.2]{W:JFA}), as well as TAI algebras and even TA(1NCCW) algebras (simple algebras of this form have recently been classified, under UCT- and nuclearity-hypotheses (\cite{GLN:arXiv})). There are no simple, stably finite, nuclear $\mathrm C^*$-algebras which are known not to have locally finite decomposition rank.

Outside the nuclear setting there exist nonquasidiagonal traces on quasidiagonal (even residually finite dimensional) $\mathrm C^*$-algebras (\cite[Corollary 4.3.8]{B:MAMS}).
The obstruction is failure of amenability, and it remains open whether all amenable traces on quasidiagonal algebras must be quasidigaonal.

\begin{question}\label{AllTracesQD?}
Are all traces on simple, nuclear, quasidiagonal $\mathrm C^*$-algebras quasidiagonal?
\end{question}

Note that \cite[Lemma 6.1.20]{B:MAMS} shows that if the answer to Question \ref{AllTracesQD?} is yes, then the same is true without assuming simplicity.  In particular the strongest form of the Toms-Winter conjecture, which predicts that finite, simple, separable, unital, nuclear, $\Z$-stable $\mathrm C^*$-algebras have finite decomposition rank, would imply that all traces on a stably finite, nuclear $\mathrm C^*$-algebra are quasidiagonal. This is stronger than the Blackadar-Kirchberg conjecture, that stably finite, nuclear $\mathrm C^*$-algebras are quasidiagonal (\cite{BK:MA}). It is remarkable that the problem is also open -- and at this point seems just as hard -- for strongly self-absorbing $\mathrm{C}^{*}$-algebras.

We end this section by observing how the homotopy rigidity theorem follows from these classification results.
The following argument was kindly provided to us by Narutaka Ozawa. It allows the removal of the additional statement that all traces are quasidiagonal (appearing in an earlier version of this paper) as an explicit hypothesis in Corollary \ref{cor8.6} below.

\begin{prop}
\label{prop:Ozawa}
Let $A,B$ be unital $\mathrm C^*$-algebras.
Suppose that projections in $A$ separate traces, $B$ homotopy dominates $A$, and that $T(B)=T_{QD}(B)$.
Then $T(A)=T_{QD}(A)$.
\end{prop}

\begin{proof}
Let $\tau \in T(A)$.
Suppose that $\alpha:A \to B$ and $\beta:B \to A$ are $^*$-homomorphisms such that $\beta \circ \alpha$ is homotopic to $\mathrm{id}_A$.
Since homotopic projections have the same trace, we see that
\[ \tau(p) = \tau \circ \beta \circ \alpha(p) \]
for every projection $p \in A$.
Since projections separate traces on $A$, it follows that $\tau = \tau \circ \beta \circ \alpha$.
Finally, since $\tau \circ \beta$ is quasidiagonal by assumption, we conclude that $\tau$ is quasidiagonal.
\end{proof}

\begin{cor}[Theorem \ref{thm:Rigidity} (iii)]\label{cor8.6}
Let $A$ be a simple, separable, unital, nuclear, $\Z$-stable $\mathrm C^*$-algebra such that $T(A)$ is a Bauer simplex. If projections separate traces on $A$, then $A$ is homotopy rigid.
\end{cor}
\begin{proof}
If $A$ is homotopic to an ASH algebra then it satisfies the UCT.
By \cite[Corollary 2.2]{NW:ComptesRendus} and Corollary \ref{cor:TWQDTracialdr}, all traces on an ASH algebra are quasidiagonal; thus, by Proposition \ref{prop:Ozawa}, $T(A)=T_{QD}(A)$.
Consequently, $A$ is in the class $\mathcal C_0$ of Corollary \ref{cor:BauerClass}. Further by \cite{E:CMS}, there exists an ASH-algebra of slow dimension growth $B\in \mathcal C_0$ with the same Elliott invariant as $A$. Thus $A\cong B$ by Corollary \ref{cor:BauerClass}.
\end{proof}
 
\clearpage\section{Kirchberg algebras}
\label{sec:KirAlgs}

\noindent
In this section we show how our methods also give the optimal bound of $1$ for the nuclear dimension of Kirchberg algebras (Corollary \ref{DimKirchberg}), generalizing the UCT case, handled using higher rank graph algebras in \cite{RSS:arXiv}. Finiteness -- with an upper bound of 5 -- of nuclear dimension for UCT Kirchberg algebras was already established in \cite{WZ:Adv}; this bound was later improved by Enders in \cite{End:dimnuc}. Two proofs that Kirchberg algebras without the UCT have nuclear dimension at most 3 where given in \cite{MS:DMJ} and \cite{BEMSW:Z-2}. The merit of these approaches is that they do not require the Kirchberg-Phillips classification and hence the UCT. Our result builds on \cite{MS:DMJ,BEMSW:Z-2}, using the $2$-coloured strategy, to obtain the optimal bound of $1$ without assuming the UCT.

Recall that a \emph{Kirchberg algebra} is, by definition, a simple, separable, nuclear, purely infinite $\mathrm C^*$-algebra. Such an algebra is automatically $\mathcal O_\infty$-stable (and therefore $\Z$-stable), by Kirchberg's celebrated $\mathcal O_\infty$-absorption theorem (\cite{Kir:ICM} or \cite[Theorem 3.15]{KP:Crelle}). Indeed, these algebras can be characterized as simple, separable, nuclear, $\mathcal Z$-stable $\mathrm C^*$-algebras with no (densely defined) traces.

The key ingredient is Theorem \ref{thm:TotFullCpcKirClass} below, which provides a uniqueness theorem for order zero maps into ultraproducts of Kirchberg algebras. Results of this nature can be obtained from Kirchberg's approach to the classification of purely infinite algebras (\cite{K:inPrep!}) (see for example \cite[Theorem 8.3.3]{R:Book}), but we do not know of a reference to the exact statement below in the literature, and so take this opportunity to show how it can be obtained via simpler versions of the methods used above in the stably finite case.

\begin{thm}
\label{thm:TotFullCpcKirClass}
Let $(B_n)_{n=1}^\infty$ be a sequence of unital Kirchberg algebras, write $B_\omega:=\prod_\omega B_n$ and let $A$ be a separable, unital, nuclear $\mathrm C^*$-algebra.
Let $\phi^{(1)},\ \phi^{(2)}:A\rightarrow B_\omega$ be c.p.c.\ order zero maps such that $f(\phi^{(i)})$ is injective for every nonzero $f\in C_0((0,1])_+$, for $i=1,2$.
Then $\phi^{(1)}$ and $\phi^{(2)}$ are unitarily equivalent.
\end{thm}

The first crucial step is the following Stinespring type result, obtained from the technical ingredients in the proof of Kirchberg's embedding theorem in \cite{KP:Crelle}. It can be proven in just the same way as \cite[Corollary 6.3.5(ii)]{R:Book}, which handles the situation of ultrapowers.

\begin{lemma}[{cf.\ \cite[Corollary 6.3.5(ii)]{R:Book}}]\label{lem:KUniquenessLem}
Let $A$ be a separable, unital, exact $\mathrm C^*$-algebra, and $(B_n)_{n=1}^\infty$ a sequence of unital Kirchberg algebras. Write $B_\omega:=\prod_\omega B_n$.
Let $(\rho_n)_{n=1}^\infty$ be a sequence of unital completely positive maps $\rho_n:A\rightarrow B_n$ which induce an injective $^*$-homomorphism $\rho:A\rightarrow B_\omega$.
Given any sequence $(\sigma_n)_{n=1}^\infty$ of unital completely positive maps $\sigma_n:A\rightarrow B_n$, inducing $\sigma:A\rightarrow B_\omega$, there is an isometry $s\in B_\omega$ such that $s^*\rho(a)s=\sigma(a)$ for $a\in A$.
\end{lemma}

We use this to provide purely infinite versions of the two key steps in the stably finite case.
In place of the strict comparison results from Section \ref{sec:StrictComp}, we show that, in the Kirchberg algebra setting, totally full elements in relative commutant sequence algebras represent maximal elements in the Cuntz semigroup.
This verifies the technical hypothesis of Lemma \ref{thm:TotallyFullClass}, and so one immediately obtains the purely infinite version of Theorem \ref{cor:TotallyFullClassFinite} (Lemma \ref{PI:Classtotallyfull}).

\begin{lemma}\label{lem:SimplePI}
Let $A$ be a separable, unital, nuclear $\mathrm C^*$-algebra and let $(B_n)_{n=1}^\infty$ be a sequence of unital Kirchberg algebras. Write $B_\omega:=\prod_\omega B_n$ and let $\pi:A\rightarrow B_\omega$ be a c.p.c.\ order zero map. 
Write
\begin{equation}\label{lem:SimplePI.DefC}
C:=B_\omega\cap\pi(A)'\cap \{1_{B_\omega}-\pi(1_{A})\}^\perp.
\end{equation}
Let $D$ be a full hereditary subalgebra of $C$, and let $b\in D_+$ be totally full. Then for every $c\in D_+$ there exists $t\in D$ with $t^*bt=c$.
\end{lemma}

\begin{proof}
Assume, without loss of generality, that $\|b\|=1$.
Set $J:=\{a\in A:\pi(a)C=0\}$, which is a left ideal and satisfies $J=J^*$ (using the fact that $C$ commutes with $\pi(A)$).
Thus $J\lhd A$ is an ideal.

Fix $c \in D_+$.
Define maps $\psi_1,\psi_2:C_0((0,1]) \otimes (A/J) \to B_\omega$ by
\begin{equation}
\label{eq:SimplePIpsiDef}
\psi_1(f \otimes (a+J)):=f(b)\pi(a)\text{ and }\\
\psi_2(f \otimes (a+J)) := f(1)c\pi(a)
\end{equation}
for $f\in C_0((0,1]),\ a+J\in A/J$.  Note that these maps are well defined, as if $a_1+J=a_2+J$, then $f(b)\pi(a_1)=f(b)\pi(a_2)$ for all $f\in C_0((0,1])$ and $c\pi(a_1)=c\pi(a_2)$.    The map $a+J\mapsto c\pi(a)=c^{1/2}\pi(a)c^{1/2}$ is c.p.c.~ and order zero, and hence $\psi_2$ is c.p.c.~ and order zero as it is the tensor product of this map and $f\mapsto f(1)$ (we actually only care that it is c.p.c.).  The map $\psi_1$ is a ${}^*$-homomorphism, since for $f_1,f_2 \in C_0((0,1]),\ a_1+J,a_2+J \in A/J$, we have
\begin{eqnarray}
\notag
\psi_1(f_1 \otimes (a_1\otimes J))\psi_1(f_2 \otimes (a_2\otimes J))
&\stackrel{\eqref{eq:SimplePIpsiDef}}=& f_1(b)\pi(a_1)f_2(b)\pi(a_2) \\
\notag
&=& f_1(b)f_2(b)\pi(a_1)\pi(a_2) \\
\notag
&\stackrel{\eqref{eq:Ord0Ident}}=& f_1(b)f_2(b)\pi(1_{A})\pi(a_1a_2) \\
\notag
&\stackrel{b \vartriangleleft \pi(1_{A})}=& f_1(b)f_2(b)\pi(a_1a_2) \\
&=& \psi_1((f_1f_2) \otimes (a_1a_2+J)).
\end{eqnarray}

If $\psi_1$ was not injective, then there would exist some nonzero elementary tensor $f\otimes (a+J)$ in $\ker\psi_1$ (by Kirchberg's slice lemma, see \cite[Lemma 4.1.9]{R:Book}, for example). Replacing with $(f\otimes (a+J))^*(f\otimes (a+J))$ if necessary, we may assume $f$ and $a$ are  positive. Then consider 
\begin{equation}
I:=\{x\in C:x\pi(a)=0\}
\end{equation}
so that $f(b)\in I$. As $\pi(A)$ commutes with $C$, $I$ is an ideal in $C$.  Since $\|b\|=1$, $b$ is totally full in $D$, and $D$ is a full hereditary subalgebra of $C$, $f(b)$ is full in $C$. Then $I=C$, and hence $a\in J$, giving a contradiction.  Thus $\psi_1$ is injective.

Denote by $\psi_1^\sim,\psi_2^\sim:(C_0((0,1])\otimes A)^\sim\rightarrow B_\omega$ the unital c.p.\ maps obtained by unitizing  $\psi_1$ and $\psi_2$. As $1_{B_\omega}\notin\psi_1(C_0((0,1])\otimes (A/J))$,  $\psi_1^\sim$ is a unital injective $^*$-homomorphism. 

As $A$ is nuclear, each $\psi_i^\sim$ lifts to a sequence of unital c.p.\ maps $(C_0((0,1]) \otimes A)^\sim \to B_n$ by the Choi-Effros lifting theorem.
We may therefore apply Lemma \ref{lem:KUniquenessLem} to $\rho:=\psi_1^\sim$ and $\sigma:=\psi_2^\sim$, yielding an isometry $s \in B_\omega$ such that $s^*\psi_1^\sim(\cdot)s = \psi_2^\sim(\cdot)$.
Plugging in the definitions of $\psi_1,\psi_2$, this says that
\begin{equation}
s^*f(b)\pi(a)s = f(1)c\pi(a), \quad f \in C_0((0,1]),\ a\in A.
\end{equation}

Define $t:=bs$, so that (as $b,c\vartriangleleft \pi(1_{A})$), 
\begin{equation}
t^*t=s^*b^2s=s^*b^2\pi(1_{A})s=c\pi(1_{A})=c\vartriangleleft \pi(1_{A}).
\end{equation}
Likewise, $t^*bt=c$ and
\begin{equation}
t^*\pi(a)t=s^*b^2\pi(a)s=c\pi(a)=t^*t\pi(a),\quad a\in A.
\end{equation}
Therefore $t\in B_\omega\cap \pi(A)'\cap \{1_{B_\omega}-\pi(1_{A})\}^\perp$, by Lemma \ref{lem:CommLem}.   Since $t^*t=c\in D$ and $tt^*=bss^*b\leq b^2\in D$, and $D$ is hereditary in $C$, we have $t\in D$ as required.
\end{proof}

\begin{remark}
When $A$ is additionally simple, similar methods show that the algebra $C=B_\omega\cap \pi(A)'\cap \{1_{B_\omega}-\pi(1_{A})\}^\perp$ from Lemma \ref{lem:SimplePI} is simple and purely infinite (analogous to Kirchberg's result that $B_\omega\cap B'$ is simple and purely infinite when $B$ is a Kirchberg algebra, see \cite[Proposition 3.4]{KP:Crelle}). Indeed, given nonzero positive contractions $b$ and $c$ in $C$, write $X$ for the spectrum of $b$ with $0$ removed, and define maps $\psi_1,\psi_2:C_0(X)\otimes A\rightarrow B_\omega$ by $\psi_1(f\otimes a)=f(b)\pi(a)$ and $\psi_2(f\otimes a)=f(1)c\pi(a)$. Simplicity of $A$ can be used to show that $\psi_1^\sim$ is a unital injective $^*$-homomorphism, and then we can follow the rest of the above proof to obtain some $t\in C$ with $t^*bt=c$.
\end{remark}

\begin{lemma}\label{PI:Classtotallyfull}
Let $A$ be a separable, unital, nuclear $\mathrm C^*$-algebra and let $(B_n)_{n=1}^\infty$ be a sequence of unital Kirchberg algebras. Write $B_\omega:=\prod_\omega B_n$ and let $\pi:A\rightarrow B_\omega$ be a c.p.c.\ order zero map.  Define $C$ as in (\ref{lem:SimplePI.DefC}).  Suppose that $a,b\in C_+$ are totally full elements of norm one.  Then $a$ and $b$ are unitarily equivalent in the unitization of $C$.
\end{lemma}

\begin{proof}
We verify the technical condition of Lemma \ref{thm:TotallyFullClass}. (The algebra $C$ is nonzero whenever $\pi$ is nonzero, whence $C$ is full in the simple purely infinite $\mathrm{C}^{*}$-algebra $B_{\omega}$.) So let $D$ be a full hereditary subalgebra of $C$ and suppose $x\in D$ has totally full positive contractions $e_l,e_r\in D_+$ such that $e_lx=xe_r=0$.   By Lemma \ref{lem:SimplePI} there exists $t\in D$ with $t^*e_lt=e_r$.  Then set $s:=t^*e_l$ so that $sx=0$.
Since $t^*e_l^2t\leq t^*e_lt=e_r$, we also have $xs=0$. Since $st=e_r$ is full in $D$, so too is $s$.

Now given totally full norm one elements $a,b\in C_+$ and a nonzero $f\in C_0((0,1])_+$, it follows that $f(a)$ and $f(b)$ are also totally full in $C$. By Lemma \ref{lem:SimplePI}, $f(a)$ is Cuntz equivalent to $f(b)$ in $C$ for all $f\in C_0((0,1])_+$. Thus $a$ and $b$ are unitarily equivalent in the unitization of $C$ by Lemma \ref{thm:TotallyFullClass}.
\end{proof}

In order to apply the preceding lemma to prove Theorem \ref{thm:TotFullCpcKirClass}, we need one more step to ensure that $\phi(1_A)$ is totally full in the appropriate relative commutant sequence algebra.

\begin{lemma}\label{lem:PI.TotallyFull}
Let $A$ be a separable, unital, nuclear $\mathrm C^*$-algebra and let $(B_n)_{n=1}^\infty$ be a sequence of unital Kirchberg algebras. Write $B_\omega:=\prod_\omega B_n$. Let $\phi:A\rightarrow B_\omega$ be a c.p.c.\ order zero map such that $f(\phi)$ is injective for every nonzero $f\in C_0((0,1])_+$. Let $\hat{\phi}:A\rightarrow B_\omega$ be a supporting order zero map for $\phi$.  Then $\phi(1_{A})$ is totally full in $C:=B_\omega\cap \hat{\phi}(A)'\cap\{1_{B_\omega}-\hat{\phi}(1_{A})\}^\perp$.
\end{lemma}
\begin{proof}
Clearly $\|\phi(1_A)\|=1$.
Fix $g\in C_0((0,1])_{+}$ with $\|g\|=1$ and a positive contraction $c\in C_+$.  We will show that there exists $t\in C$ with $t^*g(\phi(1_{A}))t=c$.  Define maps $\psi_1,\psi_2:C_0((0,1])\otimes A\rightarrow B_\omega$ by 
\begin{align}
\psi_1(f\otimes a)&=f(g(\phi(1_A)))\hat{\phi}(a)=(f\circ g)(\phi(1_{A}))\hat{\phi}(a)=(f\circ g)(\phi)(a),\nonumber\\
\psi_2(f\otimes a)&=f(1)c\hat{\phi}(a).
\end{align}
Then, as in Lemma \ref{lem:SimplePI}, $\psi_1$ is a $^*$-homomorphism: for $f_1,f_2\in C_0((0,1])$ and $a_1,a_2\in A$, 
\begin{align}
\nonumber
\psi_1(f_1\otimes a_1)\psi_1(f_2\otimes a_2)&=f_1(g(\phi(1_A)))\hat{\phi}(a_1)f_2(g(\phi(1_A)))\hat{\phi}(a_2)\\
\nonumber
&=(f_1f_2)(g(\phi(1_A)))\hat{\phi}(1_A)\hat{\phi}(a_1a_2)\\
&=\psi_1(f_1f_2\otimes a_1a_2),
\end{align}
as $(f_1f_2)(g(\phi(1_A)))\lhd \hat{\phi}(1_A)$.  Further, $\psi_1$ is injective, as otherwise (using Kirchberg's slice lemma) there would be a positive nonzero elementary tensor $f\otimes a$ in $\ker\psi_1$. That is $(f\circ g)(\phi)(a)=0$, so that $f\circ g=0$, by the injectivity hypothesis on $\phi$.  Since $g((0,1])=(0,1]$, this implies $f=0$, a contradiction. Also, as in  Lemma \ref{lem:SimplePI}, $\psi_2$ is c.p.c.\ as it is the tensor product of two commuting c.p.c.\ order zero maps. Unitizing, we obtain maps $\psi_i^\sim:(C_0((0,1])\otimes A)^\sim\rightarrow B_\omega$ with $\psi_1^\sim$ a unital injective $^*$-homomorphism.  Now we can follow the last two paragraphs of the proof of Lemma \ref{lem:SimplePI} (with $\pi$ replaced by $\hat{\phi}$) to obtain $t\in C$ with $t^*g(\phi(1_A))t=c$.  Thus $g(\phi(1_A))$ is full in $C$.  \end{proof}

\begin{proof}[Proof of Theorem \ref{thm:TotFullCpcKirClass}]
By Lemma \ref{lem:SupportingMap}, let $\hat\phi_i$ be a supporting c.p.c.\ order zero map for $\phi_i$, for $i=1,2$.
Define $\pi:A \to M_2(B_\omega)$ by
\begin{equation}
\pi(a) := \begin{pmatrix} \hat\phi_1(a) & 0 \\ 0 & \hat\phi_2(a) \end{pmatrix}, \quad a\in A.
\end{equation}
Note that $\pi$ is a supporting order zero map (satisfying \eqref{eq:Supporting}) for both
\begin{equation}
\begin{pmatrix}\phi_1(\cdot)&0\\0&0\end{pmatrix}\text{ and }\begin{pmatrix}0&0\\0&\phi_2(\cdot)\end{pmatrix}. 
\end{equation}
Set \begin{equation}
h_1:=\begin{pmatrix}\phi_1(1_A)&0\\0&0\end{pmatrix},\ h_2:=\begin{pmatrix}0&0\\0&\phi_2(1_A)\end{pmatrix}\in C\subset M_2(B_\omega).
\end{equation}
These elements clearly have norm one and Lemma \ref{lem:PI.TotallyFull} shows that $h_1$ and $h_2$ are totally full in $C:=M_2(B_\omega)\cap \pi(A)'\cap \{1_{M_2(B_\omega)}-\pi(1_A)\}^\perp$ (where we think of $M_{2}(B_{\omega})$ as $\prod_{\omega} M_{2}(B_{n})$).
Thus $h_1$ and $h_2$ are unitarily equivalent in $C^\sim$ by Lemma \ref{PI:Classtotallyfull}.  Since the condition on $\phi_1$ and $\phi_2$ also ensures that $\phi_1(1_A)$ and $\phi_2(1_A)$ are totally full in $B_\omega$, the $2\times 2$ matrix trick of Lemma \ref{lem:CombinedUnitaryEquiv} shows that $\phi_1$ and $\phi_2$ are unitarily equivalent.
\end{proof}

The final ingredient in computing the nuclear dimension of Kirchberg algebras are maps, analogous to those in Lemma \ref{lem:GoodTraceMaps}, which factor through finite dimensional $\mathrm C^*$-algebras.  We obtain these through Voiculescu's quasidiagonality of cones (\cite{V:Duke}).
\begin{lemma}
\label{lem:GoodKirchbergMaps}
Let $B$ be a unital Kirchberg algebra.
Then there exists a sequence $(\phi_n)_{n=1}^\infty$ of c.p.c.\ maps $\phi_n:B\rightarrow B$, which factorize through matrix algebras $F_n$ as
\begin{equation}\label{lem:GoodKirchbergMaps.1}
\xymatrix{B\ar[dr]_{\theta_n}\ar[rr]^{\phi_n}&&B\\&F_n \ar[ur]_{\eta_n}}
\end{equation}
with $\theta_n$ c.p.c.\ and $\eta_n$ a $^*$-homomorphism, such that the induced map $\Phi:=(\phi_n)_{n=1}^\infty:B\rightarrow B_\omega$ is a c.p.c.\ order zero map for which $(\Phi-t)_+$ is injective for every $t\in[0,1)$.
\end{lemma}

\begin{proof}
The cone over $B$ is quasidiagonal (\cite{V:Duke}), so there exist matrix algebras $F_n$ and an isometric c.p.c.\ order zero map $B \to \prod_\omega F_n$ with a c.p.c.\ lift $(\theta_n:B \to F_n)_{n=1}^\infty$.
Further, since $B$ contains a copy of $\mathcal O_\infty$, we may find an injective $^*$-homomorphism $\eta_n:F_n \to B$.
Now, simply define $\phi_n:=\eta_n\circ\theta_n$.

For $t\in[0,1)$, since $\Phi$ is isometric, $(\Phi-t)_+(1_{B})=(\Phi(1_{B})-t)_+$ has norm $1-t$.
Since $B$ is simple and $(\Phi-t)_+$ is order zero, it follows using the structure of order zero maps that $(\Phi-t)_+$ is injective.
\end{proof}

\begin{lemma}
\label{lem:TotFullCpc}
Let $A,B$ be $\mathrm C^*$-algebras, let $\phi:A \to B$ be a c.p.c.\ order zero map such that $(\phi-t)_+$ is injective for all $t\in[0,1)$, and define $\psi:A \to C_0((0,1]) \otimes B$ by
\[ \psi(a) = \mathrm{id}_{(0,1]} \otimes \phi(a). \]
Then $\psi$ is a c.p.c.\ order zero map with the property that, for any nonzero $f \in C_0((0,1])_+$, $f(\psi)$ is injective.
\end{lemma}

\begin{proof}
It is clear that $\psi$ is c.p.c.\ order zero.
Let $\pi_\phi:C_0((0,1]) \otimes A \to B$ and $\pi_\psi:C_0(0,1]\otimes A\rightarrow C_0(0,1]\otimes B$ be the ${}^*$-homomorphisms corresponding to $\phi$ and $\psi$ as in Proposition \ref{prop:Ord0Structure}; these satisfy $\pi_\phi(\mathrm{id}_{(0,1]}\otimes a)=\phi(a)$ and $\pi_\psi(\mathrm{id}_{(0,1]}\otimes a)=\psi(a)=\mathrm{id}_{(0,1]}\otimes\phi(a)$  for all $a\in A$. Then for any $n\in\N$ and $a\in A$,
\begin{equation}
\pi_\psi(\mathrm{id}_{(0,1]}^n\otimes a)=\mathrm{id}_{(0,1]}^n\otimes \phi^n(a)=(\mathrm{id}_{C_0((0,1])}\otimes\pi_\phi)(\mathrm{id}_{(0,1]}^n\otimes\mathrm{id}_{(0,1]}^n\otimes a).
\end{equation}
Therefore, by the Stone-Weierstrass theorem, for $f\in C_0((0,1])_+$, 
\begin{equation}\label{eq.9.16}
f(\psi)(a)=(\mathrm{id}_{C_0((0,1])}\otimes \pi_\phi)(f(\mathrm{id}_{(0,1]}\otimes\mathrm{id}_{(0,1]})\otimes a),\quad a\in A.
\end{equation}

For nonzero $f\in C_0((0,1])_+$, write $g:=f(\mathrm{id}_{(0,1]}\otimes\mathrm{id}_{(0,1]})\in C_0((0,1])\otimes C_0((0,1])$.  
Thus (\ref{eq.9.16}) becomes
\begin{equation} f(\psi)(a) = (\mathrm{id}_{C_0((0,1])} \otimes \pi_\phi)(g \otimes a),\quad a\in A, \end{equation}
(noting that the tensor factors are broken down differently in the two terms on the right hand side).

Under the canonical identification $C_0((0,1]^2) \cong C_0((0,1]) \otimes C_0((0,1])$, $g$ is given by $g(s,t)=f(st)$.
Since $f$ is nonzero, $g^{-1}((0,\infty))$ contains a set of the form $U\times (t,1]$ for some $t\in(0,1]$ and some nonempty open set $U \subset (0,1]$.
Hence, if $g_0 \in C_0(U)_+$, then $g_0\otimes (\mathrm{id}_{(0,1]}-t)_+$ is in the ideal of $C_0((0,1])\otimes C_0((0,1])$ generated by $g$.  In particular, for $a\in A_+$,
\begin{equation}\label{e.9.18}
(\mathrm{id}_{C_0((0,1])} \otimes \pi_\phi)(g_0 \otimes (\mathrm{id}_{(0,1]}-t)_+ \otimes a) = g_0 \otimes (\phi-t)_+(a)\end{equation}
lies in the ideal generated by $f(\psi)(a)$. Taking $g_0$ and $a$ to be nonzero, the operator in (\ref{e.9.18}) is nonzero since $(\phi-t)_+$ is injective; therefore $f(\psi)(a)$ is nonzero.
\end{proof}

\begin{cor}\label{DimKirchberg}
Let $B$ be a Kirchberg algebra.
Then $\mathrm{dim}_\mathrm{nuc}(B) = 1$.
\end{cor}

\begin{proof}
Since $B$ cannot be AF, $\mathrm{dim}_\mathrm{nuc}(B) \geq 1$ by \cite[Remark 2.2(iii)]{WZ:Adv}.

When $B$ is unital, $\mathrm{dim}_\mathrm{nuc}(B) \leq 1$ is obtained from Theorem \ref{thm:TotFullCpcKirClass} as in the proof of Theorem \ref{thm:FiniteDim}, using Lemma \ref{lem:GoodKirchbergMaps} in place of Lemma \ref{lem:GoodTraceMaps} as follows.  Let $(\phi_n)_{n=1}^\infty$ be the sequence of c.p.c.\ order zero maps $\phi_n:B\rightarrow B$ 
 which factorize as $\eta_n\circ\theta_n$ through finite dimensional algebras $F_n$ as in (\ref{lem:GoodKirchbergMaps.1}) given by Lemma \ref{lem:GoodKirchbergMaps}. Let $\Phi:B\rightarrow B_\omega$ be the induced map, which is c.p.c.\ and order zero, and let $\iota:B\rightarrow B_\omega$ be the canonical inclusion.  By Kirchberg's absorption theorem (\cite{Kir:ICM,KP:Crelle}),  $B\cong B\otimes\mathcal O_\infty$.  Fix a positive contraction $h\in \mathcal O_\infty$ of full spectrum, and consider the c.p.c.\ order zero maps $\phi^{(1)}:=\Phi\otimes h:B\rightarrow (B\otimes \mathcal O_\infty)_\omega$ and $\phi^{(2)}:=\iota\otimes h:B\rightarrow (B\otimes \mathcal O_\infty)_\omega$.

By Lemma \ref{lem:TotFullCpc} (which applies as Lemma \ref{lem:GoodKirchbergMaps} ensures that $(\Phi-t)_+$ is injective for each $t\in [0,1)$), $f(\phi^{(1)})$ and $f(\phi^{(2)})$ are injective for any nonzero $f \in C_0((0,1])_+$.
Therefore Theorem \ref{thm:TotFullCpcKirClass} provides a unitary $w^{(0)}\in (B\otimes \mathcal O_\infty)_\omega$ satisfying 
 \begin{equation}
 x\otimes h=w^{(0)}(\Phi(x)\otimes h)w^{(0)}{}^*,\quad x\in B.
 \end{equation}
 Working with $1_{\mathcal O_\infty}-h$ in place of $h$ we obtain another unitary $w^{(1)}\in (B\otimes \mathcal O_\infty)_\omega$ satisfying 
 \begin{equation}
 x\otimes (1_{\mathcal O_\infty}-h)=w^{(1)}(\Phi(x)\otimes (1_{\mathcal O_\infty}-h))w^{(1)}{}^*,\quad x\in B.
 \end{equation}
 Choosing representing sequences $(w_n^{(0)})_{n=1}^\infty$ and $(w_n^{(1)})_{n=1}^\infty$ of unitaries in $(B\otimes\mathcal O_\infty)$ for $w^{(0)}$ and $w^{(1)}$, we obtain c.p.c.\ maps $\theta_n\oplus\theta_n:B\rightarrow F_n\oplus F_n$ and $\tilde{\eta}_n:F_n\oplus F_n\rightarrow (B\otimes\mathcal O_\infty)$, defined for $(y_0,y_1)\in F_n\oplus F_n$ by
\begin{equation}
\tilde{\eta}_n(y_0,y_1)=w^{(0)}_n(\eta_n(y_0)\otimes h)w^{(0)}_n{}^*+w^{(1)}_n(\eta_n(y_1)\otimes (1_{\mathcal O_\infty}-h))w^{(1)}_n{}^*.
\end{equation}
Then, $x\otimes 1_{\mathcal O_\infty}$ is the limit as $n\rightarrow\omega$ of $(\tilde{\eta}_n\circ(\theta_n\oplus\theta_n)(x))_{n=1}^\infty$.  Since each $\tilde{\eta}_n$ is a sum of two c.p.c.\ order zero maps, the nuclear dimension of the first factor embedding $B\rightarrow B\otimes\mathcal O_\infty$ is $1$.  Since $\mathcal O_\infty$ is strongly self-absorbing, $\dimnuc(B)=1$ (see \cite[Proposition 2.6]{TW:APDE}).

General Kirchberg algebras are either unital or stable (\cite{Z:PJM}) and in the latter case, $B\cong B_0\otimes\mathcal K$ for some unital Kirchberg algebra $B_0$.
Hence, by \cite[Corollary 2.8(i)]{WZ:Adv}, $\dim_{\mathrm{nuc}}(B)=\dim_{\mathrm{nuc}}(B_0)=1$.
\end{proof}

Combining the previous result with the stably finite case from Section \ref{sec:FiniteDim} computes the nuclear dimension for simple, separable, unital, nuclear, $\Z$-stable $\mathrm C^*$-algebras  with compact (possibly empty) tracial boundary.
\begin{cor}
Let $A$ be a non-AF, simple, separable, unital, nuclear, $\Z$-stable $\mathrm C^*$-algebra such that $\partial_eT(A)$ is compact.  Then $\dim_{\mathrm{nuc}}(A)=1$.
\end{cor}
\begin{proof}
By Kirchberg's dichotomy theorem (see \cite[Theorem 4.1.10]{R:Book}), $A$ is either finite or purely infinite (in the case that $\partial_eT(A)=\emptyset$).  The result follows from Theorem \ref{thm:FiniteDim} in the first case and Corollary \ref{DimKirchberg} in the latter.
\end{proof}

Two-coloured classification results also follow from Theorem \ref{thm:TotFullCpcKirClass} as in the stably finite case in Section \ref{sec:2colour}.
The result applies to any pair of injective $^*$-homomorphisms -- there are no trace restrictions since $B_\omega$ has no traces.
More generally, we allow for the $\phi_i$ to be order zero maps rather than $^*$-homomorphisms, though for this we need to ask that $(\phi_i-t)_+$ is injective for $0 \leq t < 1$ (whereas in Theorem \ref{thm7.8}, the trace condition automatically ensured that $(\phi_2-t)_+$ is full).

\begin{cor}
Let $(B_n)_{n=1}^\infty$ be a sequence of Kirchberg algebras, write $B_\omega:=\prod_\omega B_n$ and let $A$ be a separable, unital, nuclear $\mathrm C^*$-algebra.
Let $\phi_1,\phi_2:A \to B_\omega$ be a pair of c.p.c.\ order zero maps such that $(\phi_i-t)_+(a)$ is nonzero for all $0\leq t<1$, $0 \neq a\in A$, and $i=1,2$.
Then there exist contractions $w^{(0)},w^{(1)}\in B_\omega$ such that
\begin{gather}
\notag
\phi_1(a) = 
w^{(0)}\phi_2(a)w^{(0)}{}^*+w^{(1)}\phi_2(a)w^{(1)}{}^*, \quad a\in A, \text{ and} \\
w^{(0)}{}^*w^{(0)}+w^{(1)}{}^*w^{(1)} = 1_{B_\omega},\label{e.9.22}
\end{gather}
and in addition, $w^{(i)*}w^{(i)}$ commutes with $\phi_2(A)$ (so that $w^{(i)}\phi_2(\cdot)w^{(i)}{}^*$ is a c.p.c.\ order zero map) for $i=0,1$.

If moreover both $\phi_1,\phi_2$ are $^*$-homomorphisms, then there exist $\tilde w^{(0)},\tilde w^{(1)}$ such that
\begin{gather}
\notag
\phi_1(a) = 
\tilde{w}^{(0)}\phi_2(a)\tilde{w}^{(0)}{}^*+\tilde{w}^{(1)}\phi_2(a)\tilde{w}^{(1)}{}^* \quad a\in A, \\
\notag
\phi_2(a) = 
\tilde{w}^{(0)}{}^*\phi_1(a)\tilde{w}^{(0)}+\tilde{w}^{(1)}{}^*\phi_1(a)\tilde{w}^{(1)} \quad a\in A, \\
\phi_1(1_A)=\tilde{w}^{(0)}\tilde{w}^{(0)}{}^*+\tilde{w}^{(1)}\tilde{w}^{(1)}{}^*, \quad \phi_2(1_A)=\tilde{w}^{(0)}{}^*\tilde{w}^{(0)}+\tilde{w}^{(1)}{}^*\tilde{w}^{(1)},
\end{gather}
and such that $\tilde{w}^{(i)}{}^*\tilde{w}^{(i)}$ commutes with $\phi_2(A)$ and $\tilde{w}^{(i)}\tilde{w}^{(i)}{}^*$ commutes with $\phi_1(A)$.  
In the case that $\phi_1,\phi_2$ are unital $^*$-homomorphisms, this says that $\phi_1,\phi_2$ are $2$-coloured equivalent in the sense of Definition \ref{def:ColourEquiv}, and in this case, $\tilde w^{(i)}$ can be chosen to be normal.
\end{cor}

\begin{proof}
This proof follows along the same lines as the proof of Theorem \ref{thm:2ColourUniqueness}.
Via Kirchberg's absorption theorem (\cite{Kir:ICM,KP:Crelle}), $B_n\cong B_n\otimes\mathcal O_\infty\cong B_n\otimes\Z$ for each $n\in\N$. Then, by Lemma \ref{lem:Zfacts}(\ref{lem:Zfacts2}), we may assume without loss of generality that $B_n=C_n \otimes \mathcal Z$ (for a copy $C_n$ of $B_n$) such that $\phi_i=\check\phi_i \otimes 1_{\mathcal Z}$ where $(\check\phi_i-t)_+$ is injective for all $t\in[0,1)$, for $i=1,2$ (and $\check{\phi}_i$ is a (unital) $^*$-homomorphism when $\phi_i$ is). Let $h \in \mathcal Z_+$ have spectrum $[0,1]$, so that by Lemma \ref{lem:TotFullCpc}, for every nonzero $f \in C_0((0,1])_+$, the maps $f(\check\phi_i \otimes h)$ and $f(\check\phi_i \otimes (1_\Z-h))$ are injective, for $i=1,2$.
Hence by Theorem \ref{thm:TotFullCpcKirClass}, there exist unitaries $u_h,u_{1-h} \in B_\omega = \prod_\omega (C_n \otimes \Z)$ satisfying
\begin{align}
\notag
\check{\phi}_1(a)\otimes h &=u_h(\check{\phi}_2(a) \otimes h)u_h^*, \\
\check{\phi}_1(a)\otimes (1_\Z-h)&=u_{1-h}(\check{\phi}_2(a)\otimes (1_\Z-h))u_{1-h}^*,\quad a\in A.
\end{align}
This is exactly the same as \eqref{eq:2colour.2}. To obtain (\ref{e.9.22}), follow the proof of Theorem \ref{thm:2ColourUniqueness} from \eqref{eq:2colour.2} to where condition (\ref{2ColU.3}) is established.
When both $\phi_1$ and $\phi_2$ are $^*$-homomorphisms, one obtains the symmetric statements in the second half of the corollary by following the proof of Theorem \ref{thm:2ColourUniqueness} from (\ref{e.6.13}) to the end of the proof.
\end{proof}

\section*{Addendum}

Since the first version of this paper AT, SW and WW have shown in \cite{TWW:arXiv} that faithful traces on separable nuclear $\mathrm{C}^*$-algebras in the UCT class are quasidiagonal.  Consequently, in the UCT case, Question \ref{AllTracesQD?} has a positive answer and Theorem \ref{thm:dn1} gives rise to decomposition rank estimates.  A full discussion can be found in \cite[Section 6]{TWW:arXiv}.

\end{document}